\documentclass[a4paper,reqno]{article}
\usepackage{amsmath,amsfonts,amssymb,amsthm,amsrefs,bbm,graphicx,enumerate,geometry}
\usepackage{mathrsfs}

\usepackage{subfigure} %插入多图时用子图显示的宏包

\usepackage{color}

\usepackage{framed} 
\usepackage{hyperref}  
\hypersetup{
    linktoc=page,
    linkcolor=red,          % color of internal links
    citecolor=blue,        % color of links to bibliography
    filecolor=blue,      % color of file links
    urlcolor=cyan,
   colorlinks= true          % color of external links
}

\numberwithin{figure}{section}
\numberwithin{equation}{section}

\newtheorem{thm}{Theorem}[section]

\newtheorem{lemma}[thm]{Lemma}

\newtheorem{cor}[thm]{Corollary}
\newtheorem{defn}[thm]{Definition}
\newtheorem{example}[thm]{Example}
\newtheorem{remark}[thm]{Remark}

\newtheorem{conjecture}[thm]{Conjecture}

\renewcommand{\epsilon}{\varepsilon}

\def\Im{{\rm Im}\,}

\def\<#1{\langle #1\rangle}
\begin{document}{\allowdisplaybreaks[4]}

%%%%%%%%%%%%%%%%%%%%%%%%%%%%%
%%%%%%%%%%%%%%%%%%%%%%%%%%%%%

\title{Multiple radial SLE($\kappa$) and quantum Calogero-Sutherland system}
\author{
    Jiaxin Zhang\footnotemark[1]
   }
\renewcommand{\thefootnote}{\fnsymbol{footnote}}
\footnotetext[1]{{\bf zhangjx.prob@gmail.com} Department of Mathematics, California Institute of Technology}

\maketitle

\begin{abstract}

We develop a theory for the multiple radial $\mathrm{SLE}(\kappa)$ systems with parameter $\kappa > 0$ -- a family of random multi-curve systems in a simply-connected domain $\Omega$, with marked boundary points $z_1, \ldots, z_n \in \partial \Omega$ and a marked interior point $q$.

As a consequence of the domain Markov property and conformal invariance, we show that such systems are characterized by equivalence classes of partition functions, which are not necessarily conformally covariant. Nevertheless, within each equivalence class, one can always choose a conformally covariant representative.

When $\Omega$ is taken to be the unit disk $\mathbb{D}$ and the marked interior point $q$ is set at the origin, we demonstrate that the partition function satisfies a system of second-order PDEs, known as the null vector equations, with a null vector constant $h$ and a rotation equation involving a constant $\omega$.

Motivated by the Coulomb gas formalism in conformal field theory, we construct four families of solutions to the null vector equations, which are naturally classified according to topological link patterns.

For $\kappa > 0$, the partition functions of multiple radial SLE($\kappa$) systems correspond to eigenstates of the quantum Calogero-Sutherland (CS) Hamiltonian beyond the states built upon the fermionic states.

\par \
\end{abstract}

\newpage
\tableofcontents

\newpage
\section{Introduction}
\subsection{Background}
\
\indent 
The Schramm-Loewner evolution SLE($\kappa$) with $\kappa>0$ is a one-parameter family of random conformally invariant curves in the plane describing interfaces within conformally invariant systems arising from statistical physics, as introduced in \cites{Sch00, LSW04, Smi06, Sch07, SS09}. 
Conformal field theory (CFT), a quantum field theory invariant under conformal transformations, is also widely used to study critical phenomena, see \cites{Car96,FK04}. SLE and the multiple SLE systems can be coupled to conformal field theories (CFT) through the SLE-CFT correspondence, which serves as a powerful tool for predicting phenomena and computing important quantities of SLE($\kappa$) and multiple SLE($\kappa$) systems from the CFT perspective, as demonstrated in references like \cites{BB03a, Car03, FW03, FK04, Dub15a, Pel19}. The parameter $\kappa$ measures the roughness of these fractal curves and determines the central charge $c(\kappa)=(3 \kappa-8)(6-\kappa) / 2 \kappa$ of the associated CFT.

In recent years, there has been tremendous interest in multiple SLE systems. 
These systems describe families of non-intersecting SLE curves with prescribed pairwise connections among boundary and interior points. 
In particular, multiple chordal SLE---the case with $2n$ marked boundary points and no interior points---has been thoroughly studied.

\begin{itemize}
  \item \textbf{Probabilistic constructions and classification.} 
  Works such as \cites{Dub06, KL07, Law09b, PW20} established partition functions, commutation relations, and the general framework for multiple chordal SLE($\kappa$) systems, thereby providing a rigorous probabilistic basis for the theory.  

  \item \textbf{Connections to CFT.} 
  In parallel, \cites{FK15a, Pel19, Pel20} investigated the correspondence with conformal field theory, interpreting partition functions from the CFT perspective and highlighting their role as conformal blocks.  

  \item \textbf{Deterministic limit.} 
  On the one hand, \cites{PW20} derived large deviation principles for multiple chordal SLE($\kappa$) curves from a probabilistic viewpoint. 
  On the other hand, \cites{ABKM20} identified integrals of motion for multiple chordal SLE($0$) curves via the SLE--CFT correspondence. 
  Together, these complementary approaches give a complete description of the classical limit.
\end{itemize}

Multiple radial SLE is a family of random multi-curve systems in a simply connected domain $\Omega$, with marked boundary points $z_1, \ldots, z_n \in \partial \Omega$ and a marked interior point $q$. 
In contrast to the chordal case, the theory of multiple radial SLE systems has been comparatively less developed.

\begin{itemize}
  \item \textbf{Mathematical progress.} Recent contributions such as \cites{HL21, WW24} initiated the study of multiple radial partition functions and commutation relations in special cases.

  \item \textbf{Physics perspectives.} Parallel discussions in the physics literature \cites{Car04, DC07, SKFZ11, FKZ12} studied the multiple radial SLE systems from the conformal field theory perspective but without full mathematical justification.  
\end{itemize}

Building on the above literature, the present paper advances the study of multiple radial SLE($\kappa$) systems. We investigate the structure of multiple radial SLE($\kappa$) from four perspectives:
\begin{itemize}
    \item Commutation relations and the existence of conformally covariant partition functions.
    \item Deviation of the null vector equation and the rotation equation for the partition function.
    \item Solution space for the null vector equations and the rotation equation.
    \item Relations to the quantum Calogero–Sutherland system.
\end{itemize}

The core principle throughout our study of the multiple radial SLE system is the SLE-CFT correspondence. SLE and multiple SLE systems can be coupled to a conformal field in two key aspects:
\begin{itemize} 
\item The level-two degeneracy equations for the conformal fields coincide with the null vector equations for the SLE partition functions. 
\item The correlation functions of the conformal fields serve as martingale observables for the SLE processes. \end{itemize}

\subsection{Multiple radial SLE($\kappa$) systems with $\kappa>0$}

\begin{figure}[ht]
\begin{minipage}[t]{0.48\linewidth}
    \centering
    \includegraphics[width=6cm]{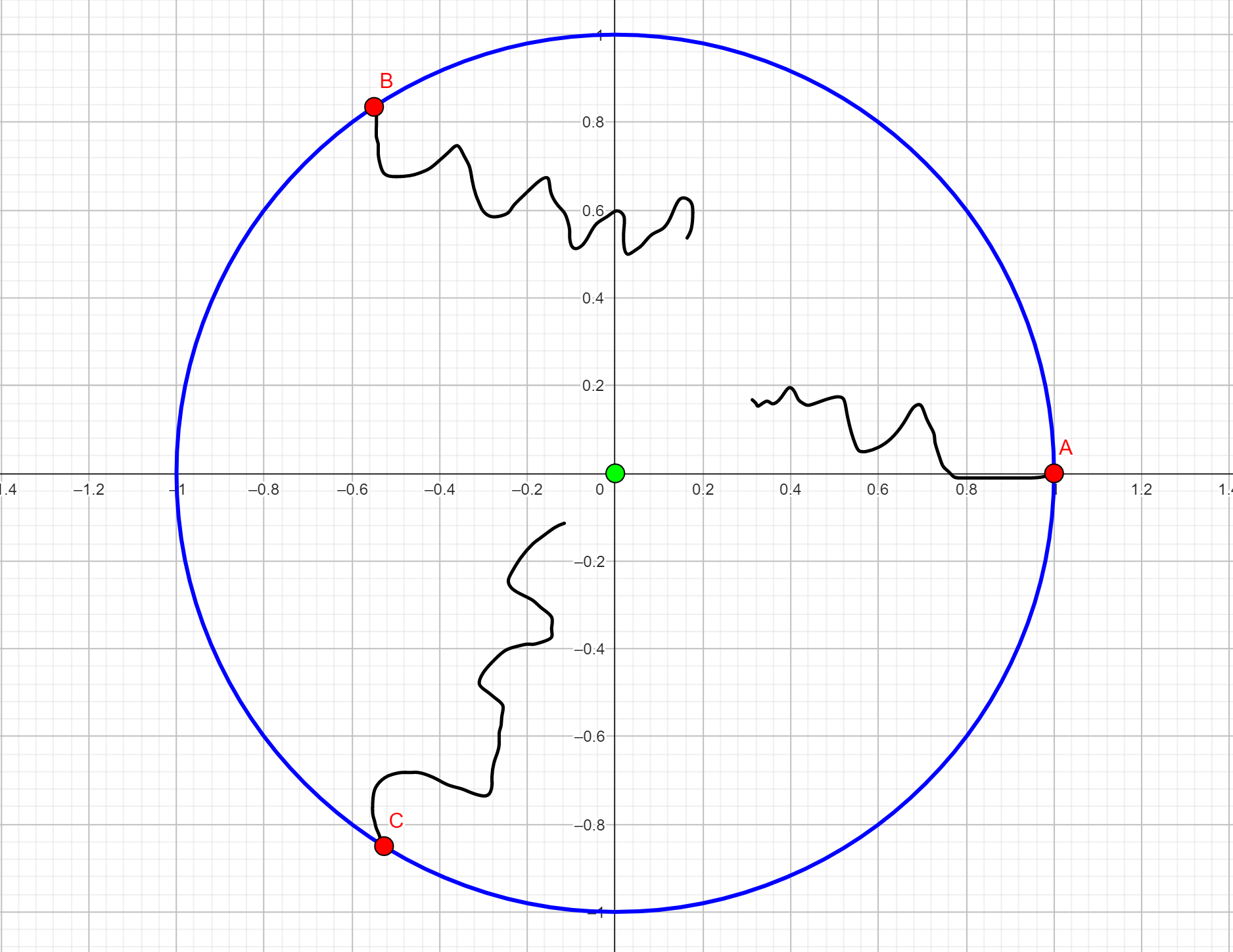}
    \caption{Multiple radial SLE($\kappa$) systems in $\mathbb{D}$}
    
\end{minipage}
\begin{minipage}[t]{0.48\linewidth}
    \centering
	\includegraphics[width=6cm]{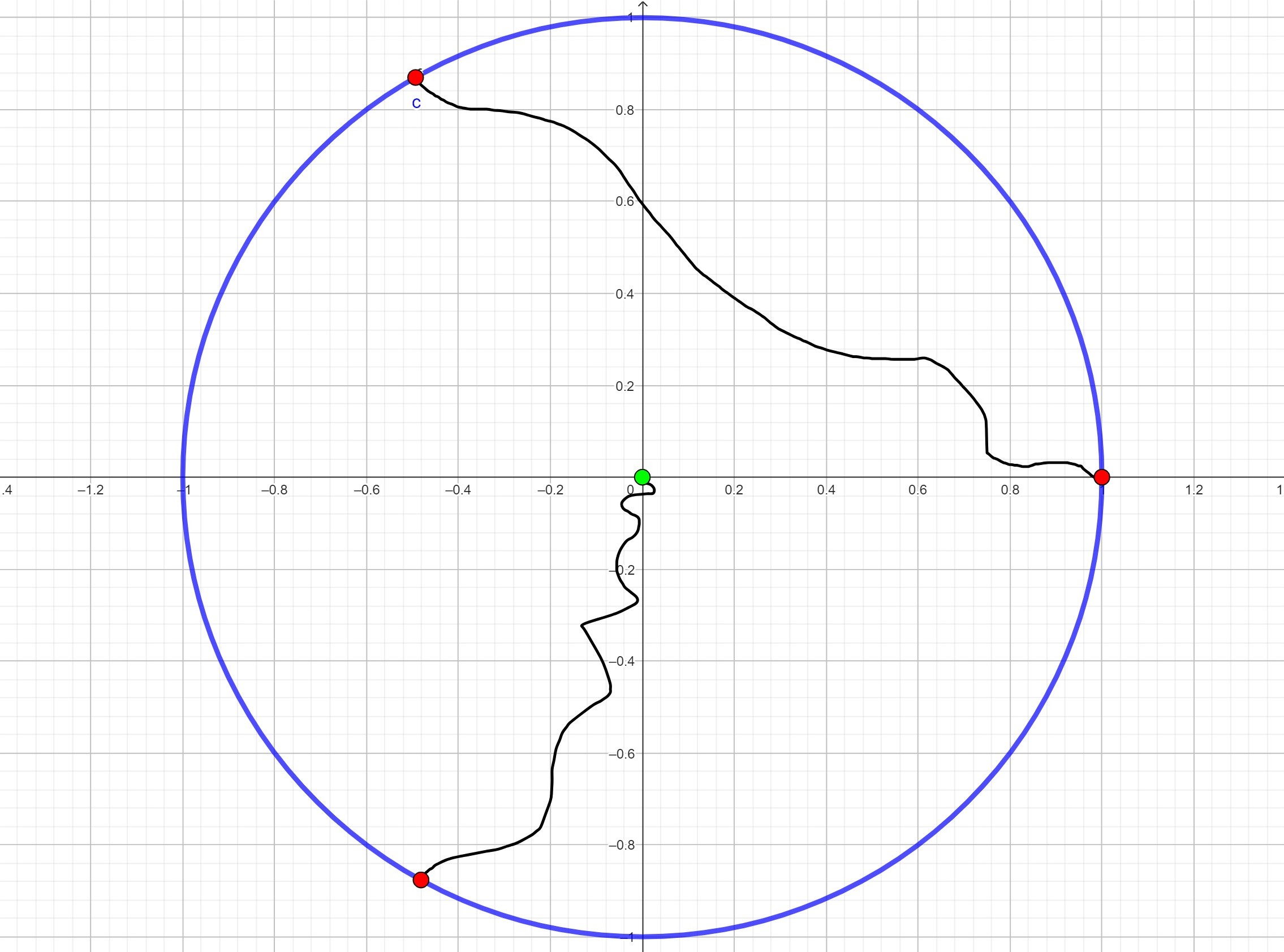}
    \caption{Multiple radial SLE($\kappa$) in $\mathbb{D}$}
    
\end{minipage}
\end{figure}

In a simply connected domain $\Omega$ with boundary points $z_1, z_2, \ldots, z_n$ and a marked interior point $q$, we define a \emph{local multiple radial SLE($\kappa$) system} as a compatible family of probability measures  
\[
\mathbb{P}_{\left(\Omega; z_1, z_2, \ldots, z_n, q\right)}^{\left(U_1, U_2, \ldots, U_n\right)}
\]  
on $n$-tuples of continuous, non-self-crossing curves starting from $z_i$ within a localization neighborhood $U_i$, none of which contains $q$. 
Similar to the chordal case, multiple radial SLE($\kappa$) systems satisfy conformal invariance, the domain Markov property, and absolute continuity to standard SLE($\kappa$) measure in the localization neighborhood $U_i$ (see section 2.1 in \cite{WW24}).

A more precise characterization of these measures is provided in Definitions \ref{localization of measure} and \ref{local_multiple_radial_SLE_kappa}.

\begin{defn}[Localization of Measures]\label{localization of measure}
Let $\Omega \subsetneq \mathbb{C}$ be a simply connected domain with an interior marked point $u \in \Omega$. Let $z_1, z_2, \ldots, z_n$ denote distinct prime ends of $\partial \Omega$, and let $U_1, U_2, \ldots, U_n$ be closed neighborhoods of $z_1, z_2, \ldots, z_n$ in $\Omega$ such that:
\begin{itemize}
    \item $U_i \cap U_j = \emptyset$ for all $1 \leq i < j \leq n$,
    \item None of the $U_j$ contain the interior point $q$.
\end{itemize}
We consider the measures 
\[
\mathbb{P}_{\left(\Omega; z_1, z_2, \ldots, z_n, q\right)}^{\left(U_1, U_2, \ldots, U_n\right)}
\]
defined on $n$-tuples of unparametrized continuous curves in $\Omega$. Each curve $\eta^{(j)}$ begins at $z_j$ and exits $U_j$ almost surely.

A family of such measures indexed by different choices of $\left(U_1, U_2, \ldots, U_n\right)$ is called \textbf{compatible} if for all $U_j \subset U_j'$, the measure 
\[
\mathbb{P}_{\left(\Omega; z_1, z_2, \ldots, z_n, q\right)}^{\left(U_1, U_2, \ldots, U_n\right)}
\]
is obtained by restricting the curves under 
\[
\mathbb{P}_{\left(\Omega; z_1, z_2, \ldots, z_n, q\right)}^{\left(U_1', U_2', \ldots, U_n'\right)}
\]
to the portions of the curves that remain inside the subdomains $U_j$ before their first exit.

\begin{figure}[ht]
    \centering
    \includegraphics[width=8cm]{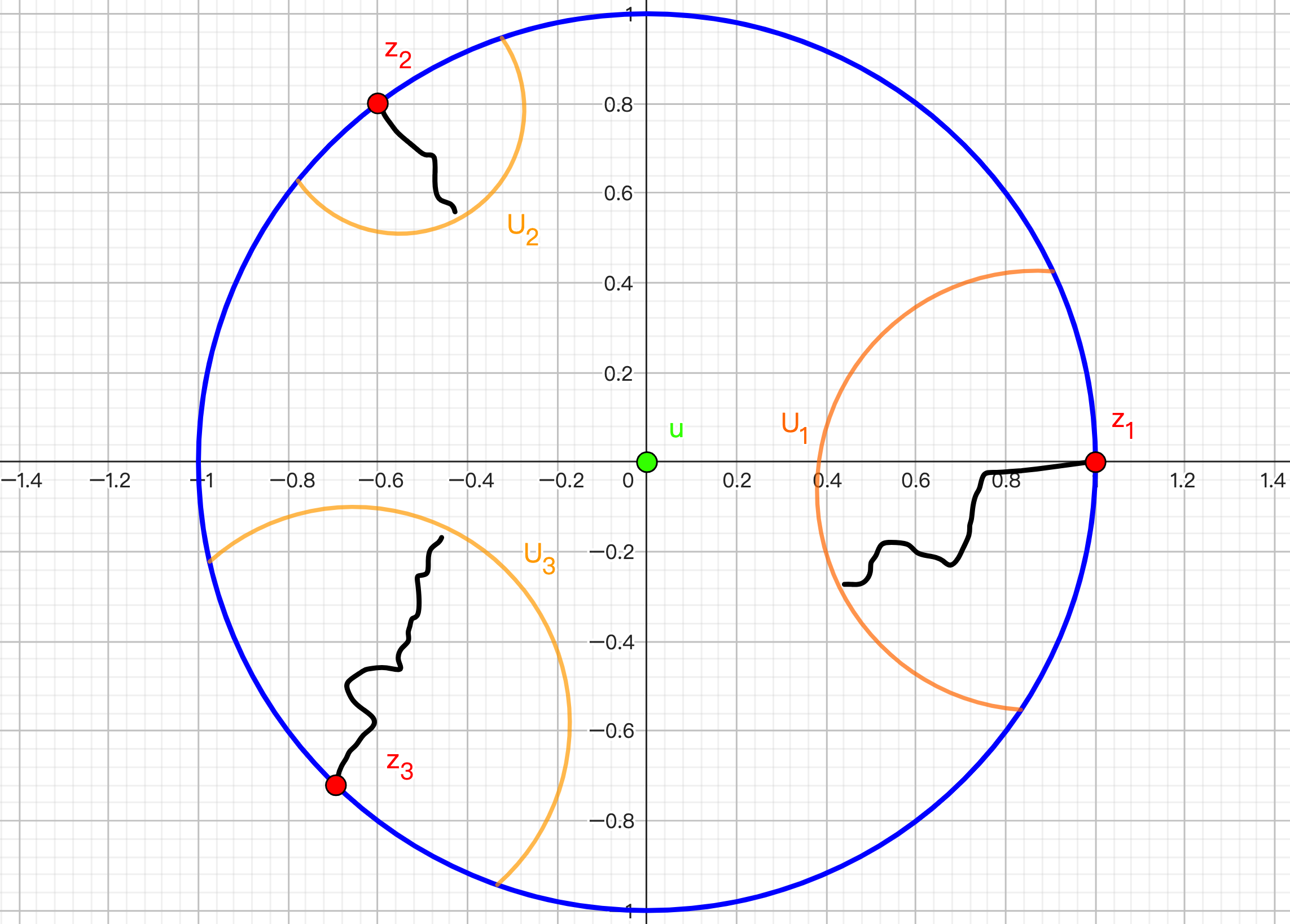}
    \caption{Localization of multiple radial SLE($\kappa$)} 

\end{figure}
\end{defn}

\begin{defn}[Local multiple radial SLE($\kappa$)]
\label{local_multiple_radial_SLE_kappa}
The locally commuting $n$-radial $\mathrm{SLE}(\kappa)$ is a compatible family of measures 
\[
\mathbb{P}_{\left(\Omega; z_1, z_2, \ldots, z_n, q\right)}^{\left(U_1, U_2, \ldots, U_n\right)}
\]
on $n$-tuples of continuous, non-self-crossing curves $\left(\gamma^{(1)}, \gamma^{(2)}, \ldots, \gamma^{(n)}\right)$ for all simply connected domains $\Omega$ with marked points $\left(z_1, z_2, \ldots, z_n, q\right)$ and target sets $\left(U_1, U_2, \ldots, U_n\right)$. These measures satisfy the following conditions:

\begin{itemize}
    \item[(i)] \textbf{Conformal invariance}: If $\varphi: \Omega \to \Omega^\prime$ is a conformal map, then the pullback measure satisfies
    \[
    \varphi^* \mathbb{P}_{\left(\Omega^\prime; \varphi(z_1), \varphi(z_2), \ldots, \varphi(z_n), \varphi(q)\right)}^{\left(\varphi(U_1), \varphi(U_2), \ldots, \varphi(U_n)\right)} 
    = \mathbb{P}_{\left(\Omega; z_1, z_2, \ldots, z_n, u\right)}^{\left(U_1, U_2, \ldots, U_n\right)}.
    \]
    It suffices to describe the measure when $\left(\Omega; z_1, z_2, \ldots, z_n, q\right) = \left(\mathbb{D}; z_1, z_2, \ldots, z_n, 0\right)$. The definition for arbitrary $\Omega$ with a marked interior point $q$ can then be extended by pulling back via a conformal equivalence $\varphi: \Omega \to \mathbb{D}$ mapping $q$ to $0$.

    \item[(ii)] \textbf{Domain Markov property}: Let $\left(\gamma^{(1)}, \gamma^{(2)}, \ldots, \gamma^{(n)}\right) \sim \mathbb{P}_{\left(\mathbb{D}; z_1, z_2, \ldots, z_n, q\right)}^{\left(U_1, U_2, \ldots, U_n\right)}$, and parametrize $\gamma^{(j)}$ by their own capacity in $\mathbb{D}$. For stopping times $\mathbf{t} = (t_1, t_2, \ldots, t_n)$, define
    \[
    \tilde{U}_j = U_j \setminus \gamma^{(j)}_{[0, t_j]}, \quad 
    \tilde{\gamma}^{(j)} = \gamma^{(j)} \setminus \gamma^{(j)}_{[0, t_j]}, \quad 
    \tilde{\Omega} = \mathbb{D} \setminus \bigcup_{j=1}^n \gamma^{(j)}_{[0, t_j]}.
    \]
    Then, conditionally on the initial segments $\bigcup_{j=1}^n \gamma^{(j)}_{[0, t_j]}$, we have
    \[
    \left(\tilde{\gamma}^{(1)}, \tilde{\gamma}^{(2)}, \ldots, \tilde{\gamma}^{(n)}\right) 
    \sim \mathbb{P}_{\left(\tilde{\Omega}; \gamma^{(1)}_{t_1}, \gamma^{(2)}_{t_2}, \ldots, \gamma^{(n)}_{t_n}, q\right)}^{\left(\tilde{U}_1, \tilde{U}_2, \ldots, \tilde{U}_n\right)}.
    \]

    \item[(iii)] \textbf{Absolute Continuity with respect to independent SLE($\kappa$)}: Let $\left(\gamma^{(1)}, \gamma^{(2)}, \ldots, \gamma^{(n)}\right) \sim \mathbb{P}_{\left(\mathbb{D}; z_1, z_2, \ldots, z_n, 0\right)}^{\left(U_1, U_2, \ldots, U_n\right)}$. Let $z_j(t) = e^{i \theta_j(t)}$, the capacity-parametrized Loewner driving function $t \mapsto \theta_j(t)$ for $\gamma^{(j)}$ satisfies
    \[ \label{sle with drift}
    \begin{aligned}
    \mathrm{d} \theta_j(t) &= \sqrt{\kappa} \, \mathrm{d} B_j(t) + b_j\left(\boldsymbol{\theta}(t)\right) \, \mathrm{d} t, \\
    \mathrm{d} \theta_k(t) &= \cot\left(\frac{\theta_k(t) - \theta_j(t)}{2}\right) \, \mathrm{d} t, \quad k \neq j,
    \end{aligned}
    \]
    where $B_j(t)$ are independent standard Brownian motions, and $b_j(\boldsymbol{\theta})$ are $C^2$ functions on the chamber 
    \[
    \mathfrak{X}^n = \left\{(\theta_1, \theta_2, \ldots, \theta_n) \in \mathbb{R}^n \mid \theta_1 < \theta_2 < \cdots < \theta_n < \theta_1 + 2\pi \right\}.
    \]
\end{itemize}
\end{defn}

The domain Markov property implies that one can sequentially map out the curves $\gamma^{(i)}_{[0, t_i]}$ using $g_{t_i}^{(i)}$, or perform the mappings in reverse order. The resulting image has the same distribution regardless of the order. This property is known as the commutation relation or reparametrization symmetry (see Section \ref{reparametrization symmetry}).

\subsection{Commutation relations and Coulomb gas solutions}
In the following, we study how commutation relations and conformal invariance impose constraints on the drift terms $b_j(\boldsymbol{\theta})$.

We study the multiple radial SLE($\kappa$) systems by exploring the following two aspects:
\begin{itemize}
    \item Commutation relations and conformal invariance
    \item Solution space of the null vector equations.
\end{itemize}

Extending the results in \cite{Dub07} on commutation relations (see also \cite{WW24} for the two radial case), we derive analogous commutation relations for multiple radial SLEs in the unit disk $\mathbb{D}$ with $z_1=e^{i\theta_1},z_2=e^{i\theta_2},\ldots,z_n=e^{i\theta_n} \in \partial{\mathbb{D}}$ and one additional marked point $q =0$, see section \ref{reparametrization symmetry}. The family of measure $\mathbb{P}_{(\theta_1,\ldots,\theta_n)}$ of a multiple radial SLE($\kappa$) system is encoded by a partition function $\psi(\boldsymbol{\theta}):\left\{\left(\theta_1, \theta_2,\ldots,\theta_n\right) \in \mathbb{R}^n \mid \theta_1<\theta_2<\ldots<\theta_n<\theta_1+2 \pi\right\} \rightarrow \mathbb{R}_{>0}$.

\begin{thm}
For a local multiple radial SLE(\(\kappa\)) system in the unit disk \( \mathbb{D} \) with boundary points $z_1=e^{i\theta_1},z_2=e^{i\theta_2},\ldots,z_n=e^{i\theta_n}$ and a marked point at \( q = 0 \), there exists a positive partition function \( \psi(\boldsymbol{\theta}) \) such that the drift term \( b_j \) in equation (\ref{sle with drift}) satisfies  
\begin{equation}
    b_j = \kappa \frac{\partial_j \psi}{\psi}, \quad j = 1,2,\dots,n.
\end{equation}  
Moreover, \( \psi(\boldsymbol{\theta}) \) satisfies the null vector equation  
\begin{equation} \label{null vector equation angular coordinate constant h}
    \frac{\kappa}{2} \partial_{ii} \psi + \sum_{j \neq i} \cot\left(\frac{\theta_j - \theta_i}{2}\right) \partial_i \psi 
    + \left(1 - \frac{6}{\kappa}\right) \sum_{j \neq i} \frac{1}{4 \sin^2\left(\frac{\theta_j - \theta_i}{2}\right)} \psi - h \psi = 0,
\end{equation}  
for some constant \( h \).  

Furthermore, there exists a real constant \( \omega \) such that for all \( \theta \in \mathbb{R} \),  
\begin{equation} \label{rotation_invariance}
    \psi(\theta_1 + \theta, \dots, \theta_n + \theta) = e^{-\omega \theta} \psi(\theta_1, \dots, \theta_n).
\end{equation}

Conversely, given a positive partition function \( \psi(\boldsymbol{\theta}) \) satisfying both the null vector equation (\ref{null vector equation angular coordinate constant h}) and the rotation invariance condition (\ref{rotation_invariance}), consider the multiple radial Loewner chain 

the multiple radial SLE($\kappa$) Loewner chain as a normalized conformal map $g_t = g_t(z)$, with the initial condition $g_0(z) = z$ and the evolution given by the Loewner equation
\begin{equation}
\partial_t g_t(z) = \sum_{j=1}^n \nu_j(t) g_t(z) \frac{z_j(t) + g_t(z)}{z_j(t) - g_t(z)}, \quad g_0(z) = z.
\end{equation}

The Loewner chain for the covering map $h_t(z) = -i \log(g_t(e^{iz}))$ is given by
\begin{equation}
\partial_t h_t(z) = \sum_{j=1}^n \nu_j(t) \cot\left( \frac{h_t(z) - \theta_j(t)}{2} \right), \quad h_0(z) = z.    
\end{equation}

driven by the functions \( \theta_j(t) \), for \( j = 1, \dots, n \), evolving as  
\begin{equation}\label{multiple_SLE_driving}
d{\theta}_j = \nu_j(t) \frac{\partial_j \log \psi(\boldsymbol{\theta})}{\partial \theta_j} dt 
+ \sum_{k \neq j} \nu_k(t) \cot\left( \frac{\theta_j - \theta_k}{2} \right) dt 
+ \sqrt{\kappa} dB_{t}^{j},
\end{equation}
where \( \boldsymbol{\nu} = (\nu_1, \dots, \nu_n) \) is a set of capacity parametrizations, with each \( \nu_i: [0, \infty) \to [0, \infty) \) assumed to be measurable. Here, $B^{j}(t)$ denotes a set of independent Brownian motion.

This process defines a local multiple radial SLE(\(\kappa\)) system.
\end{thm}

A significant difference between the multiple radial SLE($\kappa$) systems and standard multiple chordal SLE($\kappa$) systems arises when we study their conformal invariance properties. Although the multiple radial SLE($\kappa$) systems are conformally invariant, the partition functions in its corresponding equivalence classes do not necessarily exhibit conformal covariance when we have an extra marked point.

We define two partition functions as \emph{equivalent} if and only if they induce identical multiple chordal SLE($\kappa$) systems. Equivalent partition functions differ a by multiplicative function $f(u)$.
\begin{equation}
\tilde{\psi}=f(u)\cdot\psi
\end{equation}
where $f(u)$ is an arbitrary positive real smooth function depending on the marked interior point $u$
A simple example that violates conformal covariance is
when $f(u)$ is not conformally covariant. 
However, within each equivalence class, it is still possible to find at least one conformally covariant partition function.

Following \cite{FK15c} on solution space of the null vector equations for partition functions of multiple chordal SLE($\kappa$), we construct four types of solutions to the null vector equations and Ward's identities for partition functions of multiple radial SLE($\kappa$) via Coulomb gas integral method in conformal field theory. 

Choosing charges $\sigma_j$ and charges $\tau_{k}$ for all $j \in\{1,2, \ldots, n\}$ and $k \in \{1,2,\ldots,m \}$, the following trigonometric Coulomb gas integral plays an important role in the theory of multiple radial SLE:

\[
\oint \cdots \oint_{\Gamma} \prod_{1 \leq i<j \leq n}\left(\sin\frac{\theta_j-\theta_i}{2}\right)^{ \sigma_i \sigma_j} \prod_{1 \leq r<s \leq m}\left(\sin\frac{\zeta_s-\zeta_r}{2}\right)^{\tau_r \tau_s } \prod_{\substack{1 \leq i \leq n \\ 1 \leq r \leq m}}\left(\sin\frac{\zeta_r-\theta_i}{2}\right)^{\tau_r \sigma_i} \, \mathrm{d} \zeta_1 \cdots \mathrm{d} \zeta_m.
\]

The integration variables \( \zeta_1, \zeta_2, \ldots, \zeta_m \) are integrated along multiple contours \( \Gamma \) that correspond to various topological link patterns. See Section~\ref{Classification of screening solutions} for a detailed explanation.

\begin{figure}[h]
    \centering
    \includegraphics[width=8cm]{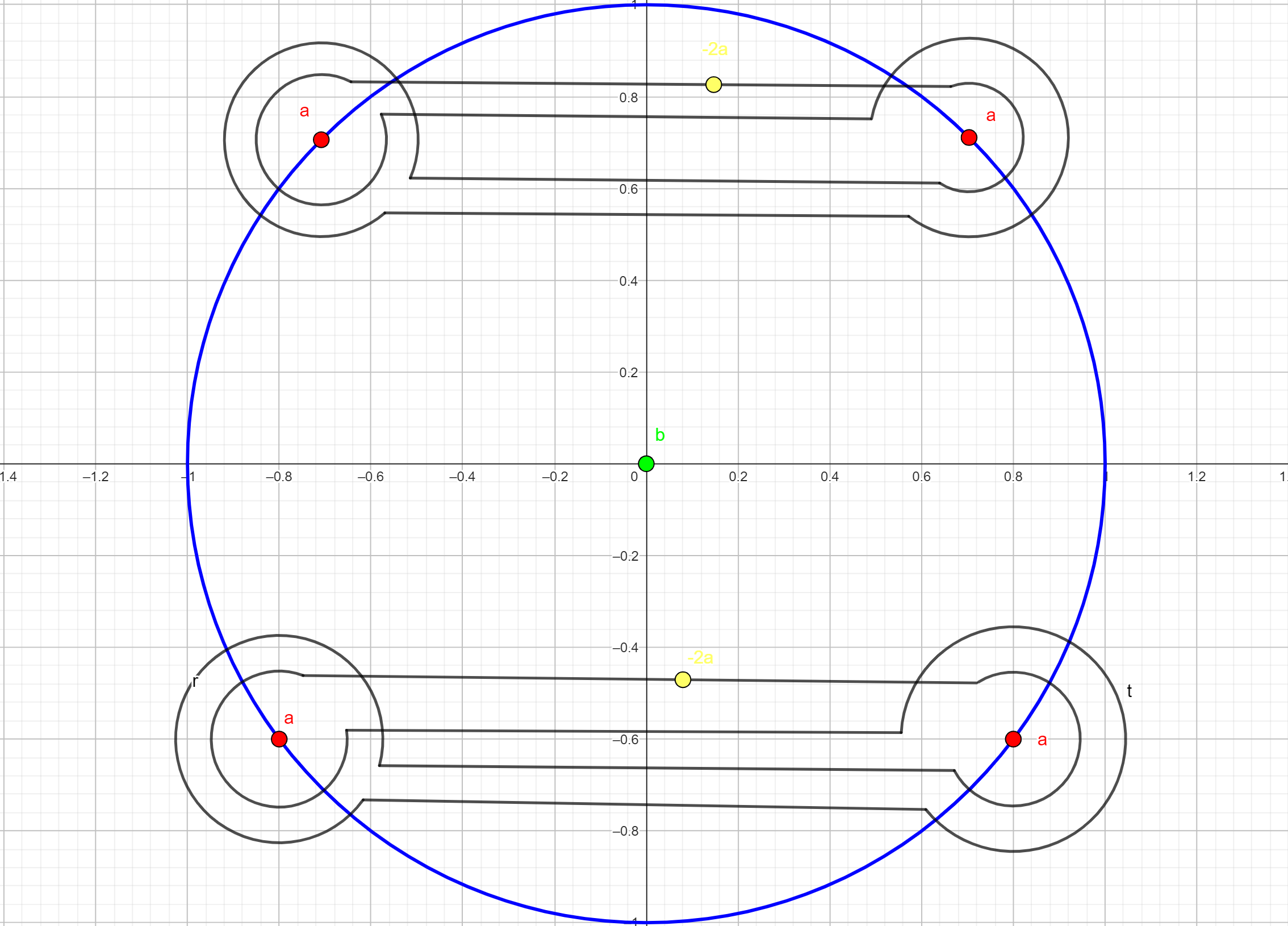}
    \caption{Integrate $\zeta_1,\zeta_2$ (yellow points) along two Pochhammer contour} 
\end{figure}

\begin{thm}
\label{solution space to null and ward} 
The following four types of Coulomb gas integrals (see definitions in Section~\ref{Classification of screening solutions}) solve the null vector equation \eqref{null vector equation angular coordinate constant h} and the rotation equation \eqref{rotation_invariance}:
\begin{itemize} 
    \item[(1)] For any link pattern \( \alpha \in LP(n,m) \), with \( m,n \in \mathbb{Z} \) and \( 1 \leq m \leq \frac{n}{2} \), the Coulomb gas integral \( \mathcal{J}_{\alpha}^{n,m}(\boldsymbol{\theta}) \) defined in (\ref{radial ground solution}) solves the null vector equation \eqref{null vector equation angular coordinate constant h} with 
    \[
    h = \frac{1 - (n - 2m)^2}{2\kappa},
    \]
    and the rotation equation \eqref{rotation_invariance} with \( \omega = 0 \).
    
    \item[(2)] For any link pattern \( \alpha \in LP(n,m) \), with \( m,n \in \mathbb{Z} \), \( 1 \leq m \leq \frac{n}{2} \), and \( n \) even, the corresponding Coulomb gas integrals $\mathcal{K}^{(m,n)}_{\alpha}(\boldsymbol{\theta})$ defined in (\ref{radial excited soultion}) solve the null vector equation \eqref{null vector equation angular coordinate constant h} with 
    \[
    h = \frac{1 - \left(n - 2m + \frac{\kappa}{2} \right)^2}{2\kappa},
    \]
    and the rotation equation \eqref{rotation_invariance} with \( \omega = 0 \).

    \item[(3)] For any link pattern \( \alpha \in LP(n,m) \), with \( m,n \in \mathbb{Z} \) and \( 1 \leq m \leq \frac{n}{2} \), the Coulomb gas integral \( \mathcal{J}^{n,m}_{\alpha}(\boldsymbol{\theta}, \eta) \) solves the null vector equation \eqref{null vector equation angular coordinate constant h} with 
    \[
    h = -\frac{(n - 2m)^2}{2\kappa} + \frac{1 + \eta^2}{2\kappa},
    \]
    and the rotation equation \eqref{rotation_invariance} with 
    \[
    \omega = \frac{\eta(n - 2m)}{\kappa}.
    \]

    \item[(4)] For any link pattern \( \alpha \in LP(n,\frac{n}{2}) \), with \( n \) even, the Coulomb gas integral \( \mathcal{L}_{\alpha}^{n}(\boldsymbol{\theta}) \) defined in (\ref{Chordal Coulomb gas Solution}) solves the null vector equation \eqref{null vector equation angular coordinate constant h} with 
    \[
    h = \frac{(6 - \kappa)(\kappa - 2)}{8\kappa},
    \]
    and the rotation equation \eqref{rotation_invariance} with \( \omega = 0 \).
\end{itemize} 
\end{thm}

Here, \( \alpha \) denotes the integration contour, and \( LP(n,m) \) represents the set of all possible multiple integration contours with \( n \) boundary points and \( m \) integration variables. The abbreviation \( LP \) stands for link pattern, which is defined in Section~\ref{Classification of screening solutions}.

We will discuss the linear independence of these solutions in our forthcoming work.
Understanding the complete classification of the solution space to the null vector equations and rotation equation remains an intriguing open question. 
The classification of the multiple radial SLE($\kappa$) systems can be reduced to studying the positive solutions to the null vector equations and rotation equations.

In section \ref{pure partition functions and meander matrix}, we propose several illuminating conjectures about pure partition functions for multiple radial SLE($\kappa$) and their relations to the Coulomb gas integral solutions.

\subsection{Relations to quantum Calegero-Sutherland system}
We show that a partition function satisfying the null vector equations (\ref{null vector equation angular coordinate constant h}) corresponds to an eigenfunction of the quantum Calogero-Sutherland Hamiltonian, as first discovered in \cite{Car04}.

\begin{thm}\label{CS results kappa>0}
The multiple radial SLE($\kappa$) is described by the partition function $\mathcal{Z}(\boldsymbol{\theta})$, which satisfies the following relation:
\begin{equation}
\mathcal{L}_j \mathcal{Z}(\boldsymbol{\theta}) = h \mathcal{Z}(\boldsymbol{\theta}),
\end{equation}
where $\mathcal{L}_j$ is the null vector differential operator given by:
\begin{equation}
\mathcal{L}_j = \frac{\kappa}{2} \left(\frac{\partial}{\partial \theta_j}\right)^2 
+ \sum_{k \neq j} \left( 
\cot\left(\frac{\theta_k - \theta_j}{2}\right) \frac{\partial}{\partial \theta_k} 
- \frac{6 - \kappa}{2\kappa} \frac{1}{2 \sin^2\left(\frac{\theta_k - \theta_j}{2}\right)}
\right).
\end{equation}

\begin{itemize}
\item[(i)] By transforming the partition function $\mathcal{Z}(\boldsymbol{\theta})$ using the Coulomb gas correlation factor $\Phi_{\frac{1}{\kappa}}^{-1}(\boldsymbol{\theta})$, we obtain:
\begin{equation}
\tilde{\mathcal{Z}}(\boldsymbol{\theta}) = \Phi_{\frac{1}{\kappa}}^{-1}(\boldsymbol{\theta}) \mathcal{Z}(\boldsymbol{\theta}),
\end{equation}
where
\[
\Phi_r(\boldsymbol{\theta}) = \prod_{1 \leq j < k \leq n} \left(\sin\frac{\theta_j - \theta_k}{2}\right)^{-2r}.
\]

The transformed partition function $\tilde{\mathcal{Z}}(\boldsymbol{\theta})$ satisfies:
\[
\left(\Phi_{\frac{1}{\kappa}}^{-1} \cdot \mathcal{L}_j \cdot \Phi_{\frac{1}{\kappa}}\right) \tilde{\mathcal{Z}}(\boldsymbol{\theta}) = h \tilde{\mathcal{Z}}(\boldsymbol{\theta}),
\]
where the differential operator $\Phi_{\frac{1}{\kappa}}^{-1} \cdot \mathcal{L}_j \cdot \Phi_{\frac{1}{\kappa}}$ is given by:
\begin{equation}
\begin{aligned}
\Phi_{\frac{1}{\kappa}}^{-1} \cdot \mathcal{L}_j \cdot \Phi_{\frac{1}{\kappa}} = & \frac{\kappa}{2} \partial_j^2 - F_j \partial_j 
+ \frac{1}{2\kappa} F_j^2 - \frac{1}{2} F_j^{\prime} \\
& - \sum_{k \neq j} \left( 
f_{jk} \left(\partial_k - \frac{1}{\kappa} F_k\right) 
- \frac{6 - \kappa}{2\kappa} f_{jk}^{\prime}
\right).
\end{aligned}
\end{equation}

The sum of the null vector differential operators is:
\begin{equation}
\Phi_{-\frac{1}{\kappa}} \cdot \mathcal{L} \cdot \Phi_{\frac{1}{\kappa}} = \kappa H_n\left(\frac{8}{\kappa}\right) - \frac{n(n^2 - 1)}{6 \kappa},
\end{equation}
where $H_n(\beta)$, with $\beta = \frac{8}{\kappa}$, is the quantum Calogero-Sutherland Hamiltonian:
\[
H_n(\beta) = \sum_{j=1}^n \frac{1}{2} \frac{\partial^2}{\partial \theta_j^2} 
- \frac{\beta(\beta - 2)}{16} \sum_{1 \leq j < k \leq n} \frac{1}{\sin^2\left(\frac{\theta_j - \theta_k}{2}\right)}.
\]

\item[(ii)] The commutation relation between the null vector operators $\mathcal{L}_j$ and $\mathcal{L}_k$ is:
\[
[\mathcal{L}_j, \mathcal{L}_k] = \frac{1}{\sin^2\left(\frac{\theta_j - \theta_k}{2}\right)} (\mathcal{L}_k - \mathcal{L}_j).
\]
As a result:
\[
[\mathcal{L}_j, \mathcal{L}_k] \mathcal{Z}(\boldsymbol{\theta}) 
= \frac{1}{\sin^2\left(\frac{\theta_j - \theta_k}{2}\right)} (\mathcal{L}_k - \mathcal{L}_j) \mathcal{Z}(\boldsymbol{\theta}) = 0.
\]
\end{itemize}
\end{thm}

Notably, the solutions to the null vector PDE system constructed in section \ref{Classification of screening solutions} yield eigenstates of the Calogero-Sutherland system beyond the eigenstates built upon the fermionic ground states.

\newpage

\section{Coulomb gas correlation and rational $SLE(\kappa)$}

\subsection{Schramm Loewner evolutions}

In this section, we briefly recall the basic defintions and properties of the chordal and radial SLE. We will describe the radial Loewner chain in $\mathbb{D}$, where $\mathbb{D}=\{z \in \mathbb{C}|| z \mid<1\}$ and chordal Loewner chain in $\mathbb{H}=\{ \textnormal{Im}(z) >0\}$.

\begin{defn}[Conformal radius]
The conformal radius of a simply connected domain $\Omega$ with respect to a point $z \in \Omega$, defined as
$$
\operatorname{CR}(\Omega, z):=\left|f^{\prime}(0)\right|,
$$
where $f: \mathbb{D} \rightarrow \Omega$ is a conformal map from the open unit disk $\mathbb{D}$ onto $\Omega$ with $f(0)=z$.
\end{defn}
\begin{defn}[Capacity in $\mathbb{D}$]

For any compact subset $K$ of $\overline{\mathbb{D}}$ such that $\mathbb{D} \backslash K$ is simply connected and contains 0 , let $g_K$ be the unique conformal map $\mathbb{D} \backslash K \rightarrow \mathbb{D}$ such that $g_K(0)=0$ and $g_K^{\prime}(0)>0$ . The conformal radius of $\mathbb{D} \backslash K$ is
$$
\operatorname{CR}(\mathbb{D} \backslash K):=\left(g_K^{\prime}(0)\right)^{-1} .
$$

The capacity of $K$ is
$$
\operatorname{cap}(K)=\log g_K^{\prime}(0)=-\log \mathrm{CR}(\mathbb{D} \backslash K,0).
$$

\end{defn}

\begin{defn}[Capacity in $\mathbb{H}$]
    For any compact subset $K \subset \overline{\mathbb{H}}$ such that $\mathbb{H} \backslash K$ is a simply connected domain. The half-plane capacity of a hull $K$ is the quantity
$$
\operatorname{hcap}(K):=\lim _{z \rightarrow \infty} z\left[g_K(z)-z\right] \text {, }
$$
where $g_K: \mathbb{H} \backslash K \rightarrow \mathbb{H}$ is the unique conformal map satisfying the hydrodynamic normalization $g(z)=z+O\left(\frac{1}{z}\right)$ as $z \rightarrow \infty$.
\end{defn}

\begin{defn}[Radial Loewner chain] Let $g_t$ satisfies the radial Loewner equation

\begin{equation}
\partial_t g_t(z)=g_t(z) \frac{\mathrm{e}^{\mathrm{i} \theta_t}+g_t(z)}{\mathrm{e}^{\mathrm{i} \theta_t}-g_t(z)}, \quad g_0(z)=z,
\end{equation}

where $t \mapsto \theta_t$ is real continuous and called the driving function. 
Let $K_t$ be the set of points $z$ in $\mathbb{D}$ such that the solution $g_s(z)$ blows up before or at time $t$.  $K_t$ is called the radial SLE hull driven by $\theta_t$.

Radial Loewner chain in arbitrary simply connected domain $\Omega \subsetneq \mathbb{C}$ with a marked interior point $u \in D$, is defined via a conformal map from $\mathbb{D}$ onto $\Omega$ sending 0 to $u$.

\end{defn}
\begin{defn}[Radial SLE($\kappa$)]
For $\kappa \geq 0$, the radial SLE($\kappa$) is the random Loewner chain in $\mathbb{D}$ from $1$ to $0$ driven by:
\begin{equation}
    \theta_t =\sqrt{\kappa}B_t,
\end{equation}
where $B_t$ is the standard Brownian motion.
\end{defn}

\begin{defn}[Characterization of radial SLE]
The radial SLE is a family $\mathbb{P}(\mathbb{D};\zeta,0)$ of probability measures on curves $\eta:[0, \infty) \rightarrow \overline{\mathbb{D}}$ with $\eta(0)=\zeta$ and parametrized by capacity satisfies the following properties:
\begin{itemize}
    \item (Conformal invariance) For all $a \in \mathbb{R}$, let $\rho_a(z)=\mathrm{e}^{\mathrm{i} a} z$ be the rotation map $\mathbb{D} \rightarrow \mathbb{D}$, the pullback measure $\rho_a^* \mathbb{P}(\mathbb{D};\zeta,0)=\mathbb{P}(\mathbb{D};e^{-ia}\zeta,0)$. From this, we may extend the definition to $\mathbb{P}(\Omega; a,b)$ in any simply connected domain $\Omega$ with an interior marked point $u$ by pulling back using a uniformizing conformal map $\Omega \rightarrow \mathbb{D}$ sending $u$ to 0.
\item (Domain Markov property) given an initial segment $\gamma[0, \tau]$ of the radial $\operatorname{SLE}_\kappa$ curve $\gamma \sim \mathbb{P}(\Omega ; x, y)$ up to a stopping time $\tau$, the conditional law of $\gamma[\tau, \infty)$ is the law $\mathbb{P}\left(\Omega \backslash K_\tau ; \gamma(\tau), 0\right)$ of the $\mathrm{SLE}_\kappa$ curve in the complement of the hull $K_\tau$ from the tip $\gamma(\tau)$ to $0$.
\item (Reflection symmetry) Let $\iota: z \mapsto \bar{z}$ be the complex conjugation, then $\mathbb{P}(\zeta,0) \sim \iota^* \mathbb{P}(\overline{\zeta},0)$.
\end{itemize}
\end{defn}

\begin{defn}[Chordal Loewner chain] Let $g_t$ satisfies the chordal Loewner equation

\begin{equation}
\partial_t g_t(z)= \frac{2}{g_t(z)-\xi(t)}, \quad g_0(z)=z,
\end{equation}

where $t \mapsto \xi_t$ is continuous and called the driving function. 
Let $K_t$ be the set of points $z$ in $\mathbb{H}$ such that the solution $g_s(z)$ blows up before or at time $t$.  $K_t$ is called the chordal SLE hull driven by $\xi_t$

Chordal Loewner chain in arbitrary simply connected domain $\Omega \subsetneq \mathbb{C}$ from $a$ to $b$, is defined via a uniformizing conformal map from $\Omega$ onto $\mathbb{H}$ sending $a$ to $0$ and $b$ to $\infty$.
\end{defn}

\begin{defn}[Chordal SLE($\kappa$)]
For $\kappa \geq 0$, the chordal SLE($\kappa$) is the random Loewner chain in $\mathbb{H}$ from $0$ to $\infty$ driven by
\begin{equation}
    \xi_t =\sqrt{\kappa}B_t,
\end{equation}
where $B_t$ is the standard Brownian motion.
\end{defn}

\begin{defn}[Characterization of Chordal SLE]
Chordal SLE is a family of probability measures on curves $\mathbb{P}(\mathbb{H};a,b)$ $\eta:[0, \infty] \rightarrow \overline{\mathbb{H}}$ with $\eta(0)=a, \eta(\infty)=b$ and parametrized by capacity satisfies the following properties:
    \begin{itemize}
     \item (Conformal invariance) $\rho(z)\in {\rm Aut}(\mathbb{H})$, the pullback measure $\rho^* \mathbb{P}(a,b)=\mathbb{P}(\mathbb{H};\rho(a),\rho(b))$. From this, we may extend the definition of to  $\mathbb{P}(\mathbb{H};z_1,z_2)$ in any simply connected domain $\Omega$ with two boundary points $z_1,z_2$  by pulling back using a uniformizing conformal map $\Omega \rightarrow \mathbb{H}$ sending $z_1$ to $a$ and $z_2$ to $b$.
    \item (Domain Markov property) given an initial segment $\gamma[0, \tau]$ of the $\operatorname{SLE}_\kappa$ curve $\gamma \sim \mathbb{P}(\Omega ; x, y)$ up to a stopping time $\tau$, the conditional law of $\gamma[\tau, \infty)$ is the law $\mathbb{P}\left(\Omega \backslash K_\tau ; \gamma(\tau), y\right)$ of the $\mathrm{SLE}_\kappa$ curve in the complement of the hull $K_\tau$ from the tip $\gamma(\tau)$ to $y$.
    \end{itemize} 
\end{defn}
\subsection{Coulomb gas correlation on Riemann sphere}

To define more general SLE processes beyond the chordal and radial SLEs, we introduce the concept of Coulomb gas correlations. These correlations serve as partition functions for various SLE processes and play a central role in conformal field theory.

We define the Coulomb gas correlations as the (holomorphic) differentials with conformal dimensions $\lambda_j=\sigma_j^2 / 2-\sigma_j b$ at $z_j$ (including infinity) and with values
$$
\prod_{\substack{j<k \\ z_j, z_k \neq \infty}}\left(z_j-z_k\right)^{\sigma_j \sigma_k}, \quad\left(z_j \in \widehat{\mathbb{C}}\right)
$$
in the identity chart of $\mathbb{C}$ and the chart $z \mapsto-1 / z$ at infinity. If $\sigma_j\sigma_k \notin 2\mathbb{Z}$, the Coulomb gas differential is multi-valued; in this case, we choose a single-valued branch.
After explaining this definition, we prove that under the neutrality condition, $\sum \sigma_j=2 b$, the Coulomb gas correlation functions are conformally invariant with respect to the Möbius group ${\rm Aut}(\widehat{\mathbb{C}})$.

\begin{defn}[Differential]\label{differential}
A local coordinate chart on a Riemann surface $M$ is a conformal map $\phi: U \rightarrow \phi(U) \subset \mathbb{C}$ on an open subset $U$ of $M$. A differential $f$ is an assignment of a smooth function $(f \| \phi): \phi(U) \rightarrow \mathbb{C}$ to each local chart $\phi: U \rightarrow \phi(U)$.  $f$ is a differential of conformal dimensions $\left[\lambda, \lambda_*\right]$ if for any two overlapping charts $\phi$ and $\tilde{\phi}$, we have:

\begin{equation}
(f \| \phi)=\left(h^{\prime}\right)^\lambda\left(\overline{h^{\prime}}\right)^{\lambda_*} (\tilde{f} \circ h \| \tilde{\phi}),
\end{equation}

where $h=\tilde{\phi} \circ \phi^{-1}: \phi(U \cap \tilde{U}) \rightarrow \tilde{\phi}(U \cap \tilde{U})$ is the transition map.

\begin{defn}[Neutrality Condition]
A divisor $\boldsymbol{\sigma} : \widehat{\mathbb{C}} \to \mathbb{R}$ is said to satisfy the \emph{neutrality condition} $\mathrm{(NC)}_b$ if
\begin{equation}
\int \boldsymbol{\sigma} = 2b,
\end{equation}
for some $b \in \mathbb{R}$. In the context of $\mathrm{SLE}_\kappa$, the parameter $b$ is related to $\kappa > 0$ by
\begin{equation}
b = \sqrt{\frac{8}{\kappa}} - \sqrt{\frac{\kappa}{2}}.
\end{equation}
\end{defn}

\begin{defn}[Coulomb gas correlations for a divisor on the Riemann sphere] Let the divisor
$$
\boldsymbol{\sigma}=\sum \sigma_j \cdot z_j,
$$
where $\left\{z_j\right\}_{j=1}^n$ is a finite set of distinct points on $\widehat{\mathbb{C}}$.
The Coulomb gas correlation $C_{(b)}[\boldsymbol{\sigma}]$ is a differential of conformal dimension $\lambda_j$ at $z_j$, given by

\begin{equation}
\lambda_j=\lambda_b\left(\sigma_j\right) \equiv \frac{\sigma_j^2}{2}-\sigma_j b ,
\end{equation}

where $\lambda_b(\sigma)=\frac{\sigma^2}{2}-\sigma b \quad(\sigma \in \mathbb{C})$
whose value is given by

\begin{equation}
C_{(b)}[\boldsymbol{\sigma}] =\prod_{j<k}\left(z_j-z_k\right)^{\sigma_j \sigma_k} ,
\end{equation}
where the product is taken over all finite $z_j$ and $z_k$. 

This defines a holomorphic function of $\boldsymbol{z}$ on the configuration space
\[
\mathbb{C}_{\mathrm{distinct}}^n = \left\{ \boldsymbol{z} = (z_1, \ldots, z_n) \in \mathbb{C}^n \,\middle|\, z_j \neq z_k \text{ for } j \neq k \right\}.
\]
In general, the function is multivalued, and one must choose a single-valued branch for each factor $(z_j-z_k)^{\sigma_j\sigma_k}$, except in special cases where all $\sigma_j$ are integers. If all $\sigma_j$ are even integers, the function becomes single-valued and independent of the ordering of the product. In the special case where $\sigma_j = 1$ for all $j$, the correlation function coincides with the Vandermonde determinant.

\end{defn}

\begin{thm}[see \cite{KM21} thm (\textcolor{red}{2.2})]
Under the neutrality condition $\left(\mathrm{NC}_b\right)$, the differentials $C_{(b)}[\boldsymbol{\sigma}]$ are Möbius invariant on $\hat{\mathbb{C}}$.
\end{thm}

\subsection{Coulomb gas correlation in a simply connected domain}

In this section, we define the Coulomb gas correlation differential in a simply connected domain.

\begin{defn}[Symmetric Riemann surface]
A symmetric Riemann surface is a pair $(S, j)$ consisting of a Riemann surface $S$ and an anticonformal involution $j$ on $S$. The latter means that $j$ : $S \rightarrow S$ is an anti-analytic map with $j \cdot j=$ id (the identity map). 
\end{defn}

The principal example for us is the symmetric Riemann surface obtained by taking the Schottky double of a simply connected domain domain. The construction of this is briefly as follows. (See section \textcolor{red}{2.2}, \cite{SS54}, \textcolor{red}{II.3E}, \cite{AS60} for details.)    

\begin{defn}[Schottky double]
Let $\Omega \subsetneq \mathbb{C}$ be a simply connected domain in $\mathbb{C}$ with $\Gamma=\partial \Omega$ consisting of prime ends. Take copy $\tilde{\Omega}$ of $\Omega$ and weld $\Omega$ and $\tilde{\Omega}$ together along $\Gamma$ so that a compact topological surface $\Omega^{Double}=\Omega \cup \Gamma \cup \tilde{\Omega}$ is obtained. If $z \in \Omega$ let $\tilde{z}$ denote the corresponding point on $\tilde{\Omega}$. Then an involution $j$ on $\Omega^{Double}$ is defined by
$$
\begin{array}{ll}
j(z)=\tilde{z} & \text { and } \\
j(\tilde{z})=z & \text { for } z \in \Omega, \\
j(z)=z & \text { for } z \in \Gamma .
\end{array}
$$

The conformal structure on $\tilde{\Omega}$ will be the opposite to that on $\Omega$, which means that the function $\tilde{z} \mapsto \overline{z}$ serves as a local variable on $\tilde{\Omega}$, and $j$ becomes anti-analytic. 

For $p\in \partial \Omega$, let $\phi:U\subset \overline{\Omega} \rightarrow \phi(U)$ be a local boundary chart at $p$, 
let $\tilde{U}$ be the corresponding subset in $\tilde{\Omega}$, 
then $\tilde{\phi}:\tilde{U} \subset \tilde{\Omega} \rightarrow \overline{\phi}(\tilde{U})$ is a local chart at $\tilde{p}$.
Then we can define a local chart $\tau$ for $\Omega^{Double}$ at boundary point $p$ by 
$$ \tau(z)=\left\{
\begin{aligned}
&\phi(z), z \in U \\
& \overline{\phi(z)}, z \in \tilde{U}.
\end{aligned}
\right.
$$

Thus, the conformal structure on $\Omega^{Double}$, inherited from $\mathbb{C}$, extends in a natural way across $\Gamma$ to a conformal structure on all of $\Omega^{Double}$. This makes $\Omega^{Double}$ into a symmetric Riemann sphere.

For example, we identify $\widehat{\mathbb{C}}$ with the Schottky double of $\mathbb{H}$ or that of $\mathbb{D}$. Then the corresponding involution $j$ is $j_{\mathbb{H}}: z \mapsto z^*=\bar{z}$ for $\Omega=\mathbb{H}$ and $j_{\mathbb{D}}: z \mapsto z^*=1 / \bar{z}$ for $\Omega=\mathbb{D}$.
  
\end{defn}

\begin{defn}[Double divisor]
Suppose $\Omega$ is a simply connected domain $(\Omega \subsetneq \mathbb{C})$.

A double divisor $\left(\boldsymbol{\sigma^{+}}, \boldsymbol{\sigma^{-}}\right)$ is a pair of divisor in  $\overline{\Omega}$

 \begin{equation}
  \boldsymbol{\sigma}^{+}=\sum \sigma_j^{+} \cdot z_j, \boldsymbol{\sigma}^{-}=\sum \sigma_j^{-} \cdot z_j.
 \end{equation}

We introduce an equivalence relation for double divisors:

\begin{equation}
    \left(\boldsymbol{\sigma_{1}^{+}}, \boldsymbol{\sigma_{1}^{-}}\right) \sim \left(\boldsymbol{\sigma_{2}^{+}}, \boldsymbol{\sigma_{2}^{-}}\right)
\end{equation}
 if and only if
 \begin{equation}
    \boldsymbol{\sigma_{1}^{+}}+ \boldsymbol{\sigma_{1}^{-}} = \boldsymbol{\sigma_{2}^{+}}+ \boldsymbol{\sigma_{2}^{-}} \quad on \ \partial \Omega.
\end{equation}
Thus, we may choose a representative $\boldsymbol{\sigma}^{-}$ from each equivalence class  that is supported in $\Omega$, i.e., $\sigma_j^{-}=0$ if $z_j \in \partial \Omega$ .

\end{defn}

\begin{defn}
Suppose $\Omega$ is a simply connected domain $(\Omega \subsetneq \mathbb{C})$, let $\partial \Omega$ be its Carathéodory boundary (prime ends) and consider the Schottky double $S=\Omega^{\text {double }}$, which equips with the canonical involution $\iota \equiv \iota_\Omega: S \rightarrow S, z \mapsto z^*$. 

Then, for a double divisor $\left(\boldsymbol{\sigma}^{+}, \boldsymbol{\sigma}^{-}\right)$, we define the associated divisor on the Schottky double $S$ by
\begin{equation}
\boldsymbol{\sigma} = \boldsymbol{\sigma}^{+} + \boldsymbol{\sigma}_*^{-}, \quad \text{where} \quad \boldsymbol{\sigma}_*^{-} := \sum \sigma_j^{-} \cdot z_j^*,
\end{equation}
and each $z_j^*$ denotes the image of $z_j$ under the canonical involution $\iota$ of $S$. Accordingly, $\boldsymbol{\sigma}_*^{-}$ is the pushforward of $\boldsymbol{\sigma}^{-}$ under $\iota$.

\end{defn}

\begin{defn}[Neutrality condition]
 A double divisor $\left(\boldsymbol{\sigma}^{+}, \boldsymbol{\sigma}^{-}\right)$ satisfies the neutrality condition $\left(\mathrm{NC}_b\right)$  if 
\begin{equation}
\int\boldsymbol{\sigma}=\int \boldsymbol{\sigma}^{+}+ \int \boldsymbol{\sigma}^{-}=2b.
\end{equation}   
\end{defn}

\begin{defn}[Coulomb gas correlation for a double divisor in a simply connected domain]
For a double divisor $(\boldsymbol{\sigma^+},\boldsymbol{\sigma^-})$, let $\boldsymbol{\sigma}=\boldsymbol{\sigma^+}+ \boldsymbol{\sigma}_*^{-}$ be its corresponding divisor in the Schottky double $S$, we define the Coulomb gas correlation of the double divisor $(\boldsymbol{\sigma^+},\boldsymbol{\sigma^-})$ by
\end{defn}
\begin{equation}
 C_{\Omega}\left[\boldsymbol{\sigma}^{+}, \boldsymbol{\sigma}^{-}\right](\boldsymbol{z}):=C_S[\boldsymbol{\sigma}] .  
\end{equation}

We often omit the subscripts $\Omega, S$ to simplify the notations. 

If the double divisor $(\boldsymbol{\sigma}^{+}, \boldsymbol{\sigma}^{-})$ satisfies the neutrality condition $(\mathrm{NC}_b)$, then the Coulomb gas correlation function
$
C_{\Omega}\left[\boldsymbol{\sigma}^{+}, \boldsymbol{\sigma}^{-}\right]
$
is a well-defined differential on $\Omega$, with conformal weights $\left[\lambda_j^{+}, \lambda_j^{-}\right]$ at each point $z_j \in \Omega$.

If $z_j \in \partial \Omega$, then the differential is with respect to a boundary chart: that is, a local conformal map from a neighborhood of $z_j$ in $\Omega$ to the upper half-plane $\mathbb{H}$, sending $z_j$ to a boundary point of $\mathbb{H}$. The derivative $\partial_{z_j}$ is then defined as the holomorphic derivative in this local coordinate.

\begin{equation}
\lambda^{+}_{j}=\lambda_b\left(\sigma^{+}_j\right) \equiv \frac{(\sigma^{+}_j)^2}{2}-\sigma_{j}^{+} b,\quad
\lambda^{-}_{j}=\lambda_b\left(\sigma_j\right) \equiv \frac{(\sigma^{-}_{j})^2}{2}-\sigma^{-}_j b.
\end{equation}

By conformal invariance of the Coulomb gas correlation differential $C_S[\boldsymbol{\sigma}]$ on the Riemann sphere under M{\"o}bius transformation, the Coulomb gas correlation differential $C_{\Omega}\left[\boldsymbol{\sigma}^{+}, \boldsymbol{\sigma}^{-}\right](\boldsymbol{z})$ is invariant under $Aut(\Omega)$.
\end{defn}

\begin{thm}[see \cite{KM21} thm (\textcolor{red}{2.4})]
 Under the neutrality condition $\left(\mathrm{NC}_b\right)$, the value of the differential $C_{\mathbb{H}}\left[\boldsymbol{\sigma}^{+}, \boldsymbol{\sigma}^{-}\right]$ in the identity chart of $\mathbb{H}$ (and the chart $z \mapsto-1 / z$ at infinity) is given by
 
 \begin{equation}
C_{\mathbb{H}}\left[\boldsymbol{\sigma}^{+}, \boldsymbol{\sigma}^{-}\right]=\prod_{j<k}\left(z_j-z_k\right)^{\sigma_j^{+} \sigma_k^{+}}\left(\bar{z}_j-\bar{z}_k\right)^{\sigma_j^{-} \sigma_k^{-}} \prod_{j, k}\left(z_j-\bar{z}_k\right)^{\sigma_j^{+} \sigma_i^{-}},
\end{equation}

where the products are taken over finite $z_j$ and $z_k$.
\end{thm}

\begin{example}
We have
\begin{itemize}
    \item[(i)] if $\boldsymbol{\sigma}^{-}=\mathbf{0}$, then (up to a phase)
$$
C_{\mathbb{H}}\left[\boldsymbol{\sigma}^{+}, \mathbf{0}\right]=\prod_{j<k}\left(z_j-z_k\right)^{\sigma_j^{+} \sigma_k^{+}};
$$
\item[(ii)] if $\boldsymbol{\sigma}^{-}=\overline{\boldsymbol{\sigma}^{+}}$, then (up to a phase)
$$
C_{\mathbb{H}}\left[\boldsymbol{\sigma}^{+}, \overline{\boldsymbol{\sigma}^{+}}\right]=\prod_{j<k}\left|\left(z_j-z_k\right)^{\sigma_j^{+} \sigma_k^{+}}\left(z_j-\bar{z}_k\right)^{\sigma_j^{+} \overline{\sigma_k^{+}}}\right|^2 \prod_{\operatorname{Im} z_j>0}\left(2 \operatorname{Im} z_j\right)^{\left|\sigma_j^{+}\right|^2} ;
$$
\item[(iii)] if $\boldsymbol{\sigma}^{-}=-\overline{\boldsymbol{\sigma}^{+}}$, then (up to a phase)
$$
C_{\mathbb{H}}\left[\boldsymbol{\sigma}^{+},-\overline{\boldsymbol{\sigma}^{+}}\right]=\prod_{j<k}\left|\left(z_j-z_k\right)^{\sigma_j^{+} \sigma_k^{+}}\left(z_j-\overline{z_k}\right)^{-\sigma_j^{+} \overline{\sigma_k^{+}}}\right|^2 \prod_{\operatorname{Im} z_j>0}\left(2 \operatorname{Im} z_j\right)^{-\left|\sigma_j^{+}\right|^2} .
$$

where the products are taken over finite $z_j$ and $z_k$.
\end{itemize}
\end{example}

\begin{thm}[see \cite{KM21} thm (\textcolor{red}{2.5})]
Under the neutrality condition $\left(\mathrm{NC}_b\right)$, the value of the differential $C_{\mathbb{D}}\left[\boldsymbol{\sigma}^{+}, \boldsymbol{\sigma}^{-}\right]$ in the identity chart of $\mathbb{D}$ is given by
\begin{equation}
C_{\mathbb{D}}\left[\boldsymbol{\sigma}^{+}, \boldsymbol{\sigma}^{-}\right]=\prod_{j<k}\left(z_j-z_k\right)^{\sigma_j^{+} \sigma_k^{+}}\left(\bar{z}_j-\bar{z}_k\right)^{\sigma_j^{-} \sigma_k^{-}} \prod_{j, k}\left(1-z_j \bar{z}_k\right)^{\sigma_j^{+} \sigma_k^{-}},
\end{equation}
where the product is taken over finite $z_j$ and $z_k$.
\end{thm}

\subsection{Rational $SLE_\kappa[\boldsymbol{\sigma}]$}

\begin{defn}[Rational SLE] \label{rational SLE}
In the unit disk $\mathbb{D}$, let $e^{i\theta} \in \partial\mathbb{D}$ be the growth point, and let $u_1 = e^{i\theta_1}, u_2 = e^{i\theta_2}, \dots, u_k = e^{i\theta_k} \in \overline{\mathbb{D}}$ be marked points. A symmetric double divisor $(\boldsymbol{\sigma}^+, \boldsymbol{\sigma}^-)$ assigns a charge distribution on $e^{i\theta}$ and $\{u_1, \dots, u_k\}$, where
\[
\boldsymbol{\sigma}^+ = a \cdot e^{i\theta} + \sum_{j=1}^{k} \sigma_j \cdot u_j,
\quad \text{and} \quad
\boldsymbol{\sigma}^- = \overline{\boldsymbol{\sigma}^+}|_{\mathbb{D}},
\]
and the total charge satisfies the neutrality condition $(\mathrm{NC}_b)$.

We define the rational $\mathrm{SLE}_\kappa[\boldsymbol{\sigma}]$ as a random normalized conformal map $g_t(z)$, with initial condition $g_0(z) = z$ and normalization $g_t'(0) = e^{-t}$. It evolves according to the radial Loewner equation:
\[
\partial_t g_t(z) = g_t(z) \frac{e^{i\theta(t)} + g_t(z)}{e^{i\theta(t)} - g_t(z)}, \quad g_0(z) = z.
\]

Let $h_t(z)$ be the covering map of $g_t(z)$, defined via
\[
e^{i h_t(z)} = g_t(e^{i z}),
\]
so that $h_0(z) = z$, and
\[
\partial_t h_t(z) = \cot\left( \frac{h_t(z) - \theta(t)}{2} \right).
\]

The driving function $\theta(t)$ evolves as
\[
d\theta(t) = \frac{\partial \log \mathcal{Z}(\theta)}{\partial \theta} \, dt + \sqrt{\kappa} \, dB_t,
\]
where the Coulomb gas partition function is
\begin{equation} \label{Coulomb-gas-theta}
\mathcal{Z}(\boldsymbol{\theta}) = \prod_{j < k} \sin\left( \frac{\theta_j - \theta_k}{2} \right)^{\sigma_j \sigma_k} 
\cdot \prod_j e^{\frac{i}{2} \sigma_j (\sigma_0 - \sigma_\infty)\theta_j}.
\end{equation}

The flow map $g_t$ is well-defined up to the first time $\tau$ when $\zeta(t) = g_t(w)$ for some $w$ in the support of $\boldsymbol{\sigma}$. For any $z \in \mathbb{D}$, the process $t \mapsto g_t(z)$ is defined up to time $\tau_z \wedge \tau$, where $\tau_z$ is the first time such that $g_t(z) = e^{i\theta(t)}$. Define the associated hull by
\[
K_t = \left\{ z \in \overline{\mathbb{D}} : \tau_z \le t \right\}.
\]

Furthermore, the law of the rational $\mathrm{SLE}_\kappa[\boldsymbol{\sigma}]$ Loewner chain is invariant under Möbius transformations $\mathrm{Aut}(\mathbb{D})$, up to a time change, due to the conformal covariance of the Coulomb gas correlation functions. Thus, we define rational $\mathrm{SLE}_\kappa[\boldsymbol{\sigma}]$ in a general simply connected domain $\Omega$ via a conformal map $\phi: \Omega \to \mathbb{D}$ by pulling back the flow.
\end{defn}

In definition (\ref{rational SLE}), we define the rational SLE from the perspective of the partition function. This approach helps us to understand the SLE within the framework of conformal field theory and can be naturally extended to various settings, including multiple SLE($\kappa$) systems.

\begin{example}Double divisor for chordal and radial $\mathrm{SLE}(\kappa, \rho)$, where $\xi$ denotes the growth point and $q$ is the marked boundary point (in the chordal case) or interior point (in the radial case). 

\end{example}

\begin{figure}[htbp]
\centering
\begin{minipage}[b]{0.49\textwidth}
\centering
\includegraphics[width=5.9cm]{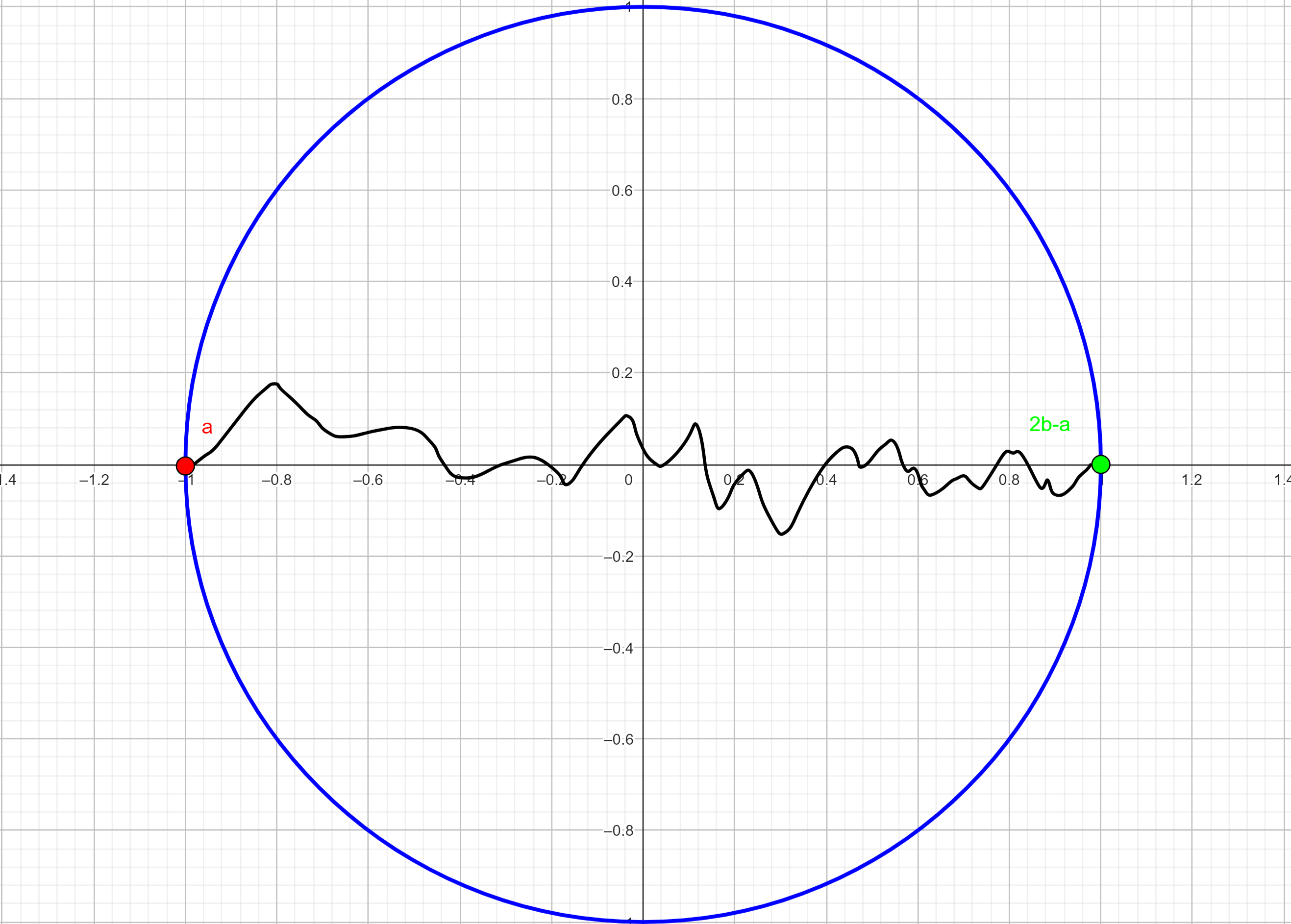}
\caption{Chordal SLE($\kappa$) $\boldsymbol{\sigma}^+= a \cdot \xi+ (2b-a) \cdot q$,
$\boldsymbol{\sigma}^-= 0$}
\end{minipage}
\begin{minipage}[b]{0.49\textwidth}
\centering
\includegraphics[width=5.9cm]{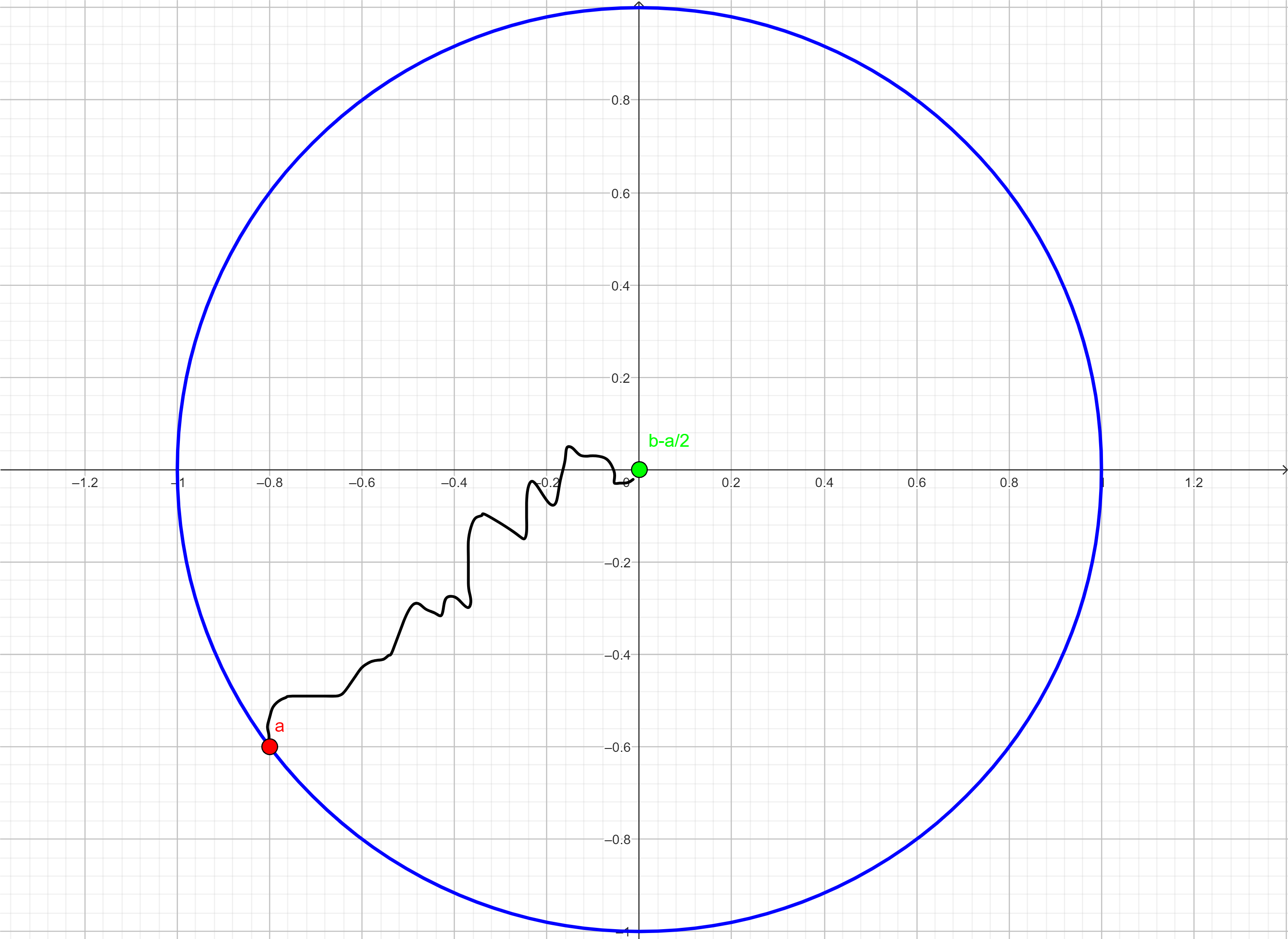}
\caption{Radial SLE($\kappa$)
$\boldsymbol{\sigma}^+= a \cdot \xi+ (b-a) \cdot q$,
$\boldsymbol{\sigma}^-= b \cdot q$}
\end{minipage}

\end{figure}

In addition to the aforementioned definition, another widely used equivalent is known as SLE($\kappa$,$\rho$). We prove the equivalence between rational $SLE_{\kappa}[\boldsymbol{\sigma^+},\boldsymbol{\sigma^-}]$ and SLE($\kappa$,$\rho$) in the following theorem.

\begin{defn}[Radial SLE($\kappa,\rho$)] Let $\xi$ be the growth point on the unit circle, and let
\[
\boldsymbol{\rho} = \sum_{j=1}^{n} \rho_j \delta_{u_j} + \sigma_0 \cdot \delta_0 + \sigma_{\infty} \cdot \delta_{\infty}
\]
be a divisor on $\widehat{\mathbb{C}}$, where $\rho_j \in \mathbb{C}$, and the divisor $\boldsymbol{\rho}$ is symmetric under inversion, i.e.,
\[
\boldsymbol{\rho}(z) = \overline{\boldsymbol{\rho}\left(\frac{z}{|z|^2}\right)} \quad \text{for all } z \in \widehat{\mathbb{C}}.
\]
We say $\boldsymbol{\rho}$ satisfies the neutrality condition for $\mathrm{SLE}(\kappa,\rho)$ if
\[
\int \boldsymbol{\rho} = \kappa - 6.
\].

Define the radial $\operatorname{SLE}(\kappa, \xi, \boldsymbol{\rho})$ Loewner chain by

\begin{equation}
\partial_t g_t(z)=g_t(z)\frac{\xi(t)+g_t(z)}{\xi(t)-g_t(z)}, \quad g_0(z)=z.
\end{equation}

Let $\xi(t)=e^{i\theta(t)}$, $u_j=e^{iq_j}$ and $h_t(z)$ be the covering map of $g_t(z)$ (i.e. $h_t(z)=g_t(e^{iz})$) , then the Loewner differential equation for $h_t(z)$ is given by
\begin{equation}
\partial_t h_t(z)=\cot(\frac{h_t(z)-\theta(t)}{2}), \quad h_0(z)=z,  
\end{equation}
 the driving function $\theta(t)$ evolves as

\begin{equation}
d\theta(t)=\sqrt{\kappa}dB_t+\sum_j \rho_j \cot( \frac{\theta(t)-q_j(t)}{2}).
\end{equation}

\end{defn}
Note that although the lifts of $\theta(t)$ in universal cover are not unique, different lifts lead to the same differential equation for $h_t(z)$ by periodicity $\cot(z+k\pi)=\cot(z)$, $k \in \mathbb{Z}$.

\begin{thm}
For a symmetric double divisor  $\boldsymbol{\sigma^+}=a\cdot \xi+ \sum \sigma_j \cdot u_j$ and $\boldsymbol{\sigma^-} = \overline{\boldsymbol{\sigma}^+}|_{\Omega}$ satisfying neutrality condition ($NC_{b}$), let $\boldsymbol{\rho}=\sum_{j=1}^{m}\rho_j \cdot u_j$ where $\rho_j = (\kappa a )\sigma_j$. Then two definitions $SLE_{\kappa}[\boldsymbol{\sigma^+},\boldsymbol{\sigma^-}]$ and SLE($\kappa,\rho$) are equivalent.
\end{thm}

\begin{proof}
The equivalence in one chart can be verified by directly computing the drift term in the Loewner equation. The conformal invariance of $\mathrm{SLE}(\kappa,\rho)$ under the neutrality condition ($NC_b$), where the divisor $\boldsymbol{\rho}$ consists of real charges, is established in \cite{SW05}. Moreover, their argument extends naturally to the case where the charges $\boldsymbol{\rho}$ are complex.
\end{proof}

\subsection{Classical limit of Coulomb gas correlation}
Now, we extend our definition of Coulomb gas correlation to $\kappa=0$ by normalizing the Coulomb gas correlation.

\begin{defn}[Normalized Coulomb gas correlations for a divisor on the Riemann sphere] Let the divisor
$$
\boldsymbol{\sigma}=\sum \sigma_j \cdot z_j,
$$
where $\left\{z_j\right\}_{j=1}^n$ is a finite set of distinct points on $\widehat{\mathbb{C}}$.
The normalized Coulomb gas correlation $C[\boldsymbol{\sigma}]$ is a differential of conformal dimension $\lambda_j$ at $z_j$ by

Let $\lambda(\sigma)=\sigma^2+2\sigma  \quad(\sigma \in \mathbb{R})$.
\begin{equation}
\lambda_j=\lambda_b\left(\sigma_j\right) \equiv \sigma_j^2+2\sigma_j  ,
\end{equation}

whose value is given by

\begin{equation}
C[\boldsymbol{\sigma}] =\prod_{j<k}\left(z_j-z_k\right)^{2\sigma_j \sigma_k} ,
\end{equation}
where the product is taken over all finite $z_j$ and $z_k$.

\end{defn}

\begin{remark}
    The normalized Coulomb gas correlation can be viewed as taking the $\kappa \rightarrow 0$ limit of the divisor $\sqrt{2\kappa}\boldsymbol{\sigma}$, the Coulomb gas correlation function $C_{(b)}[\boldsymbol{\sigma}]^{\kappa}$, and conformal dimension $\kappa \lambda_j$.
\end{remark}

\begin{defn}[Neutrality condition] A divisor $\boldsymbol{\sigma}: \widehat{\mathbb{C}} \rightarrow \mathbb{R}$ satisfies the neutrality condition if
\begin{equation}
\int \boldsymbol{\sigma}=-2.
\end{equation}
\begin{thm}
Under the neutrality condition $\int\boldsymbol{\sigma}=-2$, the normalized Coulomb gas correlation differentials $C[\boldsymbol{\sigma}]$ are Möbius invariant on $\hat{\mathbb{C}}$.
\end{thm}

\begin{proof}
By direct computation, similar to the $\kappa>0$ case.
\end{proof}

\begin{defn}[Coulomb gas correlation for a double divisor in a simply connected domain]
For a double divisor $(\boldsymbol{\sigma^+},\boldsymbol{\sigma^-})$, let $\boldsymbol{\sigma}=\boldsymbol{\sigma^+}+ \boldsymbol{\sigma}_*^{-}$ be its corresponding divisor in the Schottky double $S$, we define the Coulomb gas correlation of the double divisor $(\boldsymbol{\sigma^+},\boldsymbol{\sigma^-})$ by:
\end{defn}
\begin{equation}
 C_{\Omega}\left[\boldsymbol{\sigma}^{+}, \boldsymbol{\sigma}^{-}\right](\boldsymbol{z}):=C_S[\boldsymbol{\sigma}] .  
\end{equation}

We often omit the subscripts $\Omega, S$ to simplify the notations. 

If the double divisor $(\boldsymbol{\sigma}^{+},\boldsymbol{\sigma}^{-})$ satisfies the neutrality condition, then $C\left[\boldsymbol{\sigma}^{+}, \boldsymbol{\sigma^{-}}\right]$is a well-defined differential with conformal dimensions $\left[\lambda_j^{+}, \lambda_j^{-}\right]$at $z_j$.

\begin{equation}
\lambda^{+}_{j}=\lambda\left(\sigma^{+}_j\right) \equiv \frac{(\sigma^{+}_j)^2}{2}+2\sigma_{j}^{+} ,\quad
\lambda^{-}_{j}=\lambda \left(\sigma_j\right) \equiv \frac{(\sigma^{-}_{j})^2}{2}+2\sigma^{-}_j .  
\end{equation}

By conformal invariance of the Coulomb gas correlation differential $C_S[\boldsymbol{\sigma}]$ on the Riemann sphere under M{\"o}bius transformation, the Coulomb gas correlation differential $C_{\Omega}\left[\boldsymbol{\sigma}^{+}, \boldsymbol{\sigma}^{-}\right](\boldsymbol{z})$ is invariant under $Aut(\Omega)$.
\end{defn}

\begin{defn}[Neutrality condition]
 A double divisor $\left(\boldsymbol{\sigma}^{+}, \boldsymbol{\sigma}^{-}\right)$ satisfies the neutrality condition  if 
\begin{equation}
\int\boldsymbol{\sigma}=\int \boldsymbol{\sigma}^{+}+ \int \boldsymbol{\sigma}^{-}=-2.
\end{equation}   
\end{defn}

\begin{thm}
 Under the neutrality condition $\int \boldsymbol{\sigma}^{+}+ \int \boldsymbol{\sigma}^{-}=-2$ , the value of the differential $C_{\mathbb{H}}\left[\boldsymbol{\sigma}^{+}, \boldsymbol{\sigma}^{-}\right]$ in the identity chart of $\mathbb{H}$ (and the chart $z \mapsto-1 / z$ at infinity) is given by
 
 \begin{equation}
C_{\mathbb{H}}\left[\boldsymbol{\sigma}^{+}, \boldsymbol{\sigma}^{-}\right]=\prod_{j<k}\left(z_j-z_k\right)^{2\sigma_j^{+} \sigma_k^{+}}\left(\bar{z}_j-\bar{z}_k\right)^{2\sigma_j^{-} \sigma_k^{-}} \prod_{j, k}\left(z_j-\bar{z}_k\right)^{2\sigma_j^{+} \sigma_i^{-}},
\end{equation}

where the products are taken over finite $z_j$ and $z_k$.
\end{thm}

\begin{thm}
Under the neutrality condition $\int \boldsymbol{\sigma}^{+}+ \int \boldsymbol{\sigma}^{-}=-2$, the value of the differential $C_{\mathbb{D}}\left[\boldsymbol{\sigma}^{+}, \boldsymbol{\sigma}^{-}\right]$ in the identity chart of $\mathbb{D}$ is given by
\begin{equation}
C_{\mathbb{D}}\left[\boldsymbol{\sigma}^{+}, \boldsymbol{\sigma}^{-}\right]=\prod_{j<k}\left(z_j-z_k\right)^{2\sigma_j^{+} \sigma_k^{+}}\left(\bar{z}_j-\bar{z}_k\right)^{2\sigma_j^{-} \sigma_k^{-}} \prod_{j, k}\left(1-z_j \bar{z}_k\right)^{2\sigma_j^{+} \sigma_k^{-}} ,
\end{equation}
where the product is taken over finite $z_j$ and $z_k$.
\end{thm}

\subsection{Rational $SLE_{0}[\boldsymbol{\sigma}]$}
\begin{defn}[Rational $\mathrm{SLE}_0$] \label{rational SLE0}
In the unit disk $\mathbb{D}$, let $e^{i\theta} \in \partial \mathbb{D}$ be the growth point, and let $u_1, u_2, \dots, u_m \in \overline{\mathbb{D}}$ be marked points. A symmetric double divisor $(\boldsymbol{\sigma}^+, \boldsymbol{\sigma}^-)$ assigns a charge distribution on $e^{i\theta}$ and $\{u_1, \dots, u_k\}$, where
\[
\boldsymbol{\sigma}^+ = a \cdot e^{i\theta} + \sum_{j=1}^{k} \sigma_j \cdot u_j,
\quad \text{and} \quad
\boldsymbol{\sigma}^- = \overline{\boldsymbol{\sigma}^+}|_{\mathbb{D}},
\]
and the total charge satisfies the neutrality condition $\int\boldsymbol{\sigma}=-2$.

We define the rational $\mathrm{SLE}_0[\boldsymbol{\sigma}]$ Loewner chain as a normalized conformal map $g_t(z)$ with initial conditions $g_0(z) = z$ and $g_t'(0) = e^{-t}$. The evolution of $g_t$ is governed by the Loewner differential equation:
\[
\partial_t g_t(z) = g_t(z) \frac{e^{i\theta(t)} + g_t(z)}{e^{i\theta(t)} - g_t(z)}, \quad g_0(z) = z.
\]

In the angular coordinate, let $h_t(z)$ be the covering map of $g_t(z)$ defined by $e^{i h_t(z)} = g_t(e^{i z})$. Then $h_t(z)$ evolves according to
\[
\partial_t h_t(z) = \cot\left( \frac{h_t(z) - \theta(t)}{2} \right), \quad h_0(z) = z.
\]

The driving function $\theta(t)$ evolves according to
\[
d\theta(t) = \frac{\partial \log \mathcal{Z}(\boldsymbol{\theta})}{\partial \theta} \, dt.
\]

where the Coulomb gas partition function is
\begin{equation} 
\mathcal{Z}(\boldsymbol{\theta}) = \prod_{j < k} \sin\left( \frac{\theta_j - \theta_k}{2} \right)^{\sigma_j \sigma_k} 
\cdot \prod_j e^{\frac{i}{2} \sigma_j (\sigma_0 - \sigma_\infty)\theta_j}.
\end{equation}

The flow $g_t$ is well-defined up to the first time $\tau$ at which $w(t) = g_t(w)$ for some $w$ in the support of $\boldsymbol{\sigma}$. For each $z \in \overline{\mathbb{D}}$, the process $t \mapsto g_t(z)$ is well-defined up to $\tau_z \wedge \tau$, where $\tau_z$ is the first time such that $g_t(z) = e^{i\theta(t)}$. Denote
\[
K_t = \left\{ z \in \overline{\mathbb{D}} : \tau_z \leq t \right\}
\]
as the hull associated with the Loewner chain.

Furthermore, the rational $\mathrm{SLE}_0$ Loewner chain is invariant under Möbius transformations in $\mathrm{Aut}(\mathbb{D})$ (up to a time reparameterization), due to the conformal invariance of the Coulomb gas correlation. Consequently, rational $\mathrm{SLE}_0[\boldsymbol{\sigma}]$ in any simply connected domain $\Omega$ is defined by pulling back via a conformal map $\phi: \Omega \rightarrow \mathbb{D}$.
\end{defn}

In definition (\ref{rational SLE0}), we introduce the definition of $SLE_{0}[\beta]$ as a natural extension of $SLE_{\kappa}[\beta]$ to $\kappa=0$. The main ingredient in our definition is the normalized Coulomb gas as the partition function.

Now, we introduce another widely used definition ${\rm SLE}(0,\boldsymbol{\rho})$ which is a natural extension of ${\rm SLE}(\kappa,\boldsymbol{\rho})$ to $\kappa=0$.
We prove the equivalence between rational $SLE_0[\boldsymbol{\sigma}]$ and ${\rm SLE}(0,\boldsymbol{\rho})$ in the end.

\begin{defn}[${\rm SLE}(0,\boldsymbol{\rho})$]
 Let $w$ be the growth point on $\partial \mathbb{D}$ and $\boldsymbol{\rho}=\sum_{i=1}^{n} \rho_j \delta_{u_j} + \sigma_0 \cdot 0 + \sigma_{\infty} \cdot \infty$ be a divisor on $\widehat{\mathbb{C}}$ that is symmetric under involution, i.e. $\boldsymbol{\rho}(z)=\boldsymbol{\rho}(\frac{z}{|z|^2})$ for all $z$ and $\int \boldsymbol{\rho}= -6$. Define the radial $\operatorname{SLE}(0,w, \boldsymbol{\rho})$ Loewner chain by
\begin{equation}
\partial_t g_t(z)=g_t(z)\frac{w(t)+g_t(z)}{w(t)-g_t(z)}, \quad g_0(z)=z,
\end{equation}
where the driving function $w(t)$ evolves as

\begin{equation}
\dot{w}(t)=w(t)\sum_j \rho_j\frac{g_t(u_j)+w(t)}{g_t(u_j)-w(t)}, \quad z(0)=z_0.
\end{equation}

In the angular coordinate, $w(t)=e^{i\theta(t)}$ and $u_j(t)= e^{i q_j(t)}$, let $h_t(z)$ be the covering conformal map of $g_t(z)$ (i.e. $e^{ih_t(z)}=g_t(e^{iz}))$.

Then the Loewner differential equation for $h_t(z)$ is 

\begin{equation}
\partial_t h_t(z)=\cot(\frac{h_t(z)-\theta_t}{2}), \quad h_0(z)=z,     z\in{\overline{\mathbb{H}}},
\end{equation}

where the driving function $\theta_t$ evolves as
\begin{equation}
\dot{\theta}_t=\sum_j \rho_j \cot( \frac{\theta_t-q_j(t)}{2}), \quad x(0)=x_0.
\end{equation}

\end{defn}

\begin{thm}
  For an involution symmetric divisor $\boldsymbol{\sigma}= w + \sum_{j=1}^{m}\sigma_j \cdot z_j$ satisfying neutrality condition $\int \boldsymbol{\sigma}= -2$, let $\boldsymbol{\rho}=2\sum_{j=1}^{m}\sigma_j \cdot z_j$, then  $\int \boldsymbol{\rho}= -6$ and two definitions $SLE_{0}[\sigma]$ and SLE($0,\rho$) are equivalent.  
\end{thm}

\begin{proof}
The equivalence in one chart can be verified by directly computing the drift term in the Loewner equation. The conformal invariance of $\mathrm{SLE}(\kappa,\rho)$ under the neutrality condition ($NC_b$), where the divisor $\boldsymbol{\rho}$ consists of real charges, is established in \cite{SW05}. Moreover, their argument extends naturally to the case where the charges $\boldsymbol{\rho}$ are complex.
\end{proof}

\section{Commutation relations and conformal invariance} \label{Commutation relations and conformal invariance}

\subsection{Transformation of Loewner flow under coordinate change}
\label{transformation of Loewner under coordinate change}

In this section we show that the Loewner chain of a curve, when viewed in a different coordinate chart, is a time reparametrization of the Loewner chain in the standard coordinate chart but with different initial conditions. This result serves as a preliminary step towards understanding the local commutation relations and the conformal invariance of multiple SLE($\kappa$) systems.

Let us briefly review how Loewner chains transform under coordinate changes.

\begin{thm}[Deterministic Loewner chain under coordinate change]
\label{Loewner coordinate change in angular coordinate}
In angular (trigonometric) coordinates, suppose $\gamma(0) = \theta \in \mathbb{R}$ and let the marked points be $\theta_1, \theta_2, \ldots, \theta_n \in \mathbb{C}$. Let $\gamma(t)$ be the curve generated by the deterministic Loewner chain:
\begin{equation}
\partial_t h_t(z) = \cot\left( \frac{h_t(z) - \theta_t}{2} \right), \qquad \dot{\theta}_t = b\left( \theta_t ; h_t(\theta_1), \ldots, h_t(\theta_n) \right),
\end{equation}
with initial condition $h_0(z) = z$ and $\theta_0 = \theta$, where $b : \mathbb{R} \times \mathbb{C}^n \to \mathbb{R}$ is a smooth vector field.

Let $\tau : \mathcal{N} \to \mathbb{H}$ be a conformal map defined on a neighborhood $\mathcal{N}$ of $\theta$ such that $\gamma[0,T] \subset \mathcal{N}$ and $\tau(\partial \mathcal{N} \cap \mathbb{R}) \subset \mathbb{R}$. Define the image curve $\widetilde{\gamma}(t) := \tau(\gamma(t))$ for $t \in [0, T]$, and let $\widetilde{h}_t$ denote the conformal map associated with $\widetilde{\gamma}[0, t]$.

Define the conformal coordinate change
\[
\Psi_t := \widetilde{h}_t \circ \tau \circ h_t^{-1}.
\]

Then the image conformal map $\widetilde{h}_t(z)$ satisfies the evolution equation:
\begin{equation}
\partial_t \widetilde{h}_t(z) = \cot\left( \frac{\widetilde{h}_t(z) - \widetilde{\theta}_t}{2} \right) \cdot \left[ \Psi_t'(\theta_t) \right]^2, \qquad \widetilde{h}_0(z) = z,
\end{equation}
where the new driving function is
\[
\widetilde{\theta}_t := \widetilde{h}_t \circ \tau \circ h_t^{-1}(\theta_t) = \Psi_t(\theta_t), \quad \widetilde{\theta}_0 = \tau(\theta).
\]

Moreover, the curve $\widetilde{\gamma}(t)$ is parameterized so that its unit disk capacity satisfies
\begin{equation}
\operatorname{hcap}(\widetilde{\gamma}[0, t]) = 2 \sigma(t), \qquad \text{where} \quad \sigma(t) := \int_0^t \left| \Psi_s'(\theta_s) \right|^2 \, ds.
\end{equation}
\end{thm}

\begin{proof}
See Section 4.6.2 in \cite{Law05}.
\end{proof}

\begin{thm}[Stochastic Loewner chain under coordinate change]
\label{random coordinate change}
Suppose the driving function $\theta_t$ evolves according to the stochastic differential equation
\begin{equation}
d\theta_t = \sqrt{\kappa}\, dB_t + b\left(\theta_t; \Psi_t(\theta_1), \ldots, \Psi_t(\theta_n)\right)\, dt,
\end{equation}
where $B_t$ is standard Brownian motion, and $\Psi_t := \widetilde{h}_t \circ \tau \circ h_t^{-1}$ is the conformal coordinate change defined as in Theorem~\ref{Loewner coordinate change in angular coordinate}.

Define the transformed driving function
\[
\widetilde{\theta}_t := \Psi_t(\theta_t),
\]
and introduce the reparameterized time
\[
s(t) := \int_0^t |\Psi_u'(\theta_u)|^2 \, du.
\]

Then the process $\widetilde{\theta}_s := \Psi_{t(s)}(\theta_{t(s)})$ satisfies the following SDE:
\begin{equation}
\label{coordinate-changed-sde}
d\widetilde{\theta}_s = \sqrt{\kappa}\, dB_s + \frac{b\left(\theta_s; \Psi_{t(s)}(\theta_1), \ldots, \Psi_{t(s)}(\theta_n)\right)}{\Psi_{t(s)}'(\theta_s)}\, ds + \frac{\kappa - 6}{2} \cdot \frac{\Psi_{t(s)}''(\theta_s)}{[\Psi_{t(s)}'(\theta_s)]^2} \, ds.
\end{equation}
\end{thm}

\begin{proof}
We apply Itô's formula to the composed process $\widetilde{\theta}_t = \Psi_t(\theta_t)$. Using the chain rule for semimartingales:

\begin{align*}
d\widetilde{\theta}_t &= (\partial_t \Psi_t)(\theta_t)\, dt + \Psi_t'(\theta_t)\, d\theta_t + \frac{1}{2} \Psi_t''(\theta_t)\, d\langle \theta \rangle_t \\
&= (\partial_t \Psi_t)(\theta_t)\, dt + \Psi_t'(\theta_t) \left[ \sqrt{\kappa}\, dB_t + b(\theta_t; \Psi_t(\theta_1), \ldots, \Psi_t(\theta_n))\, dt \right] + \frac{\kappa}{2} \Psi_t''(\theta_t)\, dt.
\end{align*}

From Proposition 4.43 in \cite{Law05}, we use the identity
\[
(\partial_t \Psi_t)(\theta_t) = -3 \Psi_t''(\theta_t).
\]
Substituting into the equation:

\begin{align*}
d\widetilde{\theta}_t &= \Psi_t'(\theta_t) \sqrt{\kappa}\, dB_t + \Psi_t'(\theta_t) b(\theta_t; \Psi_t(\theta_1), \ldots, \Psi_t(\theta_n))\, dt + \left( \frac{\kappa}{2} - 3 \right) \Psi_t''(\theta_t)\, dt \\
&= \Psi_t'(\theta_t) \sqrt{\kappa}\, dB_t + \Psi_t'(\theta_t) b(\theta_t; \Psi_t(\theta_1), \ldots, \Psi_t(\theta_n))\, dt + \frac{\kappa - 6}{2} \Psi_t''(\theta_t)\, dt.
\end{align*}

Now, we reparameterize time via \( s(t) = \int_0^t |\Psi_u'(\theta_u)|^2 du \). Under this change of time, we obtain the transformed SDE for $\widetilde{\theta}_s$ by dividing all drift and diffusion terms by $|\Psi_t'(\theta_t)|$:

\[
d\widetilde{\theta}_s = \sqrt{\kappa}\, dB_s + \frac{b(\theta_s; \Psi_{t(s)}(\theta_1), \ldots, \Psi_{t(s)}(\theta_n))}{\Psi_{t(s)}'(\theta_s)}\, ds + \frac{\kappa - 6}{2} \cdot \frac{\Psi_{t(s)}''(\theta_s)}{[\Psi_{t(s)}'(\theta_s)]^2} \, ds.
\]
\end{proof}

\begin{remark}\label{drift term pre schwarz form}
By Theorem~\ref{random coordinate change}, under a conformal coordinate change \( \tau \), the drift term in the marginal law transforms as a \emph{pre-Schwarzian form}. Specifically, if the driving function satisfies
\[
d\theta_t = \sqrt{\kappa} \, dB_t + b(\theta_t) \, dt,
\]
then the drift \( b \) transforms under \( \tau \) as
\[
b(\theta) = \tau'(\theta) \cdot \widetilde{b}(\tau(\theta)) + \frac{6 - \kappa}{2} \cdot \left( \log \tau'(\theta) \right)'.
\]
Here, \( \widetilde{b} \) is the drift in the image coordinate \( \widetilde{\theta} = \tau(\theta) \), and the second term is the pre-Schwarzian derivative of \( \tau \).
\end{remark}

\begin{cor}\label{gamma 1 gamma2 capacity change}
Let $\gamma$, $\tilde{\gamma}$ be two hulls starting at $e^{ix} \in \partial \mathbb{D}$ and $e^{iy} \in \partial \mathbb{D}$ with capacity $\epsilon$ and $c \epsilon$ , let $g_{\epsilon}$ be the normalized map removing $\gamma$ and $\tilde{\epsilon}= \operatorname{hcap}(g_{\epsilon}\circ \gamma(t))$, then we have:

\begin{equation}
\tilde{\varepsilon}=
c \varepsilon\left(1-\frac{\varepsilon}{\sin^2(\frac{x-y}{2})}\right)+o\left(\varepsilon^2\right).
\end{equation}

\end{cor}

\begin{proof}
Locally, we can define $h_t(z)=-i\log(g_t(e^{iz}))$.
Then from the Loewner equation, $\partial_t h_t^{\prime}(w)=- \frac{h_t^{\prime}(w)}{2\sin^2\left(\frac{h_t(w)-x_t}{2}\right)}$, which implies $h_{\varepsilon}^{\prime}(y)=1-\frac{\varepsilon}{2\sin^2(\frac{y-x}{2})}+o(\varepsilon)$. By applying conformal transformation $h_{\epsilon}(y)$, we get
$$
\tilde{\varepsilon}=c\epsilon( h'_{\epsilon}(y)^2+ o(\epsilon)) =c \varepsilon\left(1-\frac{\varepsilon}{\sin^2(\frac{x-y}{2})}\right)+o\left(\varepsilon^2\right).
$$

\end{proof}

\subsection{Local commutation relation and null vector equations in $\kappa>0$ case} \label{reparametrization symmetry}

In this section, we explore how the commutation relations (reparametrization symmetry) and conformal invariance impose constraints on the drift terms $b_j$ and equivalently impose constraints on the partition function $\psi$ derived from $b_j$ as a consequence of commutation relation.

The pioneering work on commutation relations was done in \cite{Dub07}. The author studied the commutation relations for multiple SLEs in the upper half plane $\mathbb{H}$ with
$n$ growth points $z_1,z_2,\ldots,z_n \in \mathbb{R}$ and $m$ additional marked points $u_1,u_2,\ldots,u_m \in \mathbb{R}$.

We extend this Dubedat's commutation argument to the case where there are $n$ growth points $z_1,z_2,\ldots,z_n \in \partial\mathbb{D}$ and one interior marked point $u\in \mathbb{D}$, see theorem (\ref{commutation_relation_u0}) and (\ref{commutation_relation_u}). 

The commutation relations in the unit disk $\mathbb{D}$ with the marked point $0$ are partially studied in \cite{Dub07, WW24}.

We begin by deriving the properties of multiple radial SLE($\kappa$) systems in the unit disk $\mathbb{D}$ with the marked point \( u=0 \). These systems are characterized by the following results:

\begin{thm}[Commutation Relations for $u=0$]\label{commutation_relation_u0}
In the unit disk $\mathbb{D}$, consider $n$ radial SLEs starting at $e^{i\theta_1}, e^{i\theta_2}, \ldots, e^{i\theta_n} \in \partial \mathbb{D}$, with a marked interior point \( u=0 \).
\begin{itemize}
    \item[(i)] Let the infinitesimal diffusion generators be
    \begin{equation}
     \label{commutation of generators}
    \mathcal{M}_i = \frac{\kappa}{2} \partial_{ii} + b_i(\theta_1, \theta_2, \ldots, \theta_n) \partial_i + \sum_{j \neq i} \cot\left(\frac{\theta_j - \theta_i}{2}\right) \partial_j,
     \end{equation}
    where \( \partial_i = \partial_{\theta_i} \). If the $n$ SLEs locally commute, the associated infinitesimal generators satisfy:
    \[
    [\mathcal{M}_i, \mathcal{M}_j] = \frac{1}{\sin^2\left(\frac{\theta_j - \theta_i}{2}\right)} (\mathcal{M}_j - \mathcal{M}_i).
    \]

    There exists a positive function $\psi(\boldsymbol{\theta})$, defined on $\mathfrak{X}^n(\boldsymbol{\theta})$, such that the drift term satisfies
    \[
    b_i(\boldsymbol{\theta}) = \kappa \partial_i \log \psi,
    \]
    and $\psi$ satisfies the null vector equations:
    \begin{equation}
    \frac{\kappa}{2} \partial_{ii} \psi + \sum_{j \neq i} \cot\left(\frac{\theta_j - \theta_i}{2}\right) \partial_i \psi + \left(1 - \frac{6}{\kappa}\right) \sum_{j \neq i} \frac{1}{4 \sin^2\left(\frac{\theta_j - \theta_i}{2}\right)} \psi - h_j(\theta_j) \psi = 0,
    \end{equation}
    for \( i = 1, 2, \ldots, n \), with undetermined functions $h_j(\theta_j)$.

    \item[(ii)] By analyzing the asymptotic behavior of two adjacent growth points \( \theta_i \) and \( \theta_{i+1} \) (with no marked points between them), we further deduce that \( h_i(\theta) = h_{i+1}(\theta) \). Consequently, if all growth points are consecutive with no marked points between them, there exists a common function $h(\theta)$ such that
    \[
    h(\theta) = h_1(\theta) = \cdots = h_n(\theta).
    \]
\end{itemize}
\end{thm}

\begin{thm}[Conformal Invariance under ${\rm Aut}(\mathbb{D}, 0)$]
For a rotation map $\rho_{\theta}$, the drift term \( b_i(\boldsymbol{\theta}) \) is invariant under \(\rho_\theta\), i.e.,
\[
b_i = \widetilde{b}_i \circ \rho_{\theta}.
\]

\begin{itemize}
    \item[(i)] The function \( h(\theta) \) in the null vector equation (\ref{null vector equation angular coordinate constant h}) is rotation-invariant, and there exists a real constant \( h \) such that
   \begin{equation}
    h(\theta) = h.
   \end{equation}

    \item[(ii)] There exists a real rotation constant \( \omega \) such that, for all \( \theta \in \mathbb{R} \),
    \begin{equation}
    \psi(\theta_1 + \theta, \ldots, \theta_n + \theta) = e^{-\omega \theta} \psi(\theta_1, \ldots, \theta_n).
    \end{equation}
\end{itemize}
\end{thm}

\begin{remark} Combining Theorem (\ref{commutation_relation_u0}) with Theorem (\ref{rotation_invariance}), we conclude that a multiple radial $\mathrm{SLE}(\kappa)$ system with fixed $u=0$ is characterized by a partition function that satisfies the null vector equations (\ref{null vector equation angular coordinate constant h}) with a constant 
$h$ and has a rotation constant 
$\omega$. \end{remark}
\begin{remark}
The drift term \( b_i(\theta_1, \ldots, \theta_n) \) is \( 2\pi \)-periodic and therefore well-defined on the unit circle \( S^1 \). However, for \(\omega \neq 0\), the partition function \(\psi\) is not \( 2\pi \)-periodic, making it multivalued on \( S^1 \) and well-defined only on the real line \(\mathbb{R}\), the universal cover of \( S^1 \).

Thus, for \(\omega \neq 0\), the conformal invariance of partition functions requires the use of the group \(\widetilde{{\rm Aut}}(\mathbb{D}, 0)\) with a group action on \(\mathbb{R}\).
\end{remark}
\bigskip

When the marked point \( u \) is allowed to vary within the unit disk \( \mathbb{D} \), rather than being fixed at \( 0 \), a significant difference arises between multiple radial and chordal SLE(\( \kappa \)) systems in terms of their conformal invariance properties. This difference stems from the presence of the marked point \( u \). While multiple radial SLE(\( \kappa \)) systems are conformally invariant, the partition functions within their corresponding equivalence classes do not necessarily exhibit conformal covariance.

We define two partition functions as \emph{equivalent} if and only if they induce identical multiple radial SLE($\kappa$) systems. Equivalent partition functions differ by multiplicative function $f(u)$.
\begin{equation}
\tilde{\psi}=f(u)\cdot\psi
\end{equation}
where $f(u)$ is an arbitrary positive real smooth function depending on the marked interior point $u$
A simple example that violates conformal covariance is
when $f(u)$ is not conformally covariant. 
However, within each equivalence class, it is still possible to find at least one conformally covariant partition function.

\begin{thm}[Commutation Relations for \( u \in \mathbb{D} \)]\label{commutation_relation_u}
In the unit disk $\mathbb{D}$, let $n$ radial SLEs start at $e^{i\theta_1}, e^{i\theta_2}, \ldots, e^{i\theta_n} \in \partial \mathbb{D}$, with a marked interior point \( u \neq 0 \).

\begin{itemize}
    \item[(i)] For \( u = e^{i v} \), let the infinitesimal diffusion generators be
    \[
    \mathcal{M}_i = \frac{\kappa}{2} \partial_{ii} + b_i(\boldsymbol{\theta}, u) \partial_i + \sum_{j \neq i} \cot\left(\frac{\theta_j - \theta_i}{2}\right) \partial_j + \cot\left(\frac{v - \theta_i}{2}\right) \partial_v + \cot\left(\frac{\overline{v} - \theta_i}{2}\right) \overline{\partial}_v.
    \]
    If the $n$ SLEs locally commute, the generators satisfy:
    \[
    [\mathcal{M}_i, \mathcal{M}_j] = \frac{1}{\sin^2\left(\frac{\theta_j - \theta_i}{2}\right)} (\mathcal{M}_j - \mathcal{M}_i).
    \]
    There exists a partition function \(\psi(\boldsymbol{\theta}, u)\) such that the drift term \( b_i(\boldsymbol{\theta}, u) \) is given by
    \[
    b_i(\boldsymbol{\theta}, u) = \kappa \partial_i \log \psi,
    \]
    and \(\psi\) satisfies the null vector equations:
    \[
    \frac{\kappa}{2} \partial_{ii} \psi + \sum_{j \neq i} \frac{2}{\theta_j - \theta_i} \partial_i \psi + \frac{2}{v - \theta_i} \partial_v \psi + \frac{2}{\overline{v} - \theta_i} \overline{\partial}_v \psi + \left[\left(1 - \frac{6}{\kappa}\right) \sum_{j \neq i} \frac{1}{(\theta_j - \theta_i)^2} + h_i(\theta_i, u)\right] \psi = 0.
    \]

    \item[(ii)] By analyzing the asymptotics of adjacent points \( \theta_i \) and \( \theta_{i+1} \), we deduce that \( h_i(\theta, u) = h_{i+1}(\theta, u) \). If all points are consecutive, there exists a common function \( h(\theta, u) \) such that
    \[
    h(\theta, u) = h_1(\theta, u) = \cdots = h_n(\theta, u).
    \]

\end{itemize}
\end{thm}

Now, we discuss how \( \mathrm{Aut}(\mathbb{D}) \)-invariance imposes constraints on the drift terms of a multiple radial SLE($\kappa$) system and how to choose a conformally covariant partition function representative within its equivalence class.

\begin{defn}
The conformal group \( \mathrm{Aut}(\mathbb{D}) \) satisfies the following properties, see \cite{S85}:
\begin{itemize}
    \item \( \mathrm{Aut}(\mathbb{D}) \) is isomorphic to \( PSL_2(\mathbb{R}) \). Each element \( \tau \in \mathrm{Aut}(\mathbb{D}) \) can be written as:
    \[
    \tau(z) = T_v \circ \rho_{\theta}(z),
    \]
    where \( \rho_\theta(z) = e^{i\theta} z \) and \( T_v(z) = \frac{z - v}{1 - \overline{v} z} \).

    Geometrically, \( \mathrm{Aut}(\mathbb{D}) \) is an \( S^1 \)-bundle over \( \mathbb{H} \), and it naturally acts on \( \partial \mathbb{D} \cong S^1 \) by extending the conformal maps to the boundary.

    \item The universal cover \( \widetilde{\mathrm{Aut}}(\mathbb{D}) \) is isomorphic to \( \widetilde{SL}_2(\mathbb{R}) \). Each element \( \tau \in \widetilde{\mathrm{Aut}}(\mathbb{D}) \) can be decomposed as:
    \[
    \tau = T_v \circ A_\theta,
    \]
    where \( \theta \in \mathbb{R} \), and \( A_\theta \) represents addition by \( \theta \) on \( \mathbb{R} \).

    Geometrically, \( \widetilde{\mathrm{Aut}}(\mathbb{D}) \) is an \( \mathbb{R} \)-bundle over \( \mathbb{H} \), and it naturally acts on \( \mathbb{R} \), the universal cover of \( S^1 \).

    \begin{itemize}
        \item For \( x \in \mathbb{R} \), \( A_\theta(x) = x + \theta \).
        \item For \( v \in \mathbb{D} \), there exists a unique \( |y - x| < \pi \) such that \( e^{iy} = T_v(e^{ix}) \).
    \end{itemize}
\end{itemize}
\end{defn}

\begin{thm}\label{conformal_invariance_aut_D}
Let \( \tau \in \mathrm{Aut}(\mathbb{D}) \). The drift term \( b(\boldsymbol{\theta}, u) \) is a pre-Schwarzian form, satisfying:
\[
b_i = \tau' \widetilde{b}_i \circ \tau + \frac{6 - \kappa}{2} \left(\log \tau'\right)'.
\]
\begin{itemize}
    \item[(i)] There exists a smooth function \( F(\tau, u): \widetilde{\mathrm{Aut}}(\mathbb{D}) \times \mathbb{D} \to \mathbb{R} \) such that:
    \[
    \log \psi - \log (\psi \circ \tau) + \frac{\kappa - 6}{2\kappa} \sum_i \log \tau'(\theta_i) = F(\tau, u),
    \]
    where \( F \) satisfies the functional equation:
    
    \begin{equation}
    \label{functional equation for F}
    F(\tau_1 \tau_2, u) = F(\tau_1, \tau_2(u)) + F(\tau_2, u).
    \end{equation}

    \item[(ii)] There exists a rotation constant \( \omega \) such that:
    \begin{equation}
    F(A_\theta, 0) = \omega \theta.
    \end{equation}

    If \( \omega = 0 \), \( \mathrm{Aut}(\mathbb{D}) \) suffices to describe conformal invariance, and \( F \) reduces to a map \( \mathrm{Aut}(\mathbb{D}) \times \mathbb{D} \to \mathbb{R} \).

    \item[(iii)] Suppose \( F_1(\tau, u) \) and \( F_2(\tau, u) \) correspond to partition functions \( \psi_1 \) and \( \psi_2 \). If their rotation constants \( \omega_1 = \omega_2 \), then there exists a function \( g(u) \) such that:
    \[
   \psi_2 = g(u) \cdot \psi_1.
    \]

    \item[(iv)] For \( \tau \in \widetilde{\mathrm{Aut}}(\mathbb{D}) \), let \( \tau(u) = v \). Decompose \( \tau = T_v \circ A_\theta \circ T^{-1}_u \), where \( u, v \in \mathbb{D} \) and \( \theta \in \mathbb{R} \). Using the relations \( T_u'(u) = \frac{1}{1 - |u|^2} \) and \( T_v'(0) = 1 - |v|^2 \), define:
    \[
    \tau'(u)^{\lambda(u)} \overline{\tau'(u)}^{\overline{\lambda(u)}} := \left(\frac{1 - |v|^2}{1 - |u|^2}\right)^{2 \operatorname{Re}(\lambda(u))} e^{-2 \theta \operatorname{Im}(\lambda(u))}.
    \]
    Then:
    \[
    F(\tau, u) = \log \left(\tau'(u)^{\lambda(u)} \overline{\tau'(u)}^{\overline{\lambda(u)}}\right),
    \]
    satisfy the functional equation (\ref{functional equation for F}), with rotation constant \( \omega = \operatorname{Im}(\lambda(u)) \).
\end{itemize}
\end{thm}

\begin{thm}\label{Existence of conformal covariant partition function}
For a multiple radial SLE($\kappa$) system with \( n \) SLEs starting at \( (\theta_1, \theta_2, \ldots, \theta_n) \in \mathfrak{X}^n(\boldsymbol{\theta}) \) and a marked point \( u \in \mathbb{D} \):
\begin{itemize}
       \item[(i)] Two partition functions \(\widetilde{\psi}\) and \(\psi\) are \emph{equivalent} if they differ by a multiplicative factor \( f(u) \):
    \[
    \widetilde{\psi} = f(u) \cdot \psi,
    \]
    where \( f(u) \) is a smooth, positive function of \( u \). Under this equivalence, \(\widetilde{\psi}\) and \(\psi\) induce identical multiple radial SLE($\kappa$) systems.
    \item[(ii)] Within the equivalence class of partition functions, we can choose \( \psi \) to satisfy conformal covariance. Under \( \tau \in \mathrm{Aut}(\mathbb{D}) \), \( \psi \) transforms as:
    \[
    \psi(\theta_1, \ldots, \theta_n, u) = \left(\prod_{i=1}^n \tau'(\theta_i)^{\frac{6-\kappa}{2\kappa}}\right) \tau'(u)^{\lambda(u)} \overline{\tau'(u)}^{\overline{\lambda(u)}} \psi(\tau(\theta_1), \ldots, \tau(\theta_n), \tau(u)).
    \]

    \item[(iii)] The choice of a conformally covariant partition function is not unique. Let:
    \[
    f(u) = (\mathrm{Rad}(u, \mathbb{H}))^\alpha = (i (\overline{u} - u))^\alpha, \quad \alpha \in \mathbb{R}.
    \]
    Then for any conformally covariant \( \psi \), \( \widetilde{\psi} = f(u) \cdot \psi \) yields an equivalent solution with:
    \[
    \widetilde{\lambda}(u) = \lambda(u) + \alpha.
    \]
\end{itemize}
\end{thm}

\begin{remark} Combining Theorem (\ref{commutation_relation_u}) with Theorem (\ref{conformal_invariance_aut_D}), we show that a multiple radial $\mathrm{SLE}(\kappa)$ system with 
$u \in \mathbb{D}$ is described by a conformally covariant partition function that satisfies the null vector equations (\ref{null vector equation angular coordinate constant h}) with a constant $h$. The partition function has a rotation constant $\omega$ and a non-unique conformal dimension $\lambda(u)$, with $\omega = \Im(\lambda(u))$. Moreover, two distinct conformally covariant solutions differ by a multiplicative factor corresponding to a power of the conformal radius. \end{remark}

\begin{proof}[Proof of theorem (\ref{commutation_relation_u0}) and theorem (\ref{commutation_relation_u})]
\
The derivations of the commutation relations for \( u = 0 \) and arbitrary $u \neq 0$ are similar. We mainly discuss the case $u\neq 0$. The proof for \( u = 0 \) can be obtained by simply ignoring the \( u \)-dependence and related derivatives in the drift and diffusion generators since $u=0$ is fixed by the Loewner flow.
\begin{itemize}

\item[(i)]
We first focus on the growth of two hulls from a specific pair of growth points. 
Consider the following scenario: we grow two hulls from  $e^{ix}$ and $e^{iy}$ on the boundary $\partial \mathbb{D}$ and relabel the remaining growth points as $e^{iz_j}$ the marked point $u = e^{iv}$.
\begin{lemma} \label{two point commutation relation in angular}
 In the angular coordinate, suppose two radial SLE hulls start from $x,y \in \mathbb{R}$ with marked points $z_1,z_2,\ldots,z_n \in \mathbb{R}$ and marked interior point $u \in \mathbb{D}$. If $u=0$ is a marked point, we simply omit it since $u=0$ is fixed by the Loewner flow. Let $\mathcal{M}_1$ and $\mathcal{M}_2$ 
$$
\begin{gathered}
\mathcal{M}_1=\frac{\kappa}{2} \partial_{x x}+b(x, y \ldots) \partial_x+\cot(\frac{y-x}{2}) \partial_y+\sum_{i=1}^n \cot(\frac{z_i-x}{2}) \partial_i+ \cot\left(\frac{v-x}{2}\right) \partial_v+\cot\left(\frac{\overline{v}-x }{2}\right)\overline{\partial}_{v}\\
\mathcal{M}_2=\frac{\kappa}{2} \partial_{yy}+\tilde{b}(x, y, \ldots) \partial_y+\cot(\frac{x-y}{2}) \partial_x+\sum_{i=1}^n \cot(\frac{z_i-y}{2}) \partial_i+\cot\left(\frac{v-y}{2}\right) \partial_v+\cot\left(\frac{\overline{v}-y}{2}\right)\overline{\partial}_{v}
\end{gathered}
$$
be the infinitesimal diffusion generators of two SLE hulls,
where $\partial_i=\partial_{z_i}$.

If two SLEs locally commute, then the associated infinitesimal generators satisfy:
\begin{equation}\label{commutation of diffusion generator for two SLE in D}
[\mathcal{M}_1, \mathcal{M}_2]=\frac{1}{\sin^2(\frac{y-x}{2})}(\mathcal{M}_2-\mathcal{M}_1)
\end{equation}
Moreover, there exists a smooth positive function $\mathcal{\psi}(x,y,\boldsymbol{z},u)$ such that:
$$
b=\kappa \frac{\partial_x \psi}{\psi}, \tilde{b}=\kappa \frac{\partial_y \psi}{\psi}
$$

and $\psi$ satisfies the null vector equations:
\begin{equation}
\left\{\begin{array}{l}
\frac{\kappa}{2} \partial_{x x} \psi+\sum_i \cot(\frac{z_i-x}{2})\partial_i \psi+\cot(\frac{y-x}{2})\partial_y \psi+\cot\left(\frac{v-x}{2}\right) \partial_v\psi+\cot\left(\frac{\overline{v}-x }{2}\right)\overline{\partial}_{v}\psi \\
+\left( \left(1-\frac{6}{\kappa}\right) \frac{1}{4\sin^2(\frac{y-x}{2})}+h_1(x, z)\right) \psi=0 \\
\frac{\kappa}{2} \partial_{y y} \psi+\sum_i \cot(\frac{z_i-y}{2}) \partial_i \psi+\cot(\frac{x-y}{2}) \partial_x \psi+ \cot\left(\frac{v-y}{2}\right) \partial_v\psi+\cot\left(\frac{\overline{v}-y }{2}\right)\overline{\partial}_{v}\psi
\\
+\left(\left(1-\frac{6}{\kappa}\right) \frac{1}{4\sin^2(\frac{x-y}{2})}+h_2(y, z)\right) \psi=0
\end{array}\right.
\end{equation}

\end{lemma}

\begin{proof}
Consider a Loewner chain $\left(K_{s, t}\right)_{(s, t) \in \mathcal{T}}$ with a double time index. The associated conformal equivalence are $g_{s,t}$. We also assume that $K_{s, t}=K_{s, 0} \cup K_{0, t}$.  If $s \leq s^{\prime}, t \leq t^{\prime},\left(s^{\prime}, t^{\prime}\right) \in \mathcal{T}$, then $(s, t) \in \mathcal{T}$. 

Let $\sigma($ resp. $\tau)$ be a stopping time in the filtration generated by $\left(K_{s, 0}\right)_{(s, 0) \in \mathcal{T}}\left(\right.$ resp. $\left.\left(K_{0, t}\right)_{(0, t) \in \mathcal{T}}\right)$. Let also $\mathcal{T}^{\prime}=\{(s, t):(s+\sigma, t+\tau) \in \mathcal{T}\}$ and $\left(K_{s, t}^{\prime}\right)_{(s, t) \in \mathcal{T}^{\prime}}=\left(\overline{g_{\sigma, \tau}\left(K_{s+\sigma, t+\tau} \backslash K_{s, t}\right)}\right)$. Then $\left(K_{s, 0}^{\prime}\right)_{(s, 0) \in \mathcal{T}^{\prime}}$ is distributed as a stopped $\operatorname{SLE}_\kappa(b)$, i.e an SLE driven by:

$$
d x_s=\sqrt{\kappa} d B_s+b\left(x_s, h_s(y), \ldots, h_s\left(z_i\right), \ldots,e^{ih_s(v)}\right) d t
$$
where $h_s$ is the covering map of Loewner map $g_s$, (i.e. $e^{ih_s(z)}=g_s(e^{iz})$).

Likewise $\left(K_{0, t}^{\prime}\right)_{(0, t) \in \mathcal{T}^{\prime}}$ is distributed as a stopped $ \operatorname{SLE}_{\tilde{\kappa}}(\tilde{b})$ , i.e an SLE driven by: 
$$
d y_t=\sqrt{\kappa} d \tilde{B}_t+\tilde{b}\left(\tilde{h}_t(x), y_t, \ldots, \tilde{h}_t\left(z_i\right), \ldots,e^{i\tilde{h}_t(v)}\right) d t
$$
where $\tilde{h}_t$ is the covering map of $\tilde{g}_t$ (i.e. $e^{i\tilde{h}_s(z)}=\tilde{h}_s(e^{iz})$).

Here $B, \tilde{B}$ are standard Brownian motions, $\left(g_s\right),\left(\tilde{g}_t\right)$ are the associated conformal equivalences, $b, \tilde{b}$ are some smooth, translation invariant functions.

Note that

$$
\left(x_s, h_s(y), \ldots, h_s\left(z_i\right), \ldots, \operatorname{Re}\left(h_s\left(v\right)\right), \operatorname{Im}\left(h_s\left(v\right)\right)\right)
$$

is a Markov process with semigroup $P$ and infinitesimal generator $\mathcal{M}_1$. Similarly,

$$
\left(\tilde{h}_t(x), y_t, \ldots, \tilde{h}_t\left(z_i\right), \ldots, \operatorname{Re}\left(\tilde{h}_t\left(v\right)\right), \operatorname{Im}\left(\tilde{h}_t\left(v\right)\right)\right)
$$

is a Markov process with semigroup $Q$ and infinitesimal generator $\mathcal{M}_2$.

We denote $A^x, A^y$ the unit disk capacity of hulls growing at $e^{ix}$ and $e^{iy}$, and consider the stopping time $\sigma=\inf \left(s: A^x\left(K_{s, 0}\right) \geq a^x\right), \tau=\inf \left(t: A^y\left(K_{0, t}\right) \geq a^y\right)$, where $a^x=\varepsilon, a^y=c \varepsilon$.

We are interested in the SLE hull $K_{\sigma, \tau}$. Two ways of getting from $K_{0,0}$ to $K_{\sigma, \tau}$ are:

\begin{itemize}
    \item run the first SLE (i.e. $SLE_{\kappa}(b)$), started from $\left(x, y, \ldots, z_i, \ldots\right.$ ) until it reaches capacity $\varepsilon$.
    \item then run independently the second SLE (i.e. $SLE_{\kappa} (\tilde{b}))$ in $g_s^{-1} (\mathbb{D})$ until it reaches capacity $c \varepsilon$; this capacity is measured in the original unit disk. Let $\tilde{h}_{\tilde{\epsilon}} $ be the corresponding conformal equivalence.
    \item one gets two hulls resp. at $x$ and $y$ with capacity $\epsilon$ and $c \epsilon$ ; let $\phi=\tilde{h}_{\tilde{\epsilon}}  \circ h_{\epsilon}$ be the normalized map removing these two hulls.
    \item expand $\mathbb{E}\left(F\left(\tilde{h}_{\tilde{\epsilon}}\left(X_{\epsilon}\right), \tilde{Y}_{\tilde{\epsilon}}\right)\right)$ up to order two in $\epsilon$.
This describes (in distribution) how to get from $K_{0,0}$ to $K_{\sigma, 0}$, and then from $K_{\sigma, 0}$ to $K_{\sigma, \tau}$.
\end{itemize}

Note that under the conformal map $h_{\epsilon}$, the capacity of $\tilde{\gamma}$ is not $c\epsilon$.
 According to lemma (\ref{gamma 1 gamma2 capacity change}), the capacity is given by:
 \begin{equation}
\tilde{\varepsilon}=
c \varepsilon\left(1-\frac{\varepsilon}{\sin^2(\frac{x-y}{2})}\right)+o\left(\varepsilon^2\right)
\end{equation}

Now, let  $F$  be a test function  $\mathbb{R}^{n+2 2} \rightarrow \mathbb{R}$ , and  $c>0$ be some constant and let 
\begin{equation}
\begin{aligned}
& w=\left(x, y, \ldots, z_i, \ldots, v \right) \\
& w^{\prime}=\left(X_{\varepsilon}, g_{\varepsilon}(y), \ldots g_{\varepsilon}\left(z_i\right), \ldots, g_{\varepsilon}\left(v\right)\right) \\
& w^{\prime \prime}=\left(\tilde{g}_{\tilde{\varepsilon}}\left(X_{\varepsilon}\right), \tilde{Y}_{\tilde{\varepsilon}}, \ldots \tilde{g}_{\tilde{\varepsilon}} \circ g_{\varepsilon}\left(z_i\right), \ldots, \tilde{g}_{\tilde{\varepsilon}} \circ g_{\varepsilon}\left(v\right)\right)
\end{aligned}
\end{equation}

We consider the conditional expectation of $F(w'')$ with respect to $w$.

$$
\begin{aligned}
\mathbb{E}\left(F\left(w^{\prime \prime}\right) \mid w\right) & =\mathbb{E}\left(F\left(w^{\prime \prime}\right)\left|w^{\prime}\right| w\right)=P_{\varepsilon} \mathbb{E}\left(Q_{\tilde{\varepsilon}} F \mid w^{\prime}\right)(w) \\
& =P_{\varepsilon} \mathbb{E}\left(\left(1+\varepsilon \mathcal{M}_1+\frac{\varepsilon^2}{2} \mathcal{M}_{1}^2\right) F\left(w^{\prime}\right)\right)(w)=P_{\varepsilon} Q_{c \varepsilon\left(1-\frac{\varepsilon}{\sin^2(\frac{x-y}{2})}\right)} F(w)+o\left(\varepsilon^2\right) \\
& =\left(1+\varepsilon \mathcal{M}_1+\frac{\varepsilon^2}{2} \mathcal{M}_{1}^2\right)\left(1+c \varepsilon\left(1-\frac{\varepsilon}{\sin^2(\frac{x-y}{2})}\right) \mathcal{M}_2+\frac{c^2 \varepsilon^2}{2} \mathcal{M}_2^2\right) F(w)+o\left(\varepsilon^2\right) \\
& =\left(1+\varepsilon(\mathcal{M}_1+c \mathcal{M}_2)+\varepsilon^2\left(\frac{1}{2} \mathcal{M}_{1}^2+\frac{c^2}{2} \mathcal{M}_{2}^2+c \mathcal{M}_1 \mathcal{M}_{2}-\frac{c}{\sin^2(\frac{x-y}{2})}\mathcal{M}_{2}\right)\right) F(w)+o\left(\varepsilon^2\right)
\end{aligned}
$$
If we first grow a hull at $y$, then at $x$, one gets instead:
$$
\left(1+\varepsilon(\mathcal{M}_1+c \mathcal{M}_2)+\varepsilon^2\left(\frac{1}{2} \mathcal{M}_{1}^2+\frac{c^2}{2} \mathcal{M}_{2}^2+c \mathcal{M}_{2} \mathcal{M}_{1}-\frac{c}{\sin^2(\frac{x-y}{2})} \mathcal{M}_{1}\right)\right) F(w)+o\left(\varepsilon^2\right)
$$
Hence, the commutation condition reads:
\begin{equation}
[\mathcal{M}_1, \mathcal{M}_2]=\frac{1}{\sin^2(\frac{y-x}{2})}(\mathcal{M}_2-\mathcal{M}_1)
\end{equation}

After simplifications, one gets:
$$
\begin{aligned}
&[\mathcal{M}_1, \mathcal{M}_2]+\frac{1}{\sin^2(\frac{y-x}{2})}(\mathcal{M}_1-\mathcal{M}_2)  =\left(\kappa \partial_x \tilde{b}-\kappa \partial_y b\right) \partial_{x y} \\
& +\left[\cot(\frac{y-x}{2})\partial_x b+\sum_i \cot(\frac{y-z_i}{2}) \partial_i b+\cot\left(\frac{v-x}{2}\right) \partial_v b+\cot\left(\frac{\overline{v}-x}{2}\right)\overline{\partial}_{v}b \right. \\
&\left.-\tilde{b} \partial_y b+\frac{ b}{2\sin^2(\frac{y-x}{2})}+\frac{\cos(\frac{x-y}{2})}{4\sin^3(\frac{x-y}{2})}(\kappa-6)-\frac{\kappa}{2} \partial_{y y} b\right] \partial_x \\
& -\left[\cot(\frac{x-y}{2})\partial_y \tilde{b}+\sum_i \cot(\frac{x-z_i}{2}) \partial_i \tilde{b}-\cot\left(\frac{v-x}{2}\right) \partial_v\tilde{b}+\cot\left(\frac{\overline{v}-x }{2}\right)\overline{\partial}_{v}\tilde{b} \right.\\
&\left.-b \partial_x \tilde{b}+\frac{\tilde{ b}}{2\sin^2(\frac{y-x}{2})}+\frac{\cos(\frac{y-x}{2})}{4\sin^3(\frac{y-x}{2})}(\kappa-6)-\frac{\kappa}{2} \partial_{x x}\tilde{b}\right] \partial_y
\end{aligned}
$$

So, the commutation condition reduces to three differential conditions involving $b$ and $\tilde{b}$.

\begin{equation}
\left\{\begin{aligned}
&\kappa \partial_x \tilde{b}-\kappa \partial_y b=0 \\
&\cot(\frac{y-x}{2})\partial_x b+\sum_i \cot(\frac{y-z_i}{2}) \partial_i b+\cot\left(\frac{v-x}{2}\right) \partial_v b+\cot\left(\frac{\overline{v}-x}{2}\right)\overline{\partial}_{v}b\\
&-\tilde{b} \partial_y b+\frac{ b}{2\sin^2(\frac{y-x}{2})}+\frac{\cos(\frac{x-y}{2})}{4\sin^3(\frac{x-y}{2})}(\kappa-6)-\frac{\kappa}{2} \partial_{y y} b=0 \\
&\cot(\frac{x-y}{2})\partial_y \tilde{b}+\sum_i \cot(\frac{x-z_i}{2}) \partial_i \tilde{b}\\
&-b \partial_x \tilde{b}+\frac{\tilde{ b}}{2\sin^2(\frac{y-x}{2})}+\frac{\cos(\frac{y-x}{2})}{4\sin^3(\frac{y-x}{2})}(\kappa-6)-\frac{\kappa}{2} \partial_{x x}\tilde{b}=0
\end{aligned}\right.
\end{equation}

Now, from the first equation, one can write:
$$
b=\kappa \frac{\partial_x \psi}{\psi}, \tilde{b}=\kappa \frac{\partial_y \psi}{\psi}
$$
for some non-vanishing function $\psi$ (at least locally).
It turns out that the second equation now writes:
$$
\kappa \partial_x\left(\frac{\frac{\kappa}{2} \partial_{y y} \psi+\sum_i \cot(\frac{z_i-y}{2}) \partial_i \psi+\cot(\frac{x-y}{2}) \partial_x \psi+\cot\left(\frac{v-y}{2}\right) \partial_v \psi+\cot\left(\frac{\overline{v}-y}{2}\right)\overline{\partial}_{v}\psi+\left(1-\frac{6}{\kappa}\right) \frac{\psi}{4\sin^2(\frac{x-y}{2})}}{\psi}\right)=0 .
$$

Symmetrically, the last equation is:
$$
\kappa \partial_y\left(\frac{\frac{\kappa}{2} \partial_{x x} \psi+\sum_i \cot(\frac{z_i-x}{2})\partial_i \psi+\cot(\frac{y-x}{2})\partial_y \psi+\cot\left(\frac{v-x}{2}\right) \partial_v \psi+\cot\left(\frac{\overline{v}-x}{2}\right)\overline{\partial}_{v}\psi+ \left(1-\frac{6}{\kappa}\right) \frac{\psi}{4\sin^2(\frac{y-x}{2})}}{\psi}\right)=0 .
$$

Now, from the first question, one can write:
$$
b=\kappa \frac{\partial_x \psi}{\psi}, \tilde{b}=\kappa \frac{\partial_y \psi}{\psi}
$$
for some non-vanishing function $\psi$.
It turns out that two equations now write:

\begin{equation}
\left\{\begin{aligned}
&\frac{\kappa}{2} \partial_{x x} \psi+\sum_i \cot(\frac{z_i-x}{2})\partial_i \psi+\cot(\frac{y-x}{2})\partial_y \psi+\cot\left(\frac{v-x}{2}\right) \partial_v \psi+\cot\left(\frac{\overline{v}-x}{2}\right)\overline{\partial}_{v}\psi \\
&+\left( \left(1-\frac{6}{\kappa}\right) \frac{1}{4\sin^2(\frac{y-x}{2})}+h_1(x, \boldsymbol{z},u)\right) \psi=0 \\
&\frac{\kappa}{2} \partial_{y y} \psi+\sum_i \cot(\frac{z_i-y}{2}) \partial_i \psi+\cot(\frac{x-y}{2}) \partial_x \psi+\cot\left(\frac{v-x}{2}\right) \partial_v \psi+\cot\left(\frac{\overline{v}-x}{2}\right)\overline{\partial}_{v}\psi \\
&+\left(\left(1-\frac{6}{\kappa}\right) \frac{1}{4\sin^2(\frac{x-y}{2})}+h_2(y, \boldsymbol{z},u)\right) \psi=0
\end{aligned}\right.
\end{equation}

\end{proof}

Let us now begin our discussion on the multiple radial SLE($\kappa$) systems with $n$ distinct growth points $e^{i\theta_1},e^{i\theta_2},\ldots,e^{i\theta_n}$.
We want to grow $n$ infinitesimal hulls at $e^{i \theta_i}$, $i=1,2,\ldots,n$. We can either grow a hull $K_{\varepsilon_i}$ at $e^{i\theta_i}$, and then another one at $e^{i\theta_j}$ in the perturbed domain $\mathbb{D} \backslash K_{\varepsilon_i}$, or proceed in any order. The coherence condition is that these procedures yield the same result.

We grow two SLE hulls from $\theta_i,\theta_j$, $i\neq j$ and treat the rest as marked points. 
By lemma (\ref{two point commutation relation in angular}),
the commutation relation between two SLEs implies that the infinitesimal generator satisfies
\begin{equation}
[\mathcal{M}_i, \mathcal{M}_j]=\frac{1}{\sin^2(\frac{\theta_i-\theta_j}{2})}(\mathcal{M}_j-\mathcal{M}_i)
\end{equation}

By expanding (\ref{commutation of diffusion generator for two SLE in D}), we derive that
\begin{equation}\label{close condition}
    \kappa \partial_i b_j-\kappa \partial_j b_i=0
\end{equation}
for all $1\leq i<j \leq n$.

Since the chamber $$\mathfrak{X}^n \times \mathbb{D}=\left\{(\theta_1,\theta_2,\ldots,\theta_n,u) \in \mathbb{R}^n \times \mathbb{D} \mid \theta_1<\theta_2<\ldots<\theta_n<\theta_1+2\pi , u\in \mathbb{D}\right\}$$ 
is simply connected (contractible). Equations (\ref{close condition}) imply that we can integrate the differential form  $\sum_j b_j(\boldsymbol{\theta},u)d\theta_j$ with respect to $\theta_1,\theta_2,\ldots,\theta_n$. Stoke's theorem implies that this integral is path-independent. Consequently, there exists a positive function  $\psi(\boldsymbol{\theta},u)$ such that:
\begin{equation}
b_i(\boldsymbol{\theta},u)=\kappa\frac{\partial_i \psi}{\psi}
\end{equation}

and the null vector equations 
\begin{equation} \label{trigonometric integrability}
\left\{\begin{aligned}
&\frac{\kappa}{2} \partial_{ii} \psi+\sum_{k \neq i,j} \cot(\frac{\theta_k-\theta_i}{2})\partial_k \psi+\cot(\frac{\theta_j-\theta_i}{2})\partial_j \psi+\cot\left(\frac{v-\theta_i}{2}\right) \partial_v \psi+\cot\left(\frac{\overline{v}-\theta_i}{2}\right)\overline{\partial}_{v}\psi \\
&+\left( \left(1-\frac{6}{\kappa}\right) \frac{1}{4\sin^2(\frac{\theta_j-\theta_i}{2})}+h_i(\boldsymbol{\theta},u)\right) \psi=0 \\
&\frac{\kappa}{2} \partial_{jj} \psi+\sum_{k \neq i,j} \cot(\frac{\theta_k-\theta_j}{2}) \partial_k \psi+\cot(\frac{\theta_i-\theta_j}{2}) \partial_i \psi+\cot\left(\frac{v-\theta_i}{2}\right) \partial_v \psi+\cot\left(\frac{\overline{v}-x}{2}\right)\overline{\partial}_{v}\psi \\
&+\left(\left(1-\frac{6}{\kappa}\right) \frac{1}{4\sin^2(\frac{\theta_i-\theta_j}{2})}+h_j(\boldsymbol{\theta},u)\right) \psi=0 \\
\end{aligned}\right.
\end{equation}

We may write the first equation in (\ref{trigonometric integrability}) as
\begin{equation}
\begin{aligned}
 &   \frac{\kappa}{2} \partial_{ii} \psi+\sum_{k\neq i,j}\cot(\frac{\theta_k-\theta_i}{2})\partial_k \psi+\cot(\frac{\theta_j-\theta_i}{2})\partial_j \psi+ \cot\left(\frac{v-\theta_i}{2}\right) \partial_v \psi+\cot\left(\frac{\overline{v}-\theta_i}{2}\right)\overline{\partial}_{v}\psi \\ 
 &=-\left( \left(1-\frac{6}{\kappa}\right) \frac{1}{4\sin^2(\frac{\theta_j-\theta_i}{2})}+h_i(\boldsymbol{\theta},u)\right) \psi
\end{aligned}
\end{equation}
where $h_i$ does not depend on $\theta_j$. Since integrability conditions hold for all $ j \neq i$, by subtracting all $\left(1-\frac{6}{\kappa}\right) \frac{1}{4\sin^2\left(\frac{\theta_j-\theta_i}{2}\right)}$ terms, we obtain that 

\begin{equation} \label{equation of h angular}
\begin{aligned}
 &   \frac{\kappa}{2} \partial_{ii} \psi+\sum_{k\neq i,j}\cot(\frac{\theta_k-\theta_i}{2})\partial_k \psi+\cot(\frac{\theta_j-\theta_i}{2})\partial_j \psi+ \cot\left(\frac{v-\theta_i}{2}\right) \partial_v \psi+\cot\left(\frac{\overline{v}-\theta_i}{2}\right)\overline{\partial}_{v}\psi \\ 
 &=-\left( \left(1-\frac{6}{\kappa}\right) \frac{1}{4\sin^2(\frac{\theta_j-\theta_i}{2})}+h_i(\theta_i,u)\right) \psi
\end{aligned}
\end{equation}

where $h_i=h_i(\theta_i,u)$ only depends on $\theta_i$ and $u$.

\item[(ii)]
\begin{lemma}\label{two point commutation}
For adjacent growth points $x, y \in \mathbb{R}$ (no marked points $\{z_1,z_2,\ldots,z_n\}$  are between $x$ and $y$). If the system
\begin{equation}
\left\{\begin{array}{l}
\frac{\kappa}{2} \partial_{x x} \psi+\sum_i \cot(\frac{z_i-x}{2})\partial_i \psi+\cot(\frac{y-x}{2})\partial_y \psi+\left( \left(1-\frac{6}{\kappa}\right) \frac{1}{4\sin^2(\frac{y-x}{2})}+h_1(x, z)\right) \psi=0 \\
\frac{\kappa}{2} \partial_{y y} \psi+\sum_i \cot(\frac{z_i-y}{2}) \partial_i \psi+\cot(\frac{x-y}{2}) \partial_x \psi+\left(\left(1-\frac{6}{\kappa}\right) \frac{1}{4\sin^2(\frac{x-y}{2})}+h_2(y, z)\right) \psi=0
\end{array}\right.
\end{equation}
admits a non-vanishing solution $\psi$, then: functions $h_1, h_2$ can be written as $h_1(x, z)=h(x, z), h_2(y, z)=h(y, z)$
\end{lemma}

\begin{proof}
The problem is now to find functions $h_1, h_2$ such that the above system has solutions. So assume that we are given $h_1, h_2$, and a non-vanishing solution $\psi$ of this system. Let:
$$
\begin{aligned}
&\mathcal{L}_1=\frac{\kappa}{2} \partial_{x x} +\sum_i \cot(\frac{z_i-x}{2})\partial_i +\cot(\frac{y-x}{2})\partial_y + \left(1-\frac{6}{\kappa}\right) \frac{1}{4\sin^2(\frac{y-x}{2})}\\
&\mathcal{L}_2=\frac{\kappa}{2} \partial_{y y} +\sum_i \cot(\frac{z_i-y}{2}) \partial_i +\cot(\frac{x-y}{2}) \partial_x +\left(1-\frac{6}{\kappa}\right) \frac{1}{4\sin^2(\frac{x-y}{2})}
\end{aligned}
$$
Then $\psi$ is annihilated by all operators in the left ideal generated by $\left(\mathcal{L}_1+h_1\right),\left(\mathcal{L}_2+h_2\right)$, including in particular their commutator:
$$
\begin{aligned}
\mathcal{L} & =\left[\mathcal{L}_1+h_1, \mathcal{L}_2+h_2\right]+\frac{1}{\sin^2(\frac{x-y}{2})} ((\mathcal{L}_1+h_1)-(\mathcal{L}_2+h_2)) \\
& =\left[\mathcal{L}_1, \mathcal{L}_2\right]+\frac{1}{\sin^2(\frac{x-y}{2})}(\mathcal{L}_1-\mathcal{L}_2)+([\mathcal{L}_1, h_2]-[\mathcal{L}_2, h_1])+\frac{(h_1-h_2)}{\sin^2(\frac{x-y}{2})} \\
&=\left(\cot(\frac{y-x}{2}) \partial_y+\sum_i \cot(\frac{z_i-x}{2}) \partial_i+ \cot\left(\frac{v-x}{2}\right) \partial_v \psi+\cot\left(\frac{\overline{v}-x}{2}\right)\overline{\partial}_{v}\psi\right) h_2 \\
&-\left(\cot(\frac{x-y}{2}) \partial_x+\sum_i \cot(\frac{z_1-y}{2}) \partial_i+ \cot\left(\frac{v-y}{2}\right) \partial_v \psi+\cot\left(\frac{\overline{v}-y}{2}\right)\overline{\partial}_{v}\psi\right) h_1\\
&+\frac{4\left(h_1-h_2\right)}{\sin^2(\frac{x-y}{2})}
\end{aligned}
$$

$\mathcal{L}$ is an operator of order 0, it is a function.
Since $\mathcal{L}(\psi)=0$ for a non-vanishing $\psi$,  $\mathcal{L}$ must vanish identically. 

Note that if the two growth points $x$  and $y$ are adjacent (no marked points $\{z_1,z_2,\ldots,z_n\}$ are between $x$ and $y$), we consider the pole of $\mathcal{L}$ at $x=y$.  The second-order pole must  vanish, this implies $h_1(x, z)=h(x, z), h_2(y, z)=h(y, z)$ for a common function $h$.
\end{proof}

By applying lemma (\ref{two point commutation}) to  adjacent $\theta_i$ and $\theta_{i+1}$, we obtain that the function $h_i(\theta,u)=h_{i+1}(\theta,u)$ for each $1\leq i \leq n-1$, which implies the existence of a common function $h(\theta,u)$.

\end{itemize}
\end{proof}

We have already established the commutation relations and now we consider how conformal invariance imposes constraints on the drift term and partition functions.

The first case is the ${\rm Aut}(\mathbb{D},0)$ invariance of the multiple radial SLE($\kappa$) with a marked point $u=0$.
\begin{proof}[Proof of theorem (\ref{rotation_invariance})]
For a multiple radial SLE($\kappa$) system with marked point $u=0$.
Note that by rotation invariance of the drift term $b_i$, under a rotation  $\rho_a$, the functions $b_i(\theta_1,\theta_2,\ldots,\theta_n)$ satisify
$$
b_i(\theta_1,\theta_2,\ldots,\theta_n)=b_i(\theta_1+a,\theta_2+a,\ldots,\theta_n+a)
$$
for $i=1,2,\ldots,n$.
\begin{itemize} 
\item[(i)] By equation (\ref{equation of h angular}), for $u=0$, we simply omit the $u$-dependence and related derivatives we obtain that:

\begin{equation}
\begin{aligned}
   h(\theta_i)&= -\frac{\kappa}{2} \frac{\partial_{ii} \psi}{\psi}-\sum_j \cot(\frac{\theta_j-\theta_i}{2})\frac{\partial_j \psi}{\psi}- \left(1-\frac{6}{\kappa}\right)\sum_j \frac{1}{4\sin^2(\frac{\theta_j-\theta_i}{2})}\\
   &= -\frac{\kappa}{2} (\partial_{i} b_i+ b_{i}^2)-\sum_j \cot(\frac{\theta_j-\theta_i}{2})b_j- \left(1-\frac{6}{\kappa}\right)\sum_j \frac{1}{4\sin^2(\frac{\theta_j-\theta_i}{2})}
\end{aligned}
\end{equation}

The rotation invariance of $b_i(\boldsymbol{\theta})$ implies the rotation invariance of $h(\theta_i)$. Thus, $h$ must be a constant.

\item[(ii)]

Since $b_i = \kappa\partial_i \log(\psi)$, by the rotation invariance of $b_i$, for rotation transformation $\rho_a$:
$$
\partial_i \left(\log(\psi)- \log(\psi \circ \rho_a) \right)=0
$$

for $i=1,2,\ldots,n$. Thus, independent of $\theta_1,\theta_2,\ldots,\theta_n$. We obtain that there exists a function $F(a): \mathbb{R} \rightarrow \mathbb{R}$ such that 
$$
   \log(\psi)-\log(\psi \circ \rho_a)=F(a)
$$
Since for $a,b\in \mathbb{R}$, $F$ satisfies the Cauchy functional equation
$$F(a)+F(b)=F(a+b)$$
The only solution for the Cauchy functional equation is linear. Thus, there exists $\omega \in \mathbb{R}$.
$$F(a)=\omega a$$
By differentiating with respect to $a$, 
$$
\sum_i \partial_i \psi = \omega \psi
$$
\end{itemize}
\end{proof}

\begin{proof}[Proof of theorem (\ref{conformal_invariance_aut_D})]
\

\begin{itemize}

\item[(i)] \label{conformally convariant modification}

Note that by corollary (\ref{drift term pre schwarz form}), under a conformal map $\tau \in {\rm Aut}(\mathbb{D})$, the drift term $b_i(\theta_1,\theta_2,\ldots,\theta_n,u)$ transforms as

$$
b_i=\tau^{\prime}(\theta_i) \left(b_i \circ \tau \right)+ \frac{6-\kappa}{2}\left(\log \tau^{\prime}(\theta_i)\right)^{\prime}
$$
   Since $b_i = \kappa\partial_i \log(\psi)$
$$
\kappa\partial_i \log(\psi)=\kappa\tau^{\prime}(\theta_i) \partial_i \log(\psi \circ \tau) + \frac{6-\kappa}{2}\left(\log \tau^{\prime}(\theta_i)\right)^{\prime}
$$
which implies 
$$
\partial_i\left(\log(\psi)-\log(\psi \circ \tau)+\frac{\kappa-6}{2\kappa}\sum_i \log(\tau'(\theta_i)) \right)=0
$$
for $i=1,2,\ldots,n$. Thus, independent of variables $\theta_1,\theta_2,\ldots,\theta_n$. 

We obtain that there exists a function $F: {\rm Aut}(\mathbb{H})\times \mathbb{H} \rightarrow \mathbb{C}$ such that 
$$
   \log(\psi)-\log(\psi \circ \tau)+\frac{\kappa-6}{2\kappa}\sum_i \log(\tau'(\theta_i)) = F(\tau,u) 
$$
By direct computation and the chain rule, we can show that
$$
\begin{aligned}
F(\tau_1 \tau_2 ,u)&= \log(\psi)-\log(\psi \circ \tau_1 \tau_2)+\frac{\kappa-6}{2\kappa}\sum_i \log((\tau_1 \tau_2)'(\theta_i))  \\
&=\log(\psi)-\log(\psi\circ \tau_2)+\log(\psi\circ \tau_2)-\log(\psi \circ \tau_1 \tau_2) \\
&+\frac{\kappa-6}{2\kappa}\sum_i \log( \tau_{2}'(\theta_i))+\frac{\kappa-6}{2\kappa}\sum_i \log(\tau_1'(\tau_2(\theta_i)))\\
&=F(\tau_1,\tau_2(u))+F(\tau_2,u)
\end{aligned}
$$

\item[(ii)]
By the functional equation (\ref{functional equation for F}) and $u=0$ is the fixed point of the addition transformation $A_{\theta}(z)$, we obtain that
$$F(A_{\theta_1+\theta_2},i)=F(A_{\theta_1},A_{\theta_2}(i))+F(A_{\theta_2},i)=F(A_{\theta_1},i)+F(A_{\theta_2},i)$$
This is a Cauchy functional equation, the only solution is linear, thus there exists real constant $\beta$ such that
$$F(A_{\theta},i)= \beta \theta$$

\item[(iii)]

Let $v = \tau(u)$, let $T_u$ be the conformal map:
$$T_u(z)= \frac{z-u}{1- \overline{u}z}$$ 
then by the functional equation (\ref{functional equation for F}), we obtain that:
$$
F_i(\tau, u)= F_i(T_v \circ A_\theta \circ T^{-1}{u}, T_u(0))
= -F_i(T_u,0) +F_i(T_v \circ A_\theta, 0)
= F_i(T_v,0)-F_i(T_u,0) + \omega_i \theta
$$
for $i=1,2$.
we define $$f(u) = F_1(T_u,0)-F_2(T_u,0)$$

Now, suppose $\psi_i$ are corresponding partition functions.
By the definition of function $F(\tau,u)$, $\psi_i$ satisfies the following functional equation
\begin{equation}
   \log(\psi_i)-\log(\psi_i \circ \tau)+\frac{\kappa-6}{2\kappa}\sum_j \log(\tau'(z_j)) = F_i(\tau,u) 
\end{equation} 
Subtracting two equations, we obtain that
$$
\log(\frac{\psi_1}{\psi_2}) - \log( \frac{\psi_1 \circ \tau}{\psi_2 \circ \tau}) = f(v)-f(u)+ (\omega_1-\omega_2) \theta
$$
Then if $\omega_1=\omega_2$
$$
\log(\frac{\psi_1}{\psi_2}) - \log( \frac{\psi_1 \circ \tau}{\psi_2 \circ \tau}) = f(v)-f(u)
$$
which is equivalent to
$$ \psi_2 =c e^{f(u)} \psi_1 $$
thus $$g(u) = c e^{f(u)} $$ 
where $c>0$.

\item[(iv)]
Now we verify that $F(\tau,u)$ defined in
satisfy the functional equation (\ref{functional equation for F}).

Let $v = \tau_2(u)$, $w = \tau_1 \circ \tau_2(u)$, 
$$\tau_2 = T_v\circ A_{\theta_2} \circ T_{-u}$$
$$\tau_1 = T_w\circ A_{\theta_1} \circ T_{-v}$$
$$\tau_1 \circ \tau_2 = T_w\circ A_{\theta_1+\theta_2} \circ T_{-u}$$
then
$$F(\tau_1 \circ \tau_2,u) = 2{\rm Re}(\lambda(u))\log\left(\frac{1-|w|^2}{1-|u|^2}\right)- 2(\theta_1+\theta_2) {\rm Im}(\lambda(u))
$$
$$
F(\tau_1,\tau_2(u))=F(\tau_1,v)=2{\rm Re}(\lambda(u))\log\left(\frac{1-|w|^2}{1-|v|^2}\right)- 2\theta_1 {\rm Im}(\lambda(u))
$$
$$
F(\tau_2,u)=F(\tau_1,v)=2{\rm Re}(\lambda(u))\log\left(\frac{1-|v|^2}{1-|u|^2}\right)- 2\theta_2 {\rm Im}(\lambda(u))
$$
Combining above three equations, we obtain that
   \begin{equation} 
      F(\tau_1 \tau_2 ,u)= F(\tau_1,\tau_2(u))+F(\tau_2,u)
   \end{equation}

\end{itemize}
\end{proof}

\begin{proof}[Proof of Theorem (\ref{Existence of conformal covariant partition function})]
For a multiple radial SLE($\kappa$) system with partition function \(\psi(\boldsymbol{\theta}, u)\), we proceed as follows:
\begin{itemize}
\item[(i)]
By equation (\ref{close condition}), the drift term in the marginal law for multiple radial SLE($\kappa$) systems is given by 
$$b_i =\kappa \partial_j \log(\psi)$$

If two partition functions differ by 
a multiplicative function $f(u)$.
\begin{equation}
\tilde{\psi}=f(u)\cdot\psi
\end{equation}
where $f(u)$ is an arbitrary positive real smooth function depending on the marked interior point $u$.
Note that
$$b_i =\kappa \partial_j \log(\psi)=\kappa \partial_j \log(\widetilde{\psi})= \widetilde{b_i}$$
Thus $\widetilde{\psi}$ and $\psi$ induce identical multiple radial SLE($\kappa$) system.

\item[(ii)]  
Let \(\omega\) be the corresponding rotation constant. Define:
    \[
    \psi(\theta_1, \theta_2, \ldots, \theta_n, 0) = \left(\prod_{i=1}^{n} T_u'(\theta_i)^{\frac{6-\kappa}{2\kappa}}\right) T_u'(u)^{\lambda(u)} T_u'(u^*)^{\lambda(u^*)} \tilde{\psi}\left(T_u(\theta_1), T_u(\theta_2), \ldots, T_u(\theta_n), u\right)
    \]
    Here, \(\tilde{\psi}\) and \(\psi\) share the same rotation constant \(\omega\). By (iii) of Theorem (\ref{conformal_invariance_aut_D}), there exists a function \(f(u)\) such that:
    \[
    \tilde{\psi} = f(u) \cdot \psi.
    \]

    \item[(iii)]  
    Since \(f(u)\) is given by:
    \[
    f(u) = (1 - |u|^2)^{\alpha},
    \]
    where \(\alpha\) is the conformal dimension, we conclude that the partition function:
    \[
    (1 - |u|^2)^{\alpha} \cdot \psi
    \]
    has conformal dimension \(\lambda(u) + \alpha\).
\end{itemize}
\end{proof}

\section{Coulomb gas solutions to the null vector equations}

\subsection{Coulomb gas correlation and Coulomb gas integral}

Recall that the Coulomb gas correlation differential associated with a divisor $\boldsymbol{\sigma} = \sum_{j=1}^{n} \sigma_j \cdot z_j$ on the Riemann sphere $\hat{\mathbb{C}}$ is given by
\begin{equation}
C_{(b)}[\boldsymbol{\sigma}] = \prod_{j<k} (z_j - z_k)^{\sigma_j \sigma_k},
\end{equation}
 where the product is taken over all finite $z_j$ and $z_k$. 

\begin{defn}[Monodromy of Coulomb Gas Correlation Differential]
Let $\boldsymbol{\sigma} = \sum_{j=1}^n \sigma_j \cdot z_j$ be a divisor on $\mathbb{C}$ with associated Coulomb gas correlation differential
\[
C_{(b)}[\boldsymbol{\sigma}] = \prod_{j < k} (z_j - z_k)^{\sigma_j \sigma_k}.
\]
This function is multivalued due to the presence of non-integer exponents. Its multivaluedness is described by the \emph{monodromy representation} arising from analytic continuation around branch points.

To illustrate the basic mechanism, consider the case $n = 2$. Then
\[
C_{(b)}[\boldsymbol{\sigma}] = (z_1 - z_2)^{\sigma_1 \sigma_2},
\]
which is analytic in $z_1$ on $\mathbb{C} \setminus \{z_2\}$. If we analytically continue this function as $z_1$ travels once counterclockwise around $z_2$, the function picks up a multiplicative factor of $e^{2\pi i \sigma_1 \sigma_2}$. 

Thus, the monodromy representation
\[
\rho: \pi_1(\mathbb{C} \setminus \{z_2\}, z_1) \longrightarrow \mathbb{C}^*, \quad \rho(C_2) = e^{2\pi i \sigma_1 \sigma_2}
\]
captures how the function changes under analytic continuation around the singularity at $z_2$, where $C_2$ is the loop encircling $z_2$.

In general, for $\boldsymbol{\sigma} = \sum_{j=1}^n \sigma_j \cdot z_j$, the function $C_{(b)}[\boldsymbol{\sigma}]$ is analytic in $z_1$ on $\mathbb{C} \setminus \{z_2, \ldots, z_n\}$. The fundamental group $\pi_1(\mathbb{C} \setminus \{z_2, \ldots, z_n\}, z_1)$ is the free group generated by loops $C_j$ encircling each $z_j$ ($j = 2, \ldots, n$), and the monodromy representation
\[
\rho: \pi_1(\mathbb{C} \setminus \{z_2, \ldots, z_n\}) \to \mathbb{C}^*, \quad \rho(C_j) = e^{2\pi i \sigma_1 \sigma_j}
\]
describes the multiplicative factor acquired by $C_{(b)}[\boldsymbol{\sigma}]$ when $z_1$ loops around $z_j$ once in the counterclockwise direction.
\end{defn}

\begin{defn}[Screening Charge]
Let $\boldsymbol{\sigma}$ be a configuration of charges on the Riemann sphere, and let $C_{(b)}[\boldsymbol{\sigma}]$ denote the associated Coulomb gas correlation differential. The conformal dimension of a charge $\sigma \in \mathbb{C}$ inserted at a point $z_j$ is defined by
\begin{equation}
\lambda_b(\sigma) = \frac{\sigma^2}{2} - \sigma b.
\end{equation}
The condition $\lambda_b(\sigma) = 1$ characterizes special charges whose insertions yield integrands of weight $(1,0)$. Solving this quadratic equation yields two solutions:
\[
\sigma = -2a, \quad \sigma = 2(a + b).
\]

A charge $\tau \in \{-2a, \, 2(a + b)\}$ is called a \emph{screening charge}. Consider a divisor of the form
\begin{equation}
\boldsymbol{\sigma} = \sum_i \sigma_i \cdot z_i + \sum_j \tau_j \cdot \xi_j,
\end{equation}
where $\{\tau_j\}$ are screening charges inserted at positions $\{\xi_j\}$.

The resulting Coulomb gas differential on the Riemann sphere $\hat{\mathbb{C}}$ is given by
\begin{equation}
C_{(b)}[\boldsymbol{\sigma}] =
\prod_{i<j} (z_i - z_j)^{\sigma_i \sigma_j}
\prod_{i,k} (z_i - \xi_k)^{\sigma_i \tau_k}
\prod_{j<k} (\xi_j - \xi_k)^{\tau_j \tau_k},
\end{equation}
where the products range over all distinct pairs of points.

Since each $\tau_j$ satisfies $\lambda_b(\tau_j) = 1$, the differential $C_{(b)}[\boldsymbol{\sigma}] \, d\xi_j$ transforms as a holomorphic 1-form in each variable $\xi_j$. This allows for the definition of contour integrals of the form
\[
\int_\Gamma C_{(b)}[\boldsymbol{\sigma}] \, d\xi_1 \cdots d\xi_m,
\]
where $\Gamma$ is a suitable multidimensional integration cycle avoiding branch points.

This procedure is known as \emph{screening}, and it plays a fundamental role in the Coulomb gas formalism. By integrating out screening charges, one obtains new correlation functions that are conformally covariant and satisfy null vector differential equations, as required by conformal field theory.
\end{defn}

We now consider the simplest nontrivial case involving a single screening charge $\xi$. The corresponding Coulomb gas correlation differential takes the form
\begin{equation}
C_{(b)}[\boldsymbol{\sigma}] = 
\prod_{i<j} (z_i - z_j)^{\sigma_i \sigma_j} 
\prod_j (z_j - \xi)^{\sigma_j \tau},
\end{equation}
where $\{z_j\}$ are fixed insertion points with charges $\{\sigma_j\}$, and $\tau$ is the charge at the variable point $\xi$.

Let $\Gamma : [0,1] \to \mathbb{C} \setminus \{z_1, \ldots, z_n\}$ be a path with basepoint $p_0 = \Gamma(0)$. Due to the non-integer exponents, the integrand is multivalued in $\xi$, and its analytic continuation along $\Gamma$ depends on the monodromy of the branches. Consequently, even if $\Gamma$ is a closed loop, the contour integral
\[
\int_\Gamma C_{(b)}[\boldsymbol{\sigma}] \, d\xi
\]
is not necessarily single-valued and may depend on the homotopy class of $\Gamma$ relative to the chosen branch at $p_0$.

This multivaluedness necessitates a more refined homological framework: the integration should be understood in the context of \emph{twisted homology}, where chains are equipped with local system coefficients determined by the monodromy representation of the integrand. In this setting, valid integration cycles are twisted 1-cycles, which keep track of the phase accumulated during analytic continuation.

A canonical example of such an integration path is the \emph{Pochhammer contour} $\mathscr{P}(z_i, z_j)$, which loops around two branch points $z_i$ and $z_j$ alternately. Though homologically trivial in ordinary homology, this contour generates a nontrivial class in twisted homology and yields a well-defined integral. These twisted cycles form the natural domain of integration for Coulomb gas differentials with screening charges.

\begin{defn}[Pochhammer Contour]
Let $\{z_1, z_2, \ldots, z_n\} \subset \mathbb{C}$ be distinct points. The punctured plane $\mathbb{C} \setminus \{z_1, \ldots, z_n\}$ is homotopy equivalent to a bouquet of $n$ circles, $\bigvee_{i=1}^{n} S^1$, and its fundamental group is the free group:
\[
\pi_1\left(\mathbb{C} \setminus \{z_1, \ldots, z_n\}\right) \cong \ast_{i=1}^{n} \mathbb{Z},
\]
generated by simple loops $C_i$ encircling each puncture $z_i$ in the positive (counterclockwise) direction.

The \emph{Pochhammer contour} associated with a pair of points $(z_i, z_j)$ is defined as the commutator of the generators $C_i$ and $C_j$:
\begin{equation}
\mathscr{P}(z_i, z_j) := C_i C_j C_i^{-1} C_j^{-1}.
\end{equation}
Geometrically, this contour first winds around $z_i$, then around $z_j$, and then retraces both loops in reverse order.

Although $\mathscr{P}(z_i, z_j)$ is null-homologous in ordinary homology, it typically represents a nontrivial class in \emph{twisted homology}, where chains are valued in a local system determined by the monodromy of a multivalued function. Such contours are essential for defining well-posed integrals of Coulomb gas correlation differentials, which exhibit nontrivial monodromy around insertion points.

\begin{figure}[h]
    \centering
    \includegraphics[width=10cm]{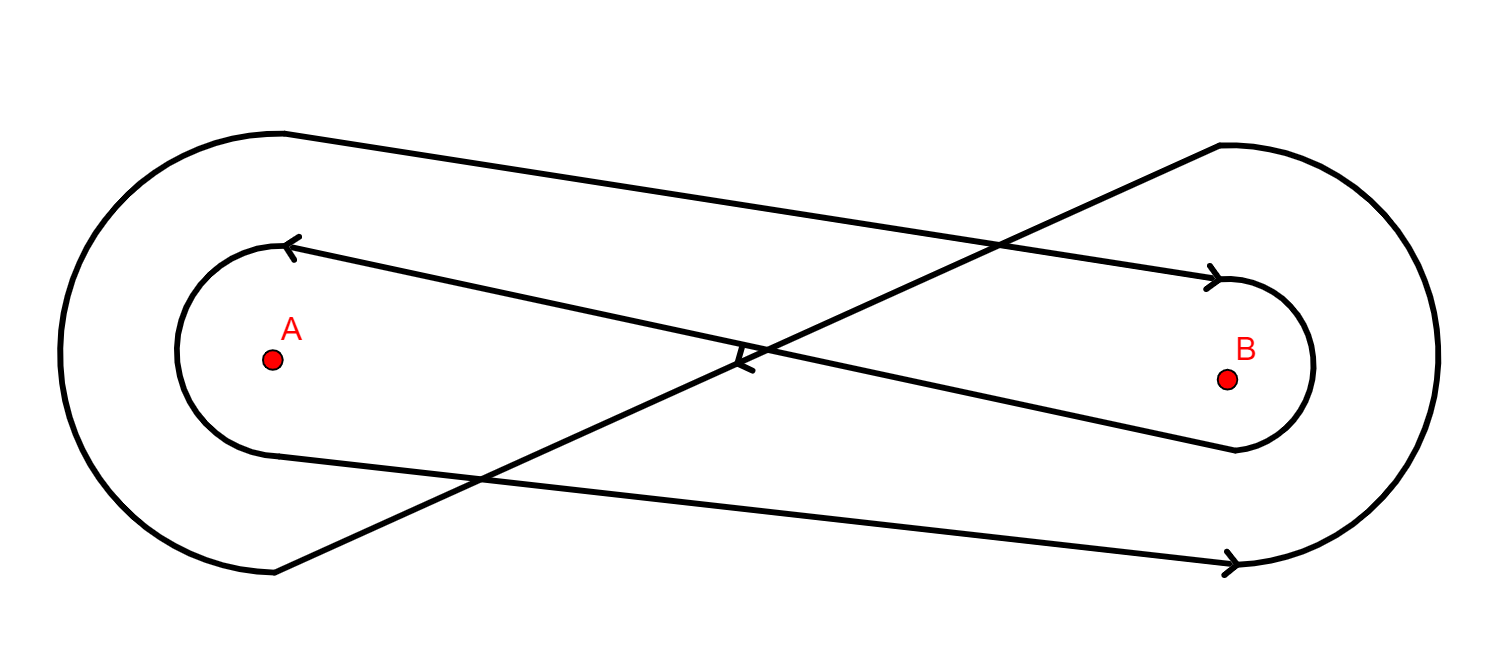}
    \caption{The Pochhammer contour $\mathscr{P}(z_i, z_j)$: a commutator loop around $z_i$ and $z_j$.}
    \label{Pochhammer}
\end{figure}
\end{defn}

We now analyze the role of the Pochhammer contour in defining single-valued integrals of multivalued Coulomb gas differentials.

By definition, the Pochhammer contour $\mathscr{P}(z_i, z_j) := C_i C_j C_i^{-1} C_j^{-1}$ is a commutator of simple loops $C_i$, $C_j$ around $z_i$ and $z_j$, respectively. Since the winding numbers of a loop and its inverse cancel, the total winding number of $\mathscr{P}(z_i, z_j)$ around any puncture vanishes:
\begin{equation}
\operatorname{wind}(\mathscr{P}(z_i, z_j), z_k) = 0, \qquad \text{for all } k = 1, \ldots, n.
\end{equation}
In particular, $\mathscr{P}(z_i, z_j)$ encircles neither $z_i$ nor $z_j$ in total:
\begin{equation}
\operatorname{wind}(\mathscr{P}(z_i, z_j), z_i) = \operatorname{wind}(\mathscr{P}(z_i, z_j), z_j) = 0.
\end{equation}

As a consequence, when the integrand is of the form
\[
C_{(b)}[\boldsymbol{\sigma}] = \prod_{k=1}^n (z_k - \xi)^{\sigma_k \tau},
\]
the analytic continuation of $C_{(b)}[\boldsymbol{\sigma}]$ along $\mathscr{P}(z_i, z_j)$ returns to the original branch, and the monodromy along this loop is trivial:
\begin{equation}
\rho(\mathscr{P}(z_i, z_j)) = 1.
\end{equation}

\begin{thm}[Base Point Independence]
Let $\Gamma = \mathscr{P}(z_i, z_j)$ be a Pochhammer contour, and let $p_0 = \Gamma(0)$ denote its base point. Then the integral
\begin{equation}
\int_{\Gamma} C_{(b)}[\boldsymbol{\sigma}] \, d\xi
\end{equation}
is independent of the choice of base point $p_0$.
\end{thm}

\begin{proof}
Let $p_0'$ be another base point, and let $\gamma$ be a path from $p_0'$ to $p_0$. Define the conjugated loop $\Gamma' = \gamma \cdot \Gamma \cdot \gamma^{-1}$. Since the integrand is single-valued along $\Gamma$, and $\rho(\Gamma) = 1$, we have
\[
\int_{\Gamma'} C_{(b)}[\boldsymbol{\sigma}] \, d\xi = \int_{\Gamma} C_{(b)}[\boldsymbol{\sigma}] \, d\xi.
\]
Hence the integral is independent of the base point.
\end{proof}

\begin{remark}
The Pochhammer contour is a canonical example of a nontrivial twisted cycle, but the base point independence property extends to any closed contour $\Gamma$ satisfying:
\begin{itemize}
    \item[(i)] $\operatorname{wind}(\Gamma, z_k) = 0$ for all $k = 1, \ldots, n$;
    \item[(ii)] $\Gamma$ represents a nontrivial class in the twisted homology group.
\end{itemize}
Under these conditions, $\Gamma$ lies in the twisted homology group $H_1(\mathbb{C} \setminus \{z_1, \ldots, z_n\}; \mathbb{C}_\rho)$, where $\mathbb{C}_\rho$ is the rank-one local system determined by the monodromy representation $\rho$ of the integrand.

This framework generalizes naturally to the case of $m$ screening charges $\xi_1, \ldots, \xi_m$. In that setting, the integration domain is the product of $m$ twisted cycles $\Gamma_1 \times \cdots \times \Gamma_m$, with each $\Gamma_j$ lying in $H_1(\mathbb{C} \setminus \{z_1, \ldots, z_n\}; \mathbb{C}_\rho)$ and chosen, for instance, as pairwise non-intersecting Pochhammer contours. The resulting integral
\begin{equation}
\int_{\Gamma_1} \cdots \int_{\Gamma_m} C_{(b)}[\boldsymbol{\sigma}] \, d\xi_1 \cdots d\xi_m
\end{equation}
defines a well-posed conformally covariant correlation function.
\end{remark}
\begin{thm}[Conformal Invariance of the Coulomb Gas Integral] \label{conformal invariance of Coulomb gas integral}
Let $\boldsymbol{\sigma} = \sum_{i=1}^{n} \sigma_i \cdot z_i + \sum_{j=1}^{m} \tau_j \cdot \xi_j$ be a divisor on the Riemann sphere $\hat{\mathbb{C}}$, where each screening charge $\tau_j \in \{-2a,\, 2(a + b)\}$ is chosen so that its conformal dimension satisfies $\lambda_b(\tau_j) = 1$. Let $h: \hat{\mathbb{C}} \to \hat{\mathbb{C}}$ be a Möbius transformation, and define
\[
\zeta_j := h(\xi_j), \quad h(\boldsymbol{\sigma}) := \sum_{i=1}^{n} \sigma_i \cdot h(z_i) + \sum_{j=1}^{m} \tau_j \cdot \zeta_j.
\]

Then the Coulomb gas integral transforms covariantly under $h$ as
\begin{equation} 
\left( \prod_{i=1}^{n} h'(z_i)^{\lambda_b(\sigma_i)} \right) 
\oint_{h(\Gamma)} C_{(b)}[h(\boldsymbol{\sigma})] \, d\zeta_1 \cdots d\zeta_m
=
\oint_{\Gamma} C_{(b)}[\boldsymbol{\sigma}] \, d\xi_1 \cdots d\xi_m,
\end{equation}
where $h(\Gamma)$ denotes the image of the integration contour $\Gamma$ under $h$.

Each differential $d\zeta_j$ transforms as $d\zeta_j = h'(\xi_j) \, d\xi_j$, and since $\lambda_b(\tau_j) = 1$, the integrand $C_{(b)}[\boldsymbol{\sigma}]\, d\xi_1 \cdots d\xi_m$ is invariant under pullback by $h$, up to the multiplicative factor $\prod_{i} h'(z_i)^{\lambda_b(\sigma_i)}$ determined by the insertion points.
\end{thm}

\begin{proof}
    The conformal invariance of the Coulomb gas integral naturally comes from the conformal invariance of the Coulomb gas correlation differential.
\begin{equation}
C_{(b)}[\boldsymbol{\sigma}]=\left(\prod_{i} h'(z_i)^{\lambda_i}\right) \left(\prod_j h'(\xi_j)\right) C_{(b)}\left[h(\boldsymbol{\sigma})\right]
\end{equation} 
Since $\zeta_j = h(\xi_j)$, then $d\xi_j = \frac{d\zeta_j}{h'(\xi_j)}$, we have
\begin{equation}
\begin{aligned}
&\left(\prod_{i} h'(z_i)^{\lambda_j}\right) \oint_{h(\Gamma)} C_{(b)}\left[h(\boldsymbol{\sigma})\right]d\zeta_1 d\zeta_2 \ldots d\zeta_m= \\
& \oint_{\Gamma} C_{(b)}\left[\boldsymbol{\sigma}\right]\frac{d\zeta}{ \left(\prod_j h'(\xi_j)\right) }=\oint_{\Gamma} C_{(b)}\left[\boldsymbol{\sigma}\right]d\xi_1 d\xi_2 \ldots d\xi_m
\end{aligned}
\end{equation}
\end{proof}

\begin{cor}\label{conformal ward for coulomb gas integral} The Coulomb gas integral $\mathcal{J}(\boldsymbol{z})=\oint_{\mathcal{C}_1} \ldots \oint_{\mathcal{C}_m} \Phi_\kappa(\boldsymbol{z}, \boldsymbol{\xi}) d \xi_m \ldots d \xi_1$ 

 $\Phi_\kappa(\boldsymbol{z}, \boldsymbol{\xi})$ is a Coulomb gas correlation function of conformal dimension $\lambda_i = \lambda_i(\sigma_i)$ at $z_i$, and screening charges $\xi_j$ of conformal dimension $1$.

satisfy the following conformal Ward's indentities:
\begin{equation} \label{Ward identities}
\begin{aligned}
&\left[\sum_{i=1}^{n} \partial_{z_i}\right] \mathcal{J}(\boldsymbol{z}) =0,\\
& \left[\sum_{i=1}^{n}\left(z_i \partial_{z_i}+\lambda_i(\sigma_i)\right)\right]\mathcal{J}(\boldsymbol{z})=0, \\
& \left[\sum_{i=1}^{n}\left(z_i^2 \partial_{z_i}+2 \lambda_i(\sigma_i) z_i\right)\right] \mathcal{J}(\boldsymbol{z})=0.
\end{aligned}
\end{equation}
\end{cor}
\begin{proof}
The Ward identities follow from the invariance of the Coulomb gas integral $\mathcal{J}(\boldsymbol{z})$ under Möbius transformations. Consider the following three one-parameter families of conformal maps:
\[
h_\epsilon^{(1)}(z) = z + \epsilon, \quad
h_\epsilon^{(2)}(z) = (1 + \epsilon) z, \quad
h_\epsilon^{(3)}(z) = \frac{z}{1 + \epsilon z},
\]
which correspond to translations, dilations, and special conformal transformations, respectively.

By Theorem~\ref{conformal invariance of Coulomb gas integral}, the Coulomb gas integral transforms covariantly under Möbius maps:
\[
\left( \prod_{i=1}^n h'(z_i)^{\lambda_i} \right) \mathcal{J}(h_\epsilon(\boldsymbol{z})) = \mathcal{J}(\boldsymbol{z}).
\]
Taking the derivative with respect to $\epsilon$ at $\epsilon = 0$, we obtain infinitesimal constraints corresponding to conformal Ward identities.

\begin{itemize}
    \item Translation: For $h_\epsilon(z) = z + \epsilon$, we have $h'(z) = 1$, so:
    \[
    \left.\frac{d}{d\epsilon}\right|_{\epsilon=0} \mathcal{J}(z_1 + \epsilon, \ldots, z_n + \epsilon) = 0.
    \]
    By the chain rule, this gives:
    \[
    \sum_{i=1}^n \partial_{z_i} \mathcal{J}(\boldsymbol{z}) = 0.
    \]

    \item Dilation: For $h_\epsilon(z) = (1 + \epsilon) z$, we have $h'(z) = 1 + \epsilon$, so:
    \[
    \left.\frac{d}{d\epsilon}\right|_{\epsilon=0} \left( \prod_{i=1}^n (1 + \epsilon)^{\lambda_i} \cdot \mathcal{J}((1 + \epsilon)\boldsymbol{z}) \right) = 0.
    \]
    Differentiating yields:
    \[
    \sum_{i=1}^n \left( z_i \partial_{z_i} + \lambda_i \right) \mathcal{J}(\boldsymbol{z}) = 0.
    \]

    \item Shearing: For $h_\epsilon(z) = \frac{z}{1 + \epsilon z}$, we compute
    \[
    h'(z) = \frac{1}{(1 + \epsilon z)^2} \approx 1 - 2 \epsilon z + o(\epsilon),
    \quad
    h_\epsilon(z) \approx z - \epsilon z^2 + o(\epsilon).
    \]
    Plugging into the covariance relation and differentiating gives:
    \[
    \sum_{i=1}^n \left( z_i^2 \partial_{z_i} + 2 \lambda_i z_i \right) \mathcal{J}(\boldsymbol{z}) = 0.
    \]
\end{itemize}

This establishes the three global conformal Ward identities in~\eqref{Ward identities}.
\end{proof}

\subsection{Classification and link pattern}
\label{Classification of screening solutions}

Throughout this section, we modify the notation by setting $z_{n+1} = u$ and $z_{n+2} = u*$. As usual, let $z_1 < z_2 < \ldots < z_{n-1} < z_n$.

We begin by considering the charge $\boldsymbol{\sigma} = \sum_{j=1}^{n+m+2} \sigma_j \cdot z_j$ and the Coulomb gas correlation:

\[
C_{(b)}[\boldsymbol{\sigma}] = \Phi\left(z_1, \ldots, z_{n+2+m}\right) = \prod_{i<j}^{n+2+m} \left(z_j - z_i\right)^{\sigma_i \sigma_j}.
\]

Our strategy is to choose the $\sigma_i$ (i.e., the charges associated with the divisor in the Coulomb gas correlation) such that for $1 \leq i \leq n$, and $\lambda_j = \frac{\sigma_j^2}{2} - \sigma_j b$:

\begin{equation} \label{desired form}
\begin{aligned}
\left[\frac{\kappa}{4} \partial_i^2 + \sum_{j \neq i}^{n} \left( \frac{\partial_j}{z_j - z_i} - \frac{\lambda_j}{(z_j - z_i)^2} \right) \right. 
&+ \frac{\partial_{n+1}}{z_{n+1} - z_i} + \frac{\partial_{n+2}}{z_{n+2} - z_i} \\
&- \frac{\lambda_{n+1}}{(z_{n+1} - z_i)^2} - \frac{\lambda_{n+2}}{(z_{n+2} - z_i)^2} \Bigg] \Phi \\
= &\sum_{k=n+3}^{n+2+m} \partial_k (\ldots),
\end{aligned}
\end{equation}

\begin{thm} \label{Coulomb gas integral screening}

\begin{itemize}
\item[(i)] If we choose 
$\sigma_j = a = \sqrt{\frac{2}{\kappa}}$, and $\lambda_j = \frac{a^2}{2} - ab = \frac{6 - \kappa}{2\kappa}$ for $1 \leq j \leq n$, and $\lambda_j = \frac{\sigma_j^2}{2} - \sigma_j b$ for $n+1 \leq j \leq n+2$, then we obtain the following null vector equation:

\begin{equation} \label{null vector equation at a}
\begin{aligned}
&\left[\frac{\kappa}{4} \partial_j^2 + \sum_{k \neq j}^{n} \left( \frac{\partial_k}{z_k - z_j} - \frac{(6 - \kappa) / 2\kappa}{(z_k - z_j)^2} \right) + \frac{\partial_{n+1}}{z_{n+1} - z_j} \right. \\
& \left. + \frac{\partial_{n+2}}{z_{n+2} - z_j} - \frac{\lambda_{n+1}}{(z_{n+1} - z_j)^2} - \frac{\lambda_{n+2}}{(z_{n+2} - z_j)^2} \right] \Phi \\
= &\sum_{k=n+3}^{n+2+m} \partial_k \left( -\frac{\Phi(z_1, \ldots, z_{n+m+2})}{z_k - z_j} \right)
\end{aligned}
\end{equation}

for all $j \in \{1, 2, \ldots, n\}$. Thus, we attain the desired form (\ref{desired form}) for all $j \in \{1, 2, \ldots, n\}$.

Currently, the number of screening charges $m$ and the values of $\sigma_k = 2a$ or $\sigma_k = 2(a + b)$ for $k \in \{n+3, n+4, \ldots, n+m+2\}$ remain unspecified. The charges $\sigma_{n+1} = \overline{\sigma_{n+2}}$ are chosen such that $\boldsymbol{\sigma} = \sum_j \sigma_j \cdot z_j$ satisfies the neutrality condition ($NC_b$).

\item[(ii)] If $n = 2k$, $m = k - 1$, and we choose $\sigma_j = a$ for all $j \in \{1, 2, \ldots, n-1\}$, $\sigma_n = 2b - a$, and the sign of $\sigma_k = -2a$ for all $k \in \{n+1, n+2, \ldots, n+m\}$, then we have the following null vector equation for $j \in \{1, 2, \ldots, n-1\}$:

\begin{equation} \label{null vector equation at 2b-a}
\begin{aligned}
&\left[\frac{\kappa}{4} \partial_n^2 + \sum_{k=1}^{n-1} \left( \frac{\partial_k}{z_k - z_n} - \frac{(6 - \kappa) / 2\kappa}{(z_k - z_n)^2} \right) \right] \Phi \\
= &\sum_{k=n+1}^{n+m-1} \partial_k \left( -\frac{\Phi(z_1, \ldots, z_{n+m})}{z_k - z_n} \right) \\
&+ \frac{1}{2} \sum_{k=n+1}^{n+m-1} \partial_k \left[ \frac{8 - \kappa}{z_k - z_n} \left( \prod_{s=1}^{n-1} \frac{z_k - z_s}{z_n - z_s} \prod_{\substack{t=n+1 \\ t \neq k}}^{n+m} \left( \frac{z_n - z_t}{z_k - z_t} \right)^2 \right) \Phi \right].
\end{aligned}
\end{equation}

Since the right-hand side of (\ref{null vector equation at 2b-a}) consists of derivatives with respect to $z_k$ for $k \in \{n+1, n+2, \ldots, n+m\}$, we obtain the desired form (\ref{desired form}) for $j = n$ as well. Therefore, the null vector equations are satisfied for all $1 \leq j \leq n$.

\end{itemize}

\end{thm}

Then we will integrate $z_{n+3}, \ldots, z_{n+2+m}$ on both sides of (\ref{desired form}) around nonintersecting closed contours $\Gamma_1, \ldots, \Gamma_m$. On the left side, the integrand is a smooth function of $z_1, \ldots, z_{n+m+2}$ because the contours do not intersect. 

\begin{figure}[ht]
    \centering
    \includegraphics[width=10cm]{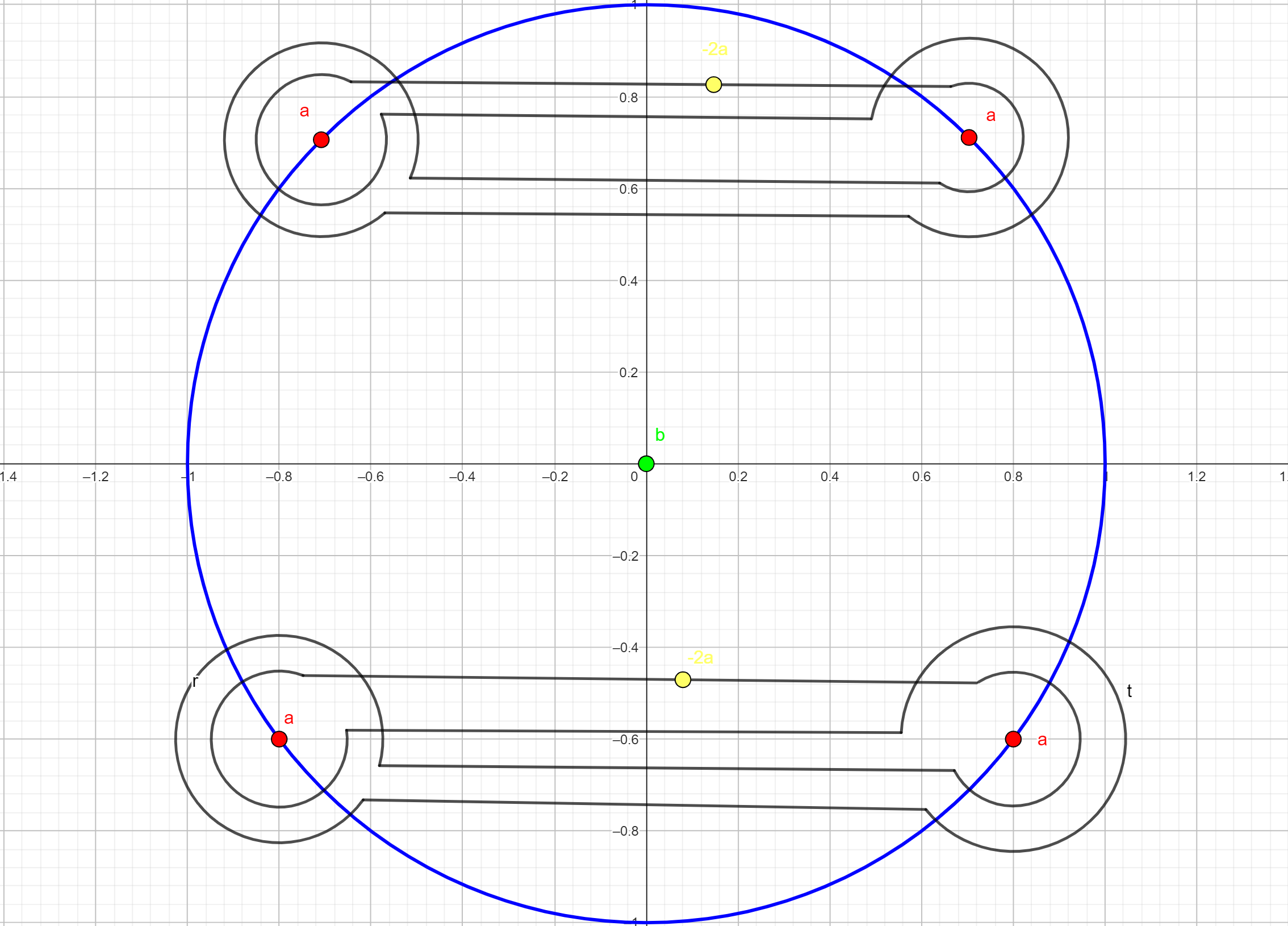}
    \caption{Example: $z_1,z_2,z_3,z_4$ with 2 screening charges $\xi_1,\xi_2$}
    \label{Screening}
\end{figure}

Integration on the right side is expected to give zero. To attain this,
we carefully choose the integration contour for $z_{n+3}, \ldots, z_{n+2+m}$. A commonly used integration contour is the Pochhammer contour encircling two points $z_i$ and $z_j$, denoted by $\mathscr{P}\left(z_i, z_{j}\right)$.

Because either side of (\ref{desired form}) is absolutely integrable on each path, we may perform these integrations in any order according to Fubini's theorem. Integrating the right side of (\ref{desired form}) therefore gives zero. Finally, because the contours do not intersect, we have sufficient continuity to use the Leibniz rule of integration to exchange the order of differentiation and integration on the left side of (\ref{desired form}). (If $\Gamma_p$ intersects $\Gamma_q$ but $\sigma_p \sigma_q>0$, then the contour integral $\oint \Phi$ is not improper. In this event, we may still use the Leibniz rule to perform this last step as long as we may continuously deform these contours so they do not intersect.) We, therefore, find that the Coulomb gas integral $\mathcal{J}:=\oint \Phi$ satisfies the null vector equations (\ref{null vector equation for Screening solutions}).

As detailed in Theorem (\ref{Coulomb gas integral screening}), we are able to construct solutions to the null vector PDEs and Ward's identities via screening.
These solutions satisfy the following null vector equations:

\begin{equation} \label{null vector equation for Screening solutions}
\left[\frac{\kappa}{4} \partial_j^2+\sum_{k \neq j}^{n}\left(\frac{\partial_k}{z_k-z_j}-\frac{(6-\kappa) /  2\kappa}{\left(z_k-z_j\right)^2}\right)+\frac{\partial_{n+1}}{u-z_j} +\frac{\partial_{n+2}}{u^*-z_j} \right.
\left. -\frac{\lambda_{(b)}(u)}{\left(u-z_j\right)^2}-\frac{\lambda_{(b)}(u^*)}{\left(u^*-z_j\right)^2} \right] \mathcal{J}\left(\boldsymbol{z},u\right)=0
\end{equation}
for $j=1,2,\ldots,n$,
and the following ward identities by corollary (\ref{conformal ward for coulomb gas integral}): 
\begin{equation} \label{Ward identities for screening solutions}
\begin{aligned}
&\left[\sum_{i=1}^{n} \partial_{z_i}+ \partial_u+\partial_{u^*}\right] \mathcal{J}(\boldsymbol{z},u) =0,\\
& \left[\sum_{i=1}^{n}\left(z_i \partial_{z_i}+\frac{6-\kappa}{2\kappa}\right)+ u\partial_u+\lambda_{(b)}(u)u+u^*\partial_{u^*}+\lambda_{(b)}(u^*)u^*\right]\mathcal{J}(\boldsymbol{z},u)=0, \\
& \left[\sum_{i=1}^{n}\left(z_i^2 \partial_{z_i}+\frac{6-\kappa}{\kappa}z_i\right)+ u^2\partial_u+2\lambda_{(b)}(u)u+(u^*)^2\partial_{u^*}+2\lambda_{(b)}(u^*)u^*\right] \mathcal{J}(\boldsymbol{z},u)=0
\end{aligned}
\end{equation}
where $\lambda_{(b)}(u)$ and $\lambda_{(b)}(u^*)$ are the conformal dimensions of $u$ and $u^*$.

We need to choose a set of integration contours to screen $\Phi$. we will explain how we choose integration contours, which lead to four types of screening solutions, see theorem (\ref{solution space to null and ward}). We conjecture that these screening solutions span the solution space of the null vector equations (\ref{null vector equation for Screening solutions}) 
and the Ward's identities (\ref{Ward identities for screening solutions}).

To do this, let's begin by defining the link patterns that characterize the topology of integration contours.

\begin{defn}[Radial link pattern] \label{radial link pattern} Given $\boldsymbol{z}=\{z_1,z_2,...,z_n\}$ on the unit circle, a radial link pattern is a homotopically equivalent class of non-intersecting curves connecting pair of boundary points (links/arcs) or connecting boundary points and the origin (rays).

The link pattern in the unit disk with one interior point is known as the standard module $W_{n,m}$ over the affine Temperley-Lieb algebra $aTL_n$. The link patterns with $n$ boundary points and $m$ links are called $(n,m)$-links, denoted by ${\rm LP}(n,m)$. 

The number of radial $(n,m)$-links is given by $|{\rm LP}(n,m)|=C_{n}^{m}$.
\end{defn}

\begin{defn}[Chordal link pattern] Given $\boldsymbol{z}=\{z_1,z_2,...,z_n\}$ on the real line, a link pattern is a homotopically equivalent class of non-intersecting curves connecting pair of boundary points (links) or connecting boundary points and the infinity (rays). The link patterns with $n$ boundary points and $m$ links are called $(n,m)$-links, denoted by ${\rm LP}(n,m)$.

The number of chordal $(n,m)$-links is given by $|{\rm LP}(n,m)|=C_{n}^{m+1}-C_{n}^{m}$.
\end{defn}

By theorem (\ref{Coulomb gas integral screening}), when all $\sigma_i= a$, $1\leq i \leq n$, we can assign charge $-2a$ or $2(a+b)$ to screening charges arbitrarily.

\begin{itemize}

\item (Radial ground solutions)
In the upper half plane $\mathbb{H}$, we assign charge $a$ to $z_1,z_2,\ldots,z_n$, charge $-2a$ to $\xi_1,\ldots,\xi_m$ and charge $\sigma_u=\sigma_{u^*}=b-\frac{(n-2m)a}{2}$ to marked points $u$ and $u^*$ 
to maintain neutrality condition ($\rm{NC_b}$).
\begin{equation} \label{multiple radial SLE(kappa) master function in H}
\begin{aligned}
\Phi_{\kappa}\left(z_1, \ldots, z_{n}, \xi_1,\xi_2,\ldots,\xi_m, u \right)= & \prod_{i<j}\left(z_i-z_j\right)^{a^2} \prod_{j<k}\left(z_j-\xi_k\right)^{-2a^2}  \prod_{j<k}\left(\xi_j-\xi_k\right)^{4a^2}  \\
& \prod_{j}(z_i-u)^{a(b-\frac{(n-2m)a}{2})}
\prod_{j}(z_i-u^*)^{a(b-\frac{(n-2m)a}{2})}
\\
&\prod_{j}(\xi_j-u)^{-2a(b-\frac{(n-2m)a}{2})}\prod_{j}(\xi_j-u^*)^{-2a(b-\frac{(n-2m)a}{2})}
\end{aligned}
\end{equation}

In the unit disk $\mathbb{D}$, if we set $u=0$, then we have:

\begin{equation}
\begin{aligned}
\Phi_{\kappa}\left(z_1, \ldots, z_{n}, \xi_1,\xi_2,\ldots,\xi_m\right)= & \prod_{i<j}\left(z_i-z_j\right)^{a^2} \prod_{j<k}\left(z_j-\xi_k\right)^{-2a^2}  \prod_{j<k}\left(\xi_j-\xi_k\right)^{4a^2}  \\
& \prod_{j}z_{i}^{a(b-\frac{(n-2m)a}{2})}\prod_{j}\xi_{j}^{-2a(b-\frac{(n-2m)a}{2})}
\end{aligned}
\end{equation}

\begin{itemize}
    \item[(1)] $(-2a)\cdot a=-\frac{4}{\kappa}$.  $\xi_i=z_j$ is a singular point of the type $\left(\xi_i-z_j\right)^{-4 / \kappa}$.
     \item[(2)] $(-2a)\cdot (-2a)=\frac{8}{\kappa}$. $\xi_i=\xi_j$ is a singular point of of the type $(\xi_i-\xi_j)^{\frac{8}{\kappa}}$
     \item[(3)] $(-2a)\cdot (b-\frac{(n-2m)a}{2})=\frac{2(n-2m+2)}{\kappa}$.  $\xi=u$ and $\xi=u^*$ are singular points of the type $(\xi_i-u)^{\frac{2(n-2m+2)}{\kappa}}$ and $(\xi_i-u^*)^{\frac{2(n-2m+2)}{\kappa}}$
\end{itemize}

In this case, for $m \leq \frac{n+2}{2}$ and a $(n,m)$ radial link pattern $\alpha$, we can choose $p$ non-intersecting Pochhammer contours $\mathcal{C}_1,\mathcal{C}_2,\ldots,\mathcal{C}_m$ surrounding pairs of points (which correspond to links in a radial link pattern), see (\ref{radial link pattern}) to integrate $\Phi_{\kappa}$, we obtain
\begin{equation}
    \mathcal{J}^{(m, n)}_{\alpha}(\boldsymbol{z}):=\oint_{\mathcal{C}_1} \ldots \oint_{\mathcal{C}_m} \Phi_\kappa(\boldsymbol{z}, \boldsymbol{\xi}) d \xi_m \ldots d \xi_1 .
\end{equation}
In particular, if $m=0$, we call $\Phi_{\kappa}$ the fermionic ground solution.

Note that the charges at $u$ and $u^*$ are given by $\sigma_u=\sigma_{u^*}=b-\frac{(n-2m)a}{2}$, thus $$\lambda_{(b)}(u)=\lambda_{(b)}(u^*)=\frac{(n-2m)^2a^2}{8}-\frac{b^2}{2}= \frac{(n-2m)^2}{4\kappa}-\frac{(\kappa-4)^2}{16\kappa}$$

The radial ground solution $\mathcal{J}^{(m,n)}_{\alpha}$ satisfies the null vector equations (\ref{null vector equation for Screening solutions}) and Ward's identities (\ref{Ward identities for screening solutions}) with above $\lambda_{(b)}(u)$ and $\lambda_{(b)}(u^*)$

\item (Radial excited solutions)In the upper half plane $\mathbb{H}$, we assign charge $a$ to $z_1,z_2,\ldots,z_n$, charge $-2a$ to $\xi_1,\ldots,\xi_m$ and charge $2(a+b)$ to $\zeta_1,\ldots,\zeta_q$.
Then, we assign charge $\sigma_u=\sigma_{u^*}=b-\frac{(n-2m)a+2q(a+b)}{2}$ to marked points $u$ and $u^*$ 
to maintain neutrality condition ($\rm{NC_b}$).
\begin{equation}
\begin{aligned}
&\Phi_{\kappa}\left(z_1, \ldots, z_{n}, \xi_1,\xi_2,\ldots,\xi_m,\zeta_1,\zeta_2,\ldots,\zeta_q, u \right)=  \\
&\prod_{i<j}\left(z_i-z_j\right)^{a^2} \prod_{j<k}\left(z_j-\xi_k\right)^{-2a^2}  \prod_{j<k}\left(\xi_j-\xi_k\right)^{4a^2}  \\
& \prod_{j<k}\left(z_j-\zeta_k\right)^{2a(a+b)}  \prod_{j<k}\left(\zeta_j-\zeta_k\right)^{4(a+b)^2} 
\\
& \prod_{j}(z_i-u)^{a\sigma_u}
\prod_{j}(z_i-u^*)^{a\sigma_{u^*}}
\\
&\prod_{j}(\xi_j-u)^{-2a\sigma_u}\prod_{j}(\xi_j-u^*)^{-2a\sigma_{u^*}}
\\
&\prod_{j}(\zeta_j-u)^{2(a+b)\sigma_u}\prod_{j}(\zeta_j-u^*)^{2(a+b)\sigma_{u^*}}
\end{aligned}
\end{equation}

In the unit disk $\mathbb{D}$, if we set $u=0$, then we have

\begin{equation}
\begin{aligned}
\Phi_{\kappa}\left(z_1, \ldots, z_{n}, \xi_1,\xi_2,\ldots,\xi_m\right)= & \prod_{i<j}\left(z_i-z_j\right)^{a^2} \prod_{j<k}\left(z_j-\xi_k\right)^{-2a^2}  \prod_{j<k}\left(\xi_j-\xi_k\right)^{4a^2}  \\
& \prod_{j}z_{i}^{a(b-\frac{(n-2m)a}{2})}\prod_{j}\xi_{j}^{-2a(b-\frac{(n-2m)a}{2})}
\end{aligned}
\end{equation}

\begin{itemize}
    \item[(1)] $(-2a)\cdot a=-\frac{4}{\kappa}$.  $\xi_i=z_j$ is a singular point of the type $\left(\xi_i-z_j\right)^{-4 / \kappa}$.
     \item[(2)] $(-2a)\cdot (-2a)=\frac{8}{\kappa}$. $\xi_i=\xi_j$ is a singular point of of the type $(\xi_i-\xi_j)^{\frac{8}{\kappa}}$
     \item[(3)] $(-2a)\cdot (b-\frac{(n-2m)a}{2}-q(a+b))=\frac{2(n-2m+2)}{\kappa}+q$.  $\xi=u$ and $\xi=u^*$ are singular points of the type $(\xi_i-u)^{\frac{2(n-2m+2)}{\kappa}+q}$ and $(\xi_i-u^*)^{\frac{2(n-2m+2)}{\kappa}+q}$
     \item[(4)] $2(a+b)\cdot (b-\frac{(n-2m)a}{2}-q(a+b))=\frac{(1-q)\kappa}{4}+\frac{-n+2m-2}{2}$.  $\xi=u$ and $\xi=u^*$ are singular points of the type $(\xi_i-u)^{\frac{(1-q)\kappa}{4}+\frac{-n+2m-2}{2}}$ and $(\xi_i-u^*)^{\frac{(1-q)\kappa}{4}+\frac{-n+2m-2}{2}}$
\end{itemize}

For $q=1$, $\zeta_1=u$ and $\zeta_1=u^*$ are two singular points of degree $\frac{-n+2m-2}{2}$.
We have two choices for screening contours to integrate $\zeta_1$
\begin{itemize}
    \item $n$ odd, Pochhammer contour $\mathscr{P}(u,u^*)$ surrounding $u$ and $u^*$, however,
    $$\int_{\mathscr{P}(u,u^*)}\Phi_{\kappa} d\zeta = 0$$
    
    \item $n$ even, the circle $C(0,\epsilon)$ around $0$ with radius $\epsilon$, this gives the excited solution
    
In this case, for $m \leq \frac{n+2}{2}$ and a $(n,m)$ radial link pattern $\alpha$, we can choose $p$ non-intersecting Pochhammer contours $\mathcal{C}_1,\mathcal{C}_2,\ldots,\mathcal{C}_m$ surrounding pairs of points (which correspond to links in a radial link pattern) to integrate $\Phi_{\kappa}$, we obtain
\begin{equation}
    \mathcal{K}^{(m,n)}_{\alpha}(\boldsymbol{z}):=\oint_{\mathcal{C}_1} \ldots \oint_{\mathcal{C}_m}\oint_{C(0,\epsilon)} \Phi_\kappa(\boldsymbol{z}, \boldsymbol{\xi}) d \xi_m \ldots d \xi_1 d\zeta_1.
\end{equation}
In particular, if $p=0$, we call $\Phi_{\kappa}$ the fermionic excited solution.
\end{itemize} 

Note that the charges at $u$ and $u^*$ are given by $\sigma_u=\sigma_{u^*}=\frac{(2m-n-2)a}{2}$
$$\lambda_{(b)}(u)=\lambda_{(b)}(u^*)=\frac{(n-2m+\frac{\kappa}{2})^2}{4\kappa}-\frac{(\kappa-4)^2}{16\kappa}$$ 
The radial excited solution $\mathcal{K}^{(m,n)}_{\alpha}$ 
satisfies the null vector equations (\ref{null vector equation for Screening solutions}) and Ward's identities (\ref{Ward identities for screening solutions}) with above $\lambda_{(b)}(u)$ and $\lambda_{(b)}(u^*)$

For $q\geq 2$, since $u$ and $u^*$ are the only singular points for screening charges,
it is impossible to choose two non-intersecting contours for $\{\zeta_1,\zeta_2,\ldots,\zeta_q \}$.

\item (Radial ground solutions with spin $\eta$)
In the upper half plane $\mathbb{H}$, we assign charge $a$ to $z_1,z_2,\ldots,z_n$, charge $-2a$ to $\xi_1,\ldots,\xi_m$.
Then, we assign charge $\sigma_u=b-\frac{(n-2m)a}{2}-\frac{i\eta a}{2}$, $\sigma_{u^*}=b-\frac{(n-2m)a}{2}+\frac{i\eta a}{2}$ to marked points $u$ and $u^*$ 
to maintain neutrality condition ($\rm{NC_b}$).
\begin{equation}
\begin{aligned}
\Phi_{\kappa}\left(z_1, \ldots, z_{n}, \xi_1,\xi_2,\ldots,\xi_m, u \right)= & \prod_{i<j}\left(z_i-z_j\right)^{a^2} \prod_{j<k}\left(z_j-\xi_k\right)^{-2a^2}  \prod_{j<k}\left(\xi_j-\xi_k\right)^{4a^2}  \\
& \prod_{j}(z_i-u)^{a(b-\frac{(n-2m)a}{2}-\frac{i\eta a}{2})}
\prod_{j}(z_i-u^*)^{a(b-\frac{(n-2m)a}{2}+\frac{i\eta a}{2})}
\\
&\prod_{j}(\xi_j-u)^{-2a(b-\frac{(n-2m)a}{2}-\frac{i\eta a}{2})}\prod_{j}(\xi_j-u^*)^{-2a(b-\frac{(n-2m)a}{2}+\frac{i\eta a}{2})}
\end{aligned}
\end{equation}

In the unit disk $\mathbb{D}$, if we set $u=0$, then we have

\begin{equation}
\begin{aligned}
\Phi_{\kappa}\left(z_1, \ldots, z_{n}, \xi_1,\xi_2,\ldots,\xi_m\right)= & \prod_{i<j}\left(z_i-z_j\right)^{a^2} \prod_{j<k}\left(z_j-\xi_k\right)^{-2a^2}  \prod_{j<k}\left(\xi_j-\xi_k\right)^{4a^2}  \\
& \prod_{j}z_{i}^{a(b-\frac{(n-2m)a}{2}-\frac{i\eta a}{2})}\prod_{j}\xi_{j}^{-2a(b-\frac{(n-2m)a}{2}+\frac{i\eta a}{2})}
\end{aligned}
\end{equation}

\begin{itemize}
    \item[(1)] $(-2a)\cdot a=-\frac{4}{\kappa}$.  $\xi_i=z_j$ is a singular point of the type $\left(\xi_i-z_j\right)^{-4 / \kappa}$.
     \item[(2)] $(-2a)\cdot (-2a)=\frac{8}{\kappa}$. $\xi_i=\xi_j$ is a singular point of of the type $(\xi_i-\xi_j)^{\frac{8}{\kappa}}$
     \item[(3)] $(-2a)\cdot (b-\frac{(n-2m)a}{2})=\frac{2(n-2m+2)}{\kappa}$.  $\xi=u$ and $\xi=u^*$ are singular points of the type $(\xi_i-u)^{\frac{2(n-2m+2)}{\kappa}}$ and $(\xi_i-u^*)^{\frac{2(n-2m+2)}{\kappa}}$
\end{itemize}

In this case, for $p \leq \frac{n+2}{2}$ and a $(n,p)$ radial link pattern $\alpha$, we can choose $p$ non-intersecting Pochhammer contours $\mathcal{C}_1,\mathcal{C}_2,\ldots,\mathcal{C}_p$ surrounding pairs of points (which correspond to links in a radial link pattern), see (\ref{radial link pattern}) to integrate $\Phi_{\kappa}$, we obtain
\begin{equation}
    \mathcal{J}^{(m, n,\eta)}_{\alpha}(\boldsymbol{z}):=\oint_{\mathcal{C}_1} \ldots \oint_{\mathcal{C}_m} \Phi_\kappa(\boldsymbol{z}, \boldsymbol{\xi}) d \xi_m \ldots d \xi_1 .
\end{equation}

Note that the charges at $u$ and $u^*$ are given by $\sigma_u=b-\frac{(n-2m)a}{2}-\frac{i\eta a}{2}$, $\sigma_{u^*}=b-\frac{(n-2m)a}{2}+\frac{i\eta a}{2}$. 
$$\lambda_{(b)}(u)=\frac{(n-2m+i\eta)^2a^2}{8}-\frac{b^2}{2}= \frac{(n-2m+i\eta)^2}{4\kappa}-\frac{(\kappa-4)^2}{16\kappa}$$
$$\lambda_{(b)}(u^*)=\frac{(n-2m-i\eta)^2a^2}{8}-\frac{b^2}{2}= \frac{(n-2m-i\eta)^2}{4\kappa}-\frac{(\kappa-4)^2}{16\kappa}$$
The radial ground solution with spin $\eta$, $\mathcal{J}^{(m,n,\eta)}_{\alpha}$ satisfies the null vector equations (\ref{null vector equation for Screening solutions}) and Ward's identities (\ref{Ward identities for screening solutions}) with above $\lambda_{(b)}(u)$ and $\lambda_{(b)}(u^*)$

\end{itemize}

  As shown in theorem (\ref{Coulomb gas integral screening}), if we attach charge $a$ for $z_1,\ldots,z_{n-1}$ and $2b-a$ for $z_c$, where $n=2k$. This corresponds to the charge distribution for multiple chordal SLE($\kappa$) as discussed in \cite{FK15c}. In this case, we can only assign charge $-2a$ to the $k-1$ screening charges and assign no spin at $u,u^*$; Otherwise, the null vector equation at $z_c$ will generally not be satisfied.
  
\begin{itemize}

\item (Chordal solutions) 
In the upper half plane $\mathbb{H}$, we assign charge $a$ to $z_1,z_2,\ldots,z_{n-1}$ and charge $2b-a$ to $z_c$, charge $-2a$ to $\xi_1,\ldots,\xi_m$, where $n=2k$, $m=k-1$, the assign the charge $\sigma_u=\sigma_{u^*}=0$.

\begin{equation}
\begin{aligned}
\Phi_{\kappa}\left(z_1, \ldots, z_{n-1},z_c, \xi_1,\ldots,\xi_m, u\right)= & \prod_{i<j}\left(z_i-z_j\right)^{a^2} \prod_{j<k}\left(z_j-\xi_k\right)^{-2a^2}  \prod_{j<k}\left(\xi_j-\xi_k\right)^{4a^2}  \\
&\prod_{i}(z_i-z_c)^{a(2b-a)} \prod_{j}(\xi_j-z_c)^{-2a(2b-a)}
\end{aligned}
\end{equation}

\begin{itemize}
    \item $(-2a)\cdot a=-\frac{4}{\kappa}$. $\xi_i=z_j$ is a singular point of the type $\left(\xi_i-z_j\right)^{-4 / \kappa}$.
    \item $(-2a)\cdot (2b-a)=\frac{12}{\kappa}-2$. $\xi_i=z_c$ is a singular point of the type $\left(\xi_i-z_c\right)^{\frac{12}{\kappa}-2}$.
     \item $(-2a)\cdot (-2a)=\frac{8}{\kappa}$. $\xi_i=\xi_j$ is a singular point of the type $(\xi_i-\xi_j)^{\frac{8}{\kappa}}$

\end{itemize}

In this case, for a $(2k,k)$ chordal link pattern, we choose $m=k-1$ non-intersecting Pochhammer contours 
$\mathcal{C}_1,\mathcal{C}_2,\ldots,\mathcal{C}_{k-1}$ surrounding pairs of points except $z_c$ (which correspond to links in a chordal link pattern not connected to $z_c$)
see \cite{FK15c} for detailed explanation.
We obtain:
\begin{equation}
    \mathcal{L}_{\alpha}^n(\boldsymbol{z}):=\oint_{\mathcal{C}_1} \ldots \oint_{\mathcal{C}_{k-1}} \Phi_\kappa(\boldsymbol{z}, \boldsymbol{\xi}) d \xi_{k-1} \ldots d \xi_1 .
\end{equation}
\end{itemize}

Note that the charges at $u$ and $u^*$ are given by $\sigma_u=\sigma_{u^*}=0$
$$\lambda_{(b)}(u)=\lambda_{(b)}(u^*)=0$$ 

The chordal solution $\mathcal{J}^{(m,n)}_{\alpha}$ 
satisfies the null vector equations (\ref{null vector equation for Screening solutions}) and Ward's identities (\ref{Ward identities for screening solutions}) with above $\lambda_{(b)}(u)$ and $\lambda_{(b)}(u^*)$.

We can also construct the Coulomb gas integral solutions in angular coordinates. Consider the following Coulomb gas correlation in the angular coordinate,
$$
\Phi(z_1,z_2,\ldots,z_{n+m})=\prod_{1 \leq j<k \leq n+m}\left(\sin \frac{z_j-z_k}{2}\right)^{\sigma_j \sigma_k}.
$$
Then, similar computations show that:

\begin{thm}\label{Trigonometric Coloumb gas integral formula}
    If we choose 
$\sigma_j=a=\sqrt{\frac{2}{\kappa}}, \quad \lambda_j=\frac{a^2}{2}-ab=\frac{6-\kappa}{2\kappa}, \quad 1 \leq j \leq n$
then we have

\begin{equation}
\begin{aligned}
& \left[\frac{\kappa}{2}\partial^{2}_{j}+\sum_{k \neq j}\left(\cot \left(\frac{z_k-z_j}{2}\right) \partial_k-\frac{(6-\kappa)/2\kappa}{2 \sin ^2\left(\frac{z_k-z_j}{2}\right)}\right)\right] \Phi\left(z_1, z_2, \ldots, z_{n+m+2}\right)\\
&=\sum_{k=n+1}^{n+m} \partial_k\left(\cot\left(\frac{z_k-z_j}{2}\right)\Phi\left(z_1, z_2, \ldots, z_{n+m+2}\right)\right)-\left[\frac{1}{2 \kappa}\left(n-2 p+\frac{\kappa}{2} q\right)^2-\frac{1}{2 \kappa}\right]\Phi\left(z_1, z_2, \ldots, z_{n+m+2}\right)
\end{aligned}
\end{equation}

for all $j \in\{1,2, \ldots, n\}$. The number of screening charges $\sigma_k= 2a$ is given by $p$, and the number of screening charges $\sigma_k= 2(a+b)$ is given by $q$, with $m=p+q$.

\end{thm}

Now, we Coulomb gas integral solutions based on the theorem (\ref{Trigonometric Coloumb gas integral formula}).

\begin{itemize} \label{radial ground solution}
    \item Radial ground solutions:
\begin{equation} \label{multiple radial SLE(kappa) master function in angular coordinate}
\Phi_\kappa(\boldsymbol{\theta}, \boldsymbol{\zeta})=\prod_{1 \leq i<j \leq n}\left(\sin\frac{\theta_i-\theta_j}{2}\right)^{a^2} \prod_{1 \leq i<j \leq m}\left(\sin\frac{\zeta_i-\zeta_j}{2}\right)^{4 a^2} \prod_{i=1}^{n} \prod_{j=1}^m\left(\sin\frac{\theta_i-\zeta_j}{2}\right)^{-2 a^2}     
\end{equation}

In this case, for $m \leq \frac{n+2}{2}$ and a $(n,m)$ radial link pattern $\alpha$, we can choose $p$ non-intersecting Pochhammer contours $\mathcal{C}_1,\mathcal{C}_2,\ldots,\mathcal{C}_m$ surrounding pairs of points (which correspond to links in a radial link pattern), see (\ref{radial link pattern}) to integrate $\Phi_{\kappa}$, we obtain
\begin{equation}
\mathcal{J}^{(m, n)}_{\alpha}(\boldsymbol{\theta}):=\oint_{\mathcal{C}_1} \ldots \oint_{\mathcal{C}_m} \Phi_\kappa(\boldsymbol{\theta}, \boldsymbol{\zeta}) d \zeta_m \ldots d \zeta_1 .
\end{equation}

By integration formula (\ref{Trigonometric Coloumb gas integral formula}), $\mathcal{J}^{(m, n)}_{\alpha}(\boldsymbol{\theta})$ satisfies the null vector equations (\ref{null vector equation angular coordinate constant h}) with constant
$$h=\frac{(6-\kappa)(\kappa-2)}{8 \kappa}-\lambda_b(0)-\overline{\lambda_b(0)}=\frac{1-(n-2m)^2}{2\kappa}$$
and the conformal dimension at $0$ is given by
$$\lambda_b(0)=\frac{(n-2m)^2a^2}{8}-\frac{b^2}{2}= \frac{(n-2m)^2}{4\kappa}-\frac{(\kappa-4)^2}{16\kappa}$$
The rotation constant $\omega=0$
$$\sum_{j=1}^n \partial_j \mathcal{J}^{(m, n)}_{\alpha}(\boldsymbol{\theta})=0$$

\item Radial excited solutions:
 \begin{equation} \label{radial excited soultion}
\begin{aligned}
\Phi_\kappa(\boldsymbol{\theta}, \boldsymbol{\zeta})&=\prod_{1 \leq i<j \leq n}\left(\sin\frac{\theta_i-\theta_j}{2}\right)^{a^2} \prod_{1 \leq i<j \leq m}\left(\sin\frac{\zeta_i-\zeta_j}{2}\right)^{4 a^2} \prod_{i=1}^{n} \prod_{j=1}^m\left(\sin\frac{\theta_i-\zeta_j}{2}\right)^{-2 a^2}  \\ 
&\prod_{i=1}^{n}\left(\sin\frac{\theta_i-\omega}{2}\right)^{\frac{(2m-n-2)}{2}}
\end{aligned}
\end{equation}

In this case, for $m \leq \frac{n+2}{2}$ and a $(n,m)$ radial link pattern $\alpha$, we can choose $p$ non-intersecting Pochhammer contours $\mathcal{C}_1,\mathcal{C}_2,\ldots,\mathcal{C}_m$ surrounding pairs of points (which correspond to links in a radial link pattern) to integrate $\zeta_1,\zeta_2,\ldots,\zeta_m$ and a vertical line from $A$ to $A+2\pi i$ to integrate $\omega$ (which corresponds to a circle surrounds the origin), we obtain
\begin{equation}
    \mathcal{K}^{(m,n)}_{\alpha}(\boldsymbol{\theta}):=\oint_{\mathcal{C}_1} \ldots \oint_{\mathcal{C}_m}\int_{A}^{A+2\pi i} \Phi_\kappa(\boldsymbol{\theta}, \boldsymbol{\zeta}) d \zeta_m \ldots d \zeta_1 d\omega.
\end{equation}

By integration formula (\ref{Trigonometric Coloumb gas integral formula}), $\mathcal{J}^{(m, n)}_{\alpha}(\boldsymbol{\theta})$ satisfies the null vector equations (\ref{null vector equation angular coordinate constant h}) with constant

$$h=\frac{(6-\kappa)(\kappa-2)}{8 \kappa}-\lambda_b(0)-\overline{\lambda_b(0)}=\frac{1-(n-2m+\frac{\kappa}{2})^2}{2\kappa}$$
and conformal dimension at $0$ is given by
$$\lambda_{(b)}(0)=\frac{(n-2m)^2}{4\kappa}-\frac{(\kappa-4)^2}{16\kappa}
$$
The rotation constant $\omega=0$,
$$\sum_{j=1}^n \partial_j \mathcal{K}^{(m, n)}_{\alpha}(\boldsymbol{\theta})=0$$

\item Radial ground solutions with spin $\eta$:
\begin{equation}
\begin{aligned}
\Phi_\kappa(\boldsymbol{\theta}, \boldsymbol{\zeta})=& \prod_{1 \leq i<j \leq n}(\sin\frac{\theta_i-\theta_j}{2})^{a^2} \prod_{1 \leq i<j \leq m}(\sin\frac{\zeta_i-\zeta_j}{2})^{4 a^2} \prod_{i=1}^{n} \prod_{j=1}^m\left(\sin\frac{\theta_i-\zeta_j}{2}\right)^{-2 a^2}  \\   
& \prod_{i=1}^{n} 
 e^{\frac{\eta a^2}{2} \theta_{i}}\prod_{j=1}^{m}e^{-\eta a^2\zeta_{j}}
\end{aligned}
\end{equation}

In this case, for $p \leq \frac{n+2}{2}$ and a $(n,p)$ radial link pattern $\alpha$, we can choose $p$ non-intersecting Pochhammer contours $\mathcal{C}_1,\mathcal{C}_2,\ldots,\mathcal{C}_p$ surrounding pairs of points (which correspond to links in a radial link pattern), see (\ref{radial link pattern}) to integrate $\Phi_{\kappa}$, we obtain
\begin{equation} \label{radial with spin}
    \mathcal{J}^{(m, n,\eta)}_{\alpha}(\boldsymbol{\theta}):=\oint_{\mathcal{C}_1} \ldots \oint_{\mathcal{C}_m} \Phi_\kappa(\boldsymbol{\theta}, \boldsymbol{\zeta}) d \zeta_m \ldots d \zeta_1 .
\end{equation}

By integration formula (\ref{Trigonometric Coloumb gas integral formula}), $\mathcal{J}^{(m, n,\eta)}_{\alpha}(\boldsymbol{\theta})$ satisfies the null vector equations (\ref{null vector equation angular coordinate constant h}) with constant
$$h=\frac{(6-\kappa)(\kappa-2)}{8 \kappa}-\lambda_b(0)-\overline{\lambda_b(0)}=-\frac{(n-2m)^2}{2\kappa}+\frac{1+\eta^2}{2\kappa}$$
and conformal dimension at $0$ is given by:
$$
\lambda_{(b)}(0)= \frac{(n-2m+i\eta)^2}{4\kappa}-\frac{(\kappa-4)^2}{16\kappa}
$$
The rotation constant $\omega= \frac{\eta(n-2m)}{\kappa}$,
$$\sum_{j=1}^n \partial_j \mathcal{J}^{(m, n,\eta)}_{\alpha}(\boldsymbol{\theta})=\frac{\eta(n-2m)}{\kappa}\mathcal{J}^{(m, n,\eta)}_{\alpha}(\boldsymbol{\theta})$$

\item Chordal solutions, for $n=2k$ and $m=k-1$:

\begin{equation} 
\begin{aligned}
\Phi_\kappa(\theta_1, \ldots, \theta_{n-1},\theta_c, \zeta_1,\ldots,\zeta_m)&=\prod_{1 \leq i<j \leq n}\left(\sin\frac{\theta_i-\theta_j}{2}\right)^{a^2} \prod_{1 \leq i<j \leq m}\left(\sin\frac{\zeta_i-\zeta_j}{2}\right)^{4 a^2} \\
& \prod_{i=1}^{n} \prod_{j=1}^m\left(\sin\frac{\theta_i-\zeta_j}{2}\right)^{-2 a^2} \prod_{i=1}^{n-1}\left(\sin\frac{\theta_i-\omega}{2}\right)^{a(2b-a)} \\
& \prod_{j=1}^{m}\left(\sin\frac{\theta_i-\omega}{2}\right)^{-2a(2b-a)} 
\end{aligned}
\end{equation}
In this case, for a $(2k,k)$ chordal link pattern, we choose $m=k-1$ non-intersecting Pochhammer contours 
$\mathcal{C}_1,\mathcal{C}_2,\ldots,\mathcal{C}_{k-1}$ surrounding pairs of points except $z_c$ (which correspond to links in a chordal link pattern not connected to $z_c$)
see \cite{FK15c} for detailed explanation.
We obtain:
\begin{equation} \label{Chordal Coulomb gas Solution}
    \mathcal{L}_{\alpha}^{n}(\boldsymbol{\theta}):=\oint_{\mathcal{C}_1} \ldots \oint_{\mathcal{C}_{k-1}} \Phi_\kappa(\boldsymbol{\theta}, \boldsymbol{\zeta}) d \zeta_{k-1} \ldots d \zeta_1 .
\end{equation}

By rewriting the chordal null vector equations in angular coordinate, $\mathcal{J}_{\alpha}(\boldsymbol{\theta})$ satisfies the null vector equations (\ref{null vector equation angular coordinate constant h})  with constant
$$h=\frac{(6-\kappa)(\kappa-2)}{8 \kappa}-\lambda_b(0)-\overline{\lambda_b(0)}=\frac{(6-\kappa)(\kappa-2)}{8 \kappa}$$ 
and conformal dimension at $0$ is given by
$$\lambda_b(0)=0$$
The rotation constant $\omega=0$,
$$\sum_{j=1}^n \partial_j \mathcal{L}^{ n}_{\alpha}(\boldsymbol{\theta})=0$$

\end{itemize}

\section{Null vector equations and quantum Calogero-Sutherland system}
In this section, we study the relationship between multiple radial SLE($\kappa$) systems and the Calogero-Sutherland systems.

\begin{proof}[Proof of theorem (\ref{CS results kappa>0})]

Recall that the null vector differential operator $\mathcal{L}_j$ is given by

\begin{equation}
\begin{aligned}
\mathcal{L}_j= & \frac{\kappa}{2}\left(\frac{\partial}{\partial \theta_j}\right)^2 +\sum_{k \neq j}\left(\cot \left(\frac{\theta_k-\theta_j}{2}\right) \frac{\partial}{\partial \theta_k}+\left(1-\frac{6}{\kappa}\right)\frac{1}{4 \sin ^2\left(\frac{\theta_k-\theta_j}{2}\right)}\right) .
\end{aligned}    
\end{equation}

Then, the null vector equations for $\psi(\boldsymbol{\theta})$ can be written as

\begin{equation}
    \mathcal{L}_j \psi(\boldsymbol{\theta})=h\psi(\boldsymbol{\theta})
\end{equation}
for $j=1,2,\ldots,n$.
\begin{itemize}
\item[(i)] To simplify the formula, we introduce the notation,
$$
f(x)=\cot \left(\frac{x}{2}\right), \quad f_{j k}=f\left(\theta_j-\theta_k\right), \quad F_j=\sum_{k \neq j} f_{j k} .
$$

$$
f^{\prime}(x)=-\frac{1}{2}\frac{1}{\sin^2(\frac{x}{2})}, \quad f^{\prime}_{j k}=f^{\prime}\left(\theta_j-\theta_k\right), \quad F^{\prime}_j=\sum_{k \neq j} f^{\prime}_{j k} 
$$

Using this notation, we have
$$
\mathcal{L}_j=\frac{\kappa}{2} \partial_j^2+\sum_{k \neq j}f_{kj} \partial_k+\sum_{k \neq j}(1-\frac{6}{\kappa}) f_{j k}^{\prime}
$$

with $\partial_j = \frac{\partial}{\partial \theta_j}$ and the Calogero-Sutherland hamiltonian can be written as
\begin{equation}
H_n(\beta)=-\sum_j\left(\frac{1}{2} \partial_j^2+\frac{\beta(\beta-2)}{16} F_j^{\prime}\right) \text {. }
\end{equation}
where $\beta=\frac{8}{\kappa}$. 
\\ \indent
To relate the null-vector equations to the Calogero-Sutherland system, we sum up the null-vector operators. Let

\begin{equation}
\mathcal{L}=\sum_j \mathcal{L}_j=\frac{\kappa}{2} \sum_j \partial_j^2+\sum_j\left(F_j \partial_j+h F_j^{\prime}\right)    
\end{equation}

Then the partition functions $\psi(\boldsymbol{\theta})$ are eigenfunctions of $\mathcal{L}$ with eigenvalue $n h$. 
\begin{equation}
\mathcal{L} \psi(\boldsymbol{\theta})=nh\psi(\boldsymbol{\theta})
\end{equation}

Recall that $$\Phi_{r}(\boldsymbol{\theta})=\prod_{1 \leq j<k \leq n}\left(\sin \frac{\theta_j-\theta_k}{2}\right)^{-2r}$$

From the properties $\partial_j \Phi_r=-r \Phi_r F_j$ and $\sum_j F_j^2=-2 \sum_j F_j^{\prime}-\frac{n\left(n^2-1\right)}{3}$, we can check that
$$
\Phi_{-\frac{1}{\kappa}} \cdot \mathcal{L} \cdot \Phi_{\frac{1}{\kappa}}=\kappa H_n\left(\frac{8}{\kappa}\right)+\frac{n\left(n^2-1\right)}{6 \kappa}
$$
which implies
$$
\tilde{\psi}(\boldsymbol{\theta})=\Phi_{\frac{1}{\kappa}}^{-1}(\boldsymbol{\theta})\psi(\boldsymbol{\theta})
$$
is an eigenfunction of the Calogero-Sutherland hamiltonian $H_n\left(\frac{8}{\kappa}\right)$, with eigenvalue
\begin{equation}
E=\frac{n}{\kappa}\left(-h+\frac{\left(n^2-1\right)}{6 \kappa}\right) .
\end{equation}  
\item[(ii)] see theorem (\ref{commutation of generators}).
\end{itemize}
\end{proof}

\section{Future direction and open problems}

\subsection{Pure partition functions and affine meander matrix*} \label{pure partition functions and meander matrix}

We have already constructed four types of solutions to the null vector equations (\ref{null vector equation for Screening solutions}) and Ward's identities (\ref{Ward identities for screening solutions}). However, not all of these Coulomb gas solutions serve as partition functions for multiple radial SLE($\kappa$) systems. A necessary condition is the positivity.

The pure partition functions are a class of positive partition functions associated with link patterns satisfying a set of asymptotics
(see Definition \ref{pure partition function}). The chordal pure partition functions have been constructed in \cite{KP16}, and positivity verified in \cite{FLPW24} for $\kappa \in (0,8)$.

\begin{defn}[Pure Partition Functions] \label{pure partition function}
The functions $\mathcal{Z}_\alpha: \mathfrak{X}_n \rightarrow \mathbb{R}^+$, indexed by link patterns $\alpha \in \mathrm{LP}(n,m)$, are called \emph{pure partition functions}. They are a collection of positive solutions to the null vector equation~\eqref{null vector equation angular coordinate constant h}, subject to boundary conditions specified by their asymptotic behavior, which is determined by the link pattern $\alpha$
\begin{itemize}
    \item[\textnormal{(ASY)}] \textbf{Asymptotics:} For all $\alpha \in \mathrm{LP}_n$, $j \in \{1, \ldots, n\}$, and $\xi \in (\theta_{j-1}, \theta_{j+2})$, the following limit exists:
    \[
    \lim_{\theta_j, \theta_{j+1} \rightarrow \xi} \frac{\mathcal{Z}_\alpha\left(\theta_1, \ldots, \theta_n\right)}{(\theta_{j+1} - \theta_j)^{\frac{6 - \kappa}{\kappa}}}
    =
    \begin{cases}
        0, & \text{if } \{j, j+1\} \notin \alpha, \\
        \mathcal{Z}_{\hat{\alpha}}\left(\theta_1, \ldots, \theta_{j-1}, \theta_{j+2}, \ldots, \theta_n\right), & \text{if } \{j, j+1\} \in \alpha,
    \end{cases}
    \]
\end{itemize}
where $\hat{\alpha} = \alpha / \{j, j+1\} \in \mathrm{LP}(n-2,m-1)$ denotes the link pattern obtained from $\alpha$ by removing the link $\{j, j+1\}$ and relabeling the remaining indices as $1, 2, \ldots, n-2$.
\end{defn}

We propose several illuminating conjectures about radial pure partition functions in both zero-spin and spin cases, which remain to be clarified.

\begin{defn}[Radial ground solutions]
For each radial link pattern $\alpha$, we choose Pochhammer contours $\mathcal{C}_1, \ldots, \mathcal{C}_n$ along which to integrate out the $\boldsymbol{\xi}$ variables. The integration is well-defined since the conformal dimension of $\Phi_\kappa(\boldsymbol{z}, \boldsymbol{\xi},u)$ is 1 at the $\boldsymbol{\xi}$ points, i.e. since $\lambda_b(-2 a)=1$. This leads to a new function of $\boldsymbol{z}$ defined by
\begin{equation}
\mathcal{J}^{(m, n)}_{\alpha}(\boldsymbol{z},u):=\oint_{\mathcal{C}_1} \ldots \oint_{\mathcal{C}_m}  \Phi_\kappa(\boldsymbol{z}, \boldsymbol{\xi},u) d \xi_m \ldots d \xi_1 .
\end{equation}

In angular coordinates, we obtain 
\begin{equation}
\mathcal{J}^{(m, n)}_{\alpha}(\boldsymbol{\theta}):=\oint_{\mathcal{C}_1} \ldots \oint_{\mathcal{C}_m} \Phi_\kappa(\boldsymbol{\theta}, \boldsymbol{\zeta}) d \zeta_m \ldots d \zeta_1 .
\end{equation}
\end{defn}
\begin{conjecture}[Pure partition function- Coulomb gas integral]
For irrational $\kappa \in(0,8)$, the pure partition functions are related to Coulomb gas integrals by affine meander matrix:
\begin{equation}
\mathcal{J}_\beta^{(m,n)}(\boldsymbol{\theta})=\sum_{\alpha \in \operatorname{LP}(n,m)} \mathcal{M}_{\kappa}(\alpha, \beta) \mathcal{Z}_\alpha(\boldsymbol{\theta}), \quad  \beta \in \mathrm{LP}(n,m).
\end{equation}

Conversely, we have 

\begin{equation}
\mathcal{Z}_\beta(\boldsymbol{\theta})=\sum_{\alpha \in \operatorname{LP}(n,m)} \mathcal{M}_\kappa(\alpha, \beta)^{-1} \mathcal{J}_\alpha^{(m,n)}(\boldsymbol{\theta}), \quad  \beta \in \mathrm{LP}(n,m).
\end{equation}
\end{conjecture}

However, the ground solutions $\mathcal{J}^{(m, n)}_{\alpha}(\boldsymbol{z})$ and $\mathcal{J}^{(m, n,\eta)}_{\alpha}(\boldsymbol{z})$ are not always positive and therefore cannot serve as partition functions for the multiple radial SLE($\kappa$) system.

To address this issue, we propose the following conjectured relations for the radial pure partition functions $\mathcal{Z}_{\alpha}(\boldsymbol{z})$, which are connected to $\mathcal{J}^{(m, n)}_{\alpha}(\boldsymbol{\theta})$ through the radial meander matrix.

\begin{defn}
$O(n)$-model fugacity function, $n: \mathbb{C} \backslash\{0\} \rightarrow \mathbb{C}$ is given by the following formula:
$$
n(\kappa):=-2 \cos (4 \pi / \kappa).
$$
\end{defn}
\begin{defn}[Radial meander matrix]
A meander formed from two link patterns $\alpha, \beta $ is the planar diagram obtained by placing $\alpha$ and the horizontal reflection $\beta$ on top of each other. We define the meander matrix $\left\{\mathcal{M}_{\kappa}(\alpha, \beta): \alpha, \beta \in \mathrm{LP}_n\right\}$ via

\begin{equation}
M_{\kappa}(\alpha,\beta)= \begin{cases}0 & \text { if  }  \text { two rays of } \alpha \text { (or } \beta) \text { are connected, } \\ a^{n_a} b^{n_b} & \text { otherwise, }\end{cases}
\end{equation}
where $a= 2$, $b= n(\kappa)$, $n_a$ and $n_b$ are respectively the numbers of non-contractible and contractible closed loops in the meander formed from $\alpha$ and $\beta$.

Of course, only one of $n_a$ or $n_b$ can be non-zero.
\end{defn}

Next, we propose a parallel conjecture for the multiple radial SLE($\kappa$) system with spin $\eta$.

\begin{defn}[Radial ground solutions with spin]
For each radial link pattern $\alpha$, we choose Pochhammer contours $\mathcal{C}_1, \ldots, \mathcal{C}_n$ along which to integrate out the $\boldsymbol{\xi}$ variables. The integration is well-defined since the conformal dimension of $\Phi_\kappa(\boldsymbol{z}, \boldsymbol{\xi},u)$ is 1 at the $\boldsymbol{\xi}$ points, i.e. since $\lambda_b(-2 a)=1$. This leads to a new function of $\boldsymbol{z}$ defined by
\begin{equation}
\mathcal{J}^{(m, n,\eta)}_{\alpha}(\boldsymbol{z},u):=\oint_{\mathcal{C}_1} \ldots \oint_{\mathcal{C}_m} \Phi_\kappa(\boldsymbol{z}, \boldsymbol{\xi},u) d \xi_m \ldots d \xi_1 .
\end{equation}

In angular coordinates, we obtain 
\begin{equation}
\mathcal{J}^{(m, n,\eta)}_{\alpha}(\boldsymbol{\theta}):=\oint_{\mathcal{C}_1} \ldots \oint_{\mathcal{C}_m} \Phi_\kappa(\boldsymbol{\theta}, \boldsymbol{\zeta}) d \zeta_m \ldots d \zeta_1 .
\end{equation}
\end{defn}

\begin{conjecture}[Pure partition function - Coulomb gas integral] \label{Pure coulomb gas no spin}
For irrational $\kappa \in(0,8)$, $\nu=n(\kappa)$ the pure partition functions are related to Coulomb gas integrals by affine meander matrix:
\begin{equation}
\mathcal{J}_\beta^{(m,n,\eta)}(\boldsymbol{\theta})=\sum_{\alpha \in \operatorname{LP}(n,m)} \mathcal{M}_{\kappa}(\alpha, \beta) \mathcal{Z}_{\alpha}^{\eta}(\boldsymbol{\theta}), \quad  \beta \in \mathrm{LP}(n,m).
\end{equation}

Conversely, we have 

\begin{equation}
\mathcal{Z}_{\beta}^{\eta}(\boldsymbol{\theta})=\sum_{\alpha \in \operatorname{LP}(n,m)} \mathcal{M}_\kappa(\alpha, \beta)^{-1} \mathcal{J}_\alpha^{(m,n,\eta)}(\boldsymbol{\theta}), \quad  \beta \in \mathrm{LP}(n,m).
\end{equation}
\end{conjecture}

The radial meander matrix is not always invertible when $\kappa$ is rational. However, we conjecture that the pure partition functions can nonetheless be analytically continued to cases where $\kappa$ is rational.

\begin{conjecture}
The pure partition function $\mathcal{Z}_{\alpha}(\boldsymbol{\theta})$ and $\mathcal{Z}_{\alpha}^{\eta}(\boldsymbol{\theta})$ can be analytically continued to all $\kappa \in (0,8)$. 
\end{conjecture}

\section*{Acknowledgement}
I express my sincere gratitude to Professor Nikolai Makarov for his invaluable guidance. I am also thankful to Professor Eveliina Peltola and Professor Hao Wu for their enlighting discussions.

\section*{Conflict of interest}
The author declares that there are no conflicts of interest regarding the publication of this work.

\section*{Data availability}
This manuscript has no associated data.

\end{document}